\documentclass[12pt,reqno]{amsart}
\usepackage{amsmath, amscd}
\usepackage{amsfonts}
\usepackage{version}
\usepackage{amssymb}
\usepackage{amssymb,amsmath}
\usepackage{ifthen}
\usepackage{epsfig}
\usepackage{graphicx}
\usepackage{caption}
\usepackage{subcaption}

\newcommand{\Gauss}{{\null_2F_1}}

\newcommand{\R}{\mathbb R}

\newcommand{\al}{\alpha}

\newcommand{\C}{\mathbb C}

\newcommand{\N}{\mathbb N}

\newcommand{\G}{\mathcal{G}}           \newcommand{\RA}{\mathcal{R}}
\newcommand{\So}{\mathcal{S}}\newcommand{\Sf}{\mathfrak S}
       \newcommand{\K}{\mathcal{K}}
          \newcommand{\Pp}{\mathcal{P}}
\newcommand{\D}{\mathbb D}

\newcommand{\Co}{\mathcal C}\newcommand{\Cf}{\mathfrak C}

\newcommand{\Hf}{\mathfrak H}
\newcommand{\Rf}{\mathfrak R}
\newcommand{\Ao}{\mathcal A}
\newcommand{\Af}{\mathfrak A}

\newcommand{\Ff}{\mathfrak F}
\newcommand{\Gf}{\mathfrak G}
\newcommand{\Uf}{\mathfrak U}

\newcommand{\Lf}{\mathfrak L}\newcommand{\Mf}{\mathfrak M}

\newcommand {\Hol}{\mathop{\rm Hol}\nolimits(\D,\C)}

\newcommand  {\Id} {\mathop{\rm Id}\nolimits}
\newcommand {\HolD}{\mathop{\rm Hol}\nolimits(\D)}

\renewcommand{\Re}{\mathop{\rm Re}\nolimits}

\newtheorem{theorem}{Theorem}[section]
\newtheorem{lemma}[theorem]{Lemma}
\newtheorem{proposition}[theorem]{Proposition}
\newtheorem{corollary}[theorem]{Corollary}
\newtheorem{question}{Question}
\newtheorem{conjecture}{Conjecture}
\newtheorem{rem}[theorem]{Remark}

\newtheorem{definition}[theorem]{Definition}

\newtheorem{example}[theorem]{Example}

\newtheorem{remark}[theorem]{Remark}
\numberwithin{equation}{section}

\numberwithin{equation}{section}

\begin{document}
\title[Filtrations of generators, survey]{Survey on filtrations (parametric embeddings) of infinitesimal generators}

\author[M. Elin]{Mark Elin}
\address{M. Elin: Department of Mathematics, Braude College of Engeneering, Karmiel 21982, Israel.}
\email{mark\_elin@braude.ac.il}

\author[F. Jacobzon]{Fiana Jacobzon}
\address{F. Jacobzon: Department of Mathematics, Braude College of Engeneering, Karmiel 21982, Israel.}
\email{fiana@braude.ac.il}

\date{\today}

\begin{abstract}
This work is devoted to the so-called filtration theory of semigroup generators in the unit disk. It should be noted that numerous filtrations studied to nowdays  have been introduced for different purposes and considered from different points of view. So, our aim is to summarize the known facts, to present common and distinct properties of filtrations as well as to study some new filtrations with an emphasis on their connection with geometric function theory and the dynamic features of semigroups generated by elements of different filtration families.

Among the dynamic properties, we mention the uniform convergence on the unit disk and the sectorial analytical extension of semigroups with respect to their parameter.

We also solve the  Fekete--Szeg\"o problem over various filtration classes, as well as over non-linear resolvents.
\end{abstract}

\maketitle
\tableofcontents

\bigskip

\section{Introduction}\label{sect-int}

Semigroups of non-linear maps are a natural generalization of semigroups of linear operators. In the one-dimensional case, Berkson and Porta in \cite{BE-PH} made a breakthrough in the theory of semigroups of holomorphic self-mappings of the open unit disk. They proved that every semigroup has a corresponding generator, and characterized the structure of generators. The study of generation theory in Banach spaces began with the work \cite{R-S-96, RS-SD-97a} by Reich and Shoikhet. During the last decades, various characterizations of semigroup generators have been found, some of them are recalled in Theorem~\ref{th-gen-prop} below. For more details see monographs \cite{B-C-DM-book, E-R-S-19, E-Sbook, R-S1, SD} and references therein. The class of all semigroup generators appears also in the study of non-autonomous problems and geometric function theory (see, for example, \cite{Duren, E-R-S-04, GI-KG}).
It is especially interesting to trace the influence of properties of a generator on the dynamic behavior of the generated semigroup. The issues under study include the following:
\begin{itemize}
\item [(1)] Finding criteria for membership in the class of all generators.
\item [(2)] Characterization of the asymptotic behavior of semigroups.
\item [(3)] Feasibility of analytic extension of a semigroup in its parameter into a domain in the complex plane $\C$.
\end{itemize}
It is also important to establish
\begin{itemize}
\item [(4)] Applications in geometric function theory.
\end{itemize}

The first problem has been investigated by many authors (also in several variables), see, for example, \cite{A-E-R-S, A-R-S, BE-PH, B-C-DM-book, E-Sbook, E-R-S-04, E-R-S-19, R-S1, SD}. Eventually the answers are given by suitable inequalities in terms of intrinsic objects (hyperbolic metric, hyperbolic distance, Green function) or extrinsic ones (Banach space norm, Hilbert/Euclidean distance).

\vspace{2mm}

Over the years, the study of the asymptotic behavior of semigroups was mainly focused on the local/global rate of convergence and the growth estimates of a semigroup with respect to its parameter. Different estimates of the rate of convergence of semigroups were obtained. Concerning the one-dimensional case the reader can be referred to the books \cite{B-C-DM-book, E-Sbook, SD}, the survey \cite{JLR} and references therein.

Another direction in the study of non-linear semigroups is focused on the possibility of analytic extension with respect to the semigroup parameter to a complex domain.
The problem is to find conditions which enable analytic extension of semigroups along with estimates for the size of the domain to which this extension is possible. The progress in this direction presented, for example, in \cite{ACP, E-S-T, E-J}.

The profound mutual influence of semigroup theory and geometric function theory should be noticed. For example, it will be explained below (see also \cite{BCDES, E-S-Tu2, SD16}) that every function that satisfies the famous Noshiro--Warschawski condition and every univalent convex function is a semigroup generator. Since not all generators are univalent, these conditions are far from necessary for a function to be a generator.
Moreover, we will see below that convex functions generate so-called exponentially squeezing semigroups (see Definition~\ref{def-squee}). Such semigroups converge uniformly on the whole open unit disk, while generically uniform convergence is ensured only on compact subsets.

In addition, each element of a semigroup is a univalent self-mapping of the open unit disk. At the same time, an arbitrary given univalent self-mapping can be embedded into a semigroup if and only if its K{\oe}nigs function (which is an eigenfunction of the composition operator with given mapping) is starlike or spirallike, see, for example, \cite{B-C-DM-book, E-Sbook}. Often dynamic properties of a semigroup are equivalent to such geometric properties of its generator as the size of the image, its placement and so on. However, it might be difficult to reveal that properties for a given generator.

 Furthermore, despite of the presence of various criteria for generators, in some specific situations it is difficult to check whether a given function belongs to the set of all generators. At the same time, there are some sufficient conditions providing that a function belongs to {\it a proper subclass} of all generators the verification of which is easier.
 Moreover, it may happen that such subclasses form a {\it one-parameter net, increasing as their parameter grows}, and properties of the generated semigroup obey the subclass to which the generator belongs.
Roughly speaking, such nets are called {\it filtrations of generators} (see Definition~\ref{def-filt} below). This work is devoted to {\it filtration theory of generators}.

Looking ahead, some function classes well-known in geometric function theory are, in fact, filtrations of generators (in addition to convex functions mentioned earlier, see Subsection~\ref{subsec-prestar}). Therefore the study of filtration families relies on their properties associated with geometric function theory.

Numerous filtrations studied to nowdays  have been introduced for different purposes and considered from different points of view, see \cite{BCDES, E-J-S, EJT, ESS, E-S-Tu2, SD16} and other works.
Although the main idea is dimension-independent, the work cited focus on the one-dimensional case. {\it The aim of this work is to summarize the known facts, to present common and distinct properties of filtrations as well as to study some new filtrations with an emphasis on their connection with geometric function theory and certain dynamic features of semigroups generated by elements of different filtration families.}

The diversity of related results in this area has led us to abandon the consideration of two related topics. First, the study of filtrations on smaller disks included in the unit disk is omitted. Second, we do not consider parametric embeddings of K{\oe}nigs functions associated with generators. The interested reader may refer to \cite{BCDES, E-S-Tu1, E-S-Tu2, SD16}.

\vspace{3mm}

The structure of  the survey is the following. In the next section we provide  some preliminary material to be used in the subsequent study. It includes elements of both geometric function theory and semigroup theory.

\vspace{2mm}

In Section~\ref{sect-basic} we first introduce the main concepts of this work like the notion of filtration and some related notions. We proceed with exponentially squeezing and sectorial filtrations. Although these filtrations are very simple, they allow us to describe some important dynamic properties of  semigroups and so are crucial.

\vspace{2mm}

Section~\ref{sect-_spec_linear} is devoted to linear filtrations of first order.  The sets of such filtrations consist of generators $f$ for which a certain harmonic function depending linearly on $f$ and $f'$ is positive. We begin with a generic approach to construct linear filtrations using an arbitrary auxiliary function~$k$. It turns out that the sets belonging to such filtrations are convex. Special choices of $k$ give concrete linear filtrations that are of interest because of their dynamic and geometric behavior.

Subsection~\ref{subsec-analytic} is devoted to the so-called analytic filtration. It was partially studied in \cite{BCDES, ESS, SD16}. We supplement the previous results by expanding the range of the parameter from the interval $[0,1]$ to $(-\infty, 1]$ and give a new characterization of the filtration sets.

Further, we present in Subsection~\ref{subsec-hyperb} hyperbolic filtrations studied in \cite{BCDES} and \cite{SD16}.
We discuss two very similar filtrations that are linear combinations of two different conditions characterizing semigroup generators. The reason for the name `hyperbolic' is the fact that both conditions originate in the hyperbolic geometry of the unit disk. It is interesting to mention that the `starting' set of this filtration consists of univalent functions.

In addition, in Subsection~\ref{subsec-gener-analytic}, we study a two-parameter family of sets of analytic functions first considered  in \cite{K-R}. In a sense, the squeezing and analytic filtrations are its `border' cases. Considering the specific range of these two parameters in the plane, it is natural to search for other filtrations in this range. 

The subsection consists of two parts. In the first one we describe properties of this two-parameter family. On this way we get a generalization of properties of the analytic filtration. In the second part we discuss additional filtrations in the mentioned range of the parameters. We prove that a curve in that range presents a filtration subject a simple sufficient condition.

\vspace{2mm}

In Section~\ref{sect-non-l} we consider filtrations determined by non-linear expressions depending on $f, f'$ and $f''$.

The study of such filtrations is usually more complicated.
In order to ease the reading, in Subsection~\ref{subsec-prestar} we present filtrations formed by well-known subclasses of univalent functions. They include starlike functions of some order, Janowski functions, classes introduced by Aksentiev \cite{Ak58}, and prestarlike functions by Ruscheweyh. Of course, in general these classes consist not only of generators and not always form a filtration. Therefore we need to specify the conditions providing that a given family of some sets is a filtration of generators.

Although many of the properties of these classes are long known, our goal is to emphasize their dynamic features.

Subsection~\ref{subsect-pseu-anal} presents a new filtration, which we call the pseudo-analytic filtration since its sets are determined by the absolute value of the same linear functional as ones of the analytic filtration.
Its peculiarity is that it admits a net of totally extremal functions whose elements are not boundary points of the corresponding classes, and has boundary points that are not totally extremal functions.

Another non-linear filtration (Subsection~\ref{sSect-Qstar}) is determined by a linear combination of inequalities characterizing the class of all generators and the class of starlike functions. We note that the `smallest' class of this filtration contains all starlike functions of order $\frac12$ while the `largest' one consists of functions that satisfy $\Re\frac{f(z)}{z}>\frac12$ for $z\neq0$ in the unit disk.
An additional filtration is constructed using the class introduced by Mocanu. After him we call it Mocanu's filtration (Subsection~\ref{subsec-mocanu1}). Elements of each filtration set are univalent functions.

\vspace{2mm}

In the last section we involve filtration classes considered in the previous sections to problems of geometric function theory. One of such classical problems is estimation of certain functionals over different classes of analytic functions.
For instance, estimates of the Taylor's coefficients including the Bieberbach conjecture (see for example, \cite{Duren}) have history longer than a century. Among other significant functionals we emphasize the generalized Zalcman functional and its particular case, the Fekete--Szeg\"o functional. The Fekete--Szeg\"o problem is to find the sharp estimate on this functional over a given class of analytic functions. We investigate this problem in Subsections~\ref{subsect-an-pre-pse} and~\ref{subsec-FS} over filtration classes presented above.

In Subsection~\ref{subsec-resolv} we solve the Fekete--Szeg\"o problem over so-called non-linear resolvents.
It is well-known (see details in \cite{R-S1}) that an analytic function is a semigroup generator if and only if it produces the so-called {\it resolvent family}. Sharp estimates on the Fekete--Szeg\"o functional over families of non-linear resolvents for (almost all) filtrations considered in the survey are established. For the class of all generators this problem was studied earlier in \cite{EJ-coeff20, EJ-est} (and also in multi-dimensional settings in \cite{EJ-spiral, GHK2020, HKK2021}).

\vspace{2mm}

In addition, we would like to draw the reader's attention to the fact that Sections~\ref{sect-_spec_linear}, \ref{sect-non-l} and \ref{sect-appl} have been supplemented with special subsections containing open questions, conjectures and comments for further investigation. Several general open questions are also posed in Subsection~\ref{subsect-notion-filt}.

\bigskip

\section{Preliminaries}\label{sect-preli-geom}

\setcounter{equation}{0}

\subsection{Elements from geometric function theory}
 Let $\D$ be the open  unit disk in the complex plane $\C$. Denote by $\Hol$ the set of holomorphic functions on $\D$, and by $\HolD := \Hol$, the set of all holomorphic self-mappings of $\D$.

 By $\Ao$ we denote the subset of $\Hol$ consisting of  functions normalized by $f(0)=f'(0)-1=0$ and by $\Pp$ the Carath\'eodory class, that is,
\begin{equation*}
  \Pp:=\left\{ q\in\Hol:\ \Re q(z)>0, \ z\in\D, \quad\mbox{and}\quad q(0)=1 \right\}.
\end{equation*}

Among properties of the class $\Pp$, we mention that its elements can be characterized by the Riesz--Herglotz integral
representation:
\begin{theorem}\label{thm-RH}
  Let $q\in\Hol.$ Then $q\in\Pp$ if and only if there is a probability measure $\mu$ on the unit circle such that
  \begin{equation*}
    q(z) =\oint\limits_{\partial\D} \frac{1+z\overline{\zeta}}{1- z\overline{\zeta}}\,d\mu(\zeta).
  \end{equation*}
  Consequently, $\displaystyle \frac{1-|z|}{1+|z|}\le \Re q(z)\le \frac{1+|z|}{1-|z|}\,.$
\end{theorem}
The last inequality is called Harnack's inequality.

The Carath\'eodory class plays an essential role in geometric function theory and in adjacent fields. For instance, many famous classes of univalent functions can be characterized by condition that the real part of certain expression  is greater than zero.

To be more concrete, recall that a function $f \in\Ao$ is called {\it convex or starlike} if it maps $\D$ onto a convex or starlike region, respectively. Corresponding classes are denoted by $\Co$ and $\So^*$. It is well known that these classes admit the following analytic description:
\begin{eqnarray*}
  &\Co=& \left\{ f\in\Ao:\ 1+\frac{zf''(z)}{f'(z)} \in\Pp \right\}, \\
  &\So^*=& \left\{ f\in\Ao:\ \frac{zf'(z)}{f(z)} \in\Pp \right\}.
\end{eqnarray*}

\vspace{2mm}

 To proceed we need an additional notion. Let $f, g\in\Ao$ and
\begin{equation*}\label{def-U}
\Omega=\{ \omega \in \HolD:\  \omega(0)=0 \}.
\end{equation*}
One says that $f$ is subordinate to $g$ and write $f\prec g$, if there exists a function $\omega\in\Omega$ such that $f(z) = g(\omega(z))$, $z\in\D.$ Specifically, if $g$ is univalent then the conditions $f\prec g$ and $f(\D)\subseteq g(\D)$ are equivalent.

Ma and Minda in \cite{Ma-Mi} proposed unified approach to study subclasses of $\So^*$ of the form
\( \displaystyle  \So^*(\varphi)= \left\{ f\in\Ao:\ \frac{zf'(z)}{f(z)} \prec \varphi(z) \right\} \)
for some $\varphi\in\Pp$ such that $\varphi(\D)$ is symmetric with respect to the real axis and starlike with respect to $\varphi(0)=1$ and $\varphi'(0)>0$. For the choice $\varphi(z)=\frac{1+az}{1+bz},$ $-1\leq b< a \leq 1$, the class $\So^*(\varphi)$ reduces to the class
\begin{equation}\label{Jano}
 \So^*(a,b)= \left\{ f\in\Ao:\ \frac{zf'(z)}{f(z)} \prec \frac{1+az}{1+bz} \right\},\quad -1\leq b< a \leq 1,
\end{equation}
of Janowski starlike functions \cite{Jan} and have been studied by many mathematicians (see, for example, \cite{J-S-S} and references therein).  For $0\leq\alpha<1$ the functions of the class
\[
\So^*(\al):=\So^*(1-2\al,-1)= \left\{ f\in\Ao:\ \Re\frac{zf'(z)}{f(z)} \ge\al \right\}
\]
are called {\it starlike functions of order $\al$}.

During the last century these classes and relations between them were studied by many mathematicians. We will use the following classical results:
\begin{theorem}[see Marx \cite{Mar} and Strohh\"acker \cite{Stro}]\label{thm-M-S}
\begin{equation*}
  f\in\Co \quad \Longrightarrow\quad  f\in\So^*\!\!\left(\frac12\right)  \quad \Longrightarrow \quad \left[2\frac{f(z)}{z}-1\right]\in\Pp .
\end{equation*}
\end{theorem}

A helpful sufficient conditions for univalence is
\begin{theorem}[Noshiro--Warschawski theorem]\label{thm-NW}
  If $f\in\Hol$ and $\Re f'(z)>0$, then $f$ is univalent in $\D$.
\end{theorem}

Thanks to this theorem, it begs to consider the class
\[
\RA:=\left\{ f\in\Ao:\ f'\in\Pp  \right\}.
\]
One can easily see that if $f\in \RA$ then
\begin{equation}\label{NW1}
\Re\frac{f(z)}z=\int_0^1\Re f'(tz)dt>0,\quad \mbox{so}\quad \frac{f(z)}z\in\Pp.
\end{equation}
Usually $\RA$ is called the Noshiro--Warschawski class (see, for example, \cite{E-V}). More detailed information on all the above mentioned classes can be found in \cite{Duren, GAW, GI-KG}.
\vspace{2mm}

The study of  different classes of functions involves the searching for extremals. We will use the following notion:
\begin{definition}[see \cite{BCDES}]\label{def-extrem}
Let $\mathcal F\subset\Ao$. We say that a function $f_*\in\mathcal{F}$
is {\sl totally extremal} for $\mathcal F$ if for every $\lambda\in\C,\ r\in [0,1]$ and $f\in\mathcal F$,
\[
\min_{|z|=r} \Re \left(\lambda\frac {f(z)}{z}\right)\ge
\min_{|z|=r}\Re\left(\lambda \frac {f_*(z)}{z}\right)\,.
\]
\end{definition}
\begin{example}\label{exa-star}
  Recall that by \cite[Theorem 3]{BHMW}, the set of all extreme points of $\mathcal{S}^*(\alpha)$ consists of the functions ${z}/{(1-xz)^{2(1-\alpha)}}$, $|x|=1$. Hence for every class $\mathcal{S}^*(\alpha)$, the function $f(z)=\frac{z}{(1-z)^{2(1-\alpha)}}$ is totally extremal.
\end{example}
\begin{rem}\label{rem-note}
Definition~\ref{def-extrem} can be explained as follows. Given $f_*\in\mathcal{F}$, write $f_*(z)=zp_*(z)$. Then $f_*$ is totally extremal for $\mathcal F$  if and only if for every $r\in[0,1]$ and $f\in\mathcal F$ with $f(z)=zp(z)$, the image $p(\D_r)$ of the disk of radius $r$ lies in the convex hull of $p_*(\D_r)$.
\end{rem}

\vspace{2mm}

In what follows we will use the next two lemmas.
\begin{lemma}\label{lemm-Psi}
Fix  $\gamma\in\R$ and $g \in \Hol$. Then $\Re g(z) > \gamma \ (z\in\D)$ if and only if the function $\omega$ defined by
\begin{equation}\label{Psi}
\omega(z)=\frac{g(z)-g(0)}{g(z)-2\gamma+g(0)}
\end{equation}
belongs to $\Omega$.
\end{lemma}
This result follows from simple geometric considerations.

The next assertion is known as Jack's lemma or the Clunie--Jack lemma (see, for example, \cite{M-M}).
\begin{lemma}\label{lem-Jack}
Let $\omega\in \Omega$. If $|\omega(z)|$ admits its maximum value on the circle $|z| = r$ at a point $z_0$, then $z_0\omega'(z_0) = k \omega(z_0),$ where $k(=k(r))$ is real and $k\ge 1$.
\end{lemma}

\vspace{2mm}

Further, for two functions $f,g\in\Hol$ with Taylor expansions $f(z)=\sum\limits_{n=0}^\infty a_nz^n$ and $g(z)=\sum\limits_{n=0}^\infty b_nz^n$, respectively, their Hadamard product is defined by

\begin{equation}\label{hadamar}
 f*g(z):=\sum\limits_{n=0}^\infty a_nb_nz^n,\qquad z\in\D.
\end{equation}
The formulas
\begin{equation}\label{hadamar-pr}
f(z)*\frac{1}{1-z\overline{\zeta}}=f(z\overline{\zeta}) \qquad\mbox{and}\qquad   {z\left(f(z)*g(z)\right)'=\left(zf'(z)\right)*g(z)}
\end{equation}
are important properties of the Hadamard product.

We will also use a kind of the Bernardi integral operator $I_\gamma[q]:\Hol\to\Hol$ with $\gamma\in\C,\ \Re\gamma>-1$ (see \cite{M-M}), defined as follows. Let a function $q\in\Hol$ have the Taylor expansion $q(z)=1+\sum\limits_{n=1}^\infty q_nz^n$, then
\begin{eqnarray}
  I_\gamma[q](z) &=& \gamma z^{-\gamma}\int_0^z \tau^{\gamma-1}q(\tau)d\tau=\int_0^1q(s^{1/\gamma}z)ds \\
   &=& 1+\sum_{n=1}^\infty \frac{\gamma q_n}{n+\gamma}z^n = q(z)*\Gauss(1,\gamma;\gamma+1;z) ,\nonumber
\end{eqnarray}\label{I-ga}
where $\Gauss$ is the Gau{\ss} hypergeometric function,
\begin{equation}\label{gauss}
\Gauss(a,b;c;z)=\sum_{n=0}^\infty\frac{(a)_n(b)_n}{(c)_nn!}z^n,
\quad z\in\D,
\end{equation}
and $(\alpha)_n=\al\cdot(\al+1)\cdot\ldots\cdot(\al+n-1)$ is the Pochhammer symbol, $a,b,c\in\C$ and $c\ne0,-1,-2,\dots.$

Clearly, $I_\gamma:\Pp\to\Pp.$ In addition, Hallenbeck and Ruscheweyh's Theorem 1 in \cite{HR75} can be rephrased as follows:
\begin{theorem}\label{thm-H-R-75}
  Let $Q(z)$ be a convex univalent function in $\D$ with $Q(0)=1$  and $q\prec Q$. Then $I_\gamma[q]\prec I_\gamma[Q]$ whenever $\Re\gamma>0.$
\end{theorem}

\bigskip

\subsection{Elements of semigroup theory}\label{subsect-semig}



We now turn to main notions and facts from semigroup theory in the unit disk. For details the reader can consults with \cite{B-C-DM-book, E-Sbook, R-S1}.
\begin{definition}\label{def-sg}
A family $\left\{ \phi _{t}\right\} _{t\geq 0}\subset\HolD$ is
called a {\sl one-parameter continuous semigroup}  (or {\sl semigroup}, for short) if
\begin{itemize}
  \item [(a)] $\phi _{t+s}=\phi _{t}\circ \phi _{s},\ t,s\geq 0$; and
  \item [(b)] $\lim\limits_{t\rightarrow 0^{+}}\phi _{t}=\Id$, where $\Id$ is the identity map on $\D$ and the limit exists in the topology of uniform convergence on compact sets in $\D$.
\end{itemize}
\end{definition}

Moreover, according to the seminal result by Berkson and Porta \cite{BE-PH}, the family $\left\{ \phi _{t}\right\} _{t\geq 0}$ is differentiable with respect to its parameter $t\geq 0$, the limit
\[
f:=\lim_{t\rightarrow 0^{+}}\frac{1}{t}\left( \Id-\phi _{t}\right)
\]
exists and defines a holomorphic function on $\D$. Furthermore, $\phi _{t}(z)$ is the solution of the Cauchy problem:
\begin{equation*}
\frac{\partial \phi _{t}(z)}{\partial t}+f(\phi _{t}(z))=0\quad \mbox{and} \quad \phi _{0}(z)= z\in \D.
\end{equation*}
The function $f$ is called the  {\sl infinitesimal generator} of the semigroup $\left\{ \phi _{t}\right\} _{t\geq 0}\subset\HolD.$ The class of all (holomorphic) generators on $\D$ is denoted by $\G$.

 The following assertion characterizes the class $\G$.
\begin{theorem}\label{th-gen-prop}
Let $f\in \Hol$. Then $f \in\G$ if and only if one of the following conditions holds:
\begin{itemize}
  \item [(a)] there are a point $\tau \in \overline{\D}$ and a function $q \in \Hol$ with $\Re q( z) \geq 0$, $z\in \D,$ such that
\begin{equation}\label{b-p}
f( z ) =\left( z-\tau \right) \left( 1-z\overline{\tau }\right) q( z),\quad z\in \D.
\end{equation}
Moreover, this representation is unique.

  \item [(b)] For any $z\in\D$ the function $f$ satisfies the inequality
  \begin{equation}\label{abate}
\Re \left[ 2f(z) \overline{z} + \left(1- | z| ^2\right) f'(z)\right] \geq 0.
\end{equation}
\end{itemize}
\end{theorem}
We notice that formula \eqref{b-p}  was obtained by Berkson and Porta in \cite{BE-PH} and is called the {\it Berkson--Porta representation}. Inequality \eqref{abate} was established by Abate in \cite{Ababook89}.

Observe that if the semigroup $\left\{ \phi_{t}\right\}_{t\geq
0}\subset \HolD$ generated by $f$ contains neither elliptic automorphism of $\D $, nor the identity mapping, then the point $\tau \in \overline{\D }$ in \eqref{b-p} is an attractive
point of the semigroup $\left\{ \phi_{t}\right\} _{t\geq
0}\in \D $, that is,
\begin{equation}\label{tau-uniq}
\lim\limits_{t \to \infty} \phi_{t}( z) =\tau
\text{ for all }\,  z\in \D.
\end{equation}

The last assertion is a continuous analog of the classical Denjoy--Wolff theorem (see \cite{AM-92}, \cite{RS-SD-97b} and \cite{SD}). The point $\tau $ in \eqref{b-p} is called the {\it Denjoy--Wolff point} of  $\left\{ \phi_{t}\right\} _{t\geq 0}$.

If $\tau \in \D$, it follows from
the uniqueness of solutions to the Cauchy problem that this point $\tau$ must be the (unique) null point of $f$ in $\D$. Otherwise, if $f$ has no null point in $\D$, there is a unique boundary point $\tau\in\partial\D$ that is attractive for the semigroup in the sense~\eqref{tau-uniq}.

 Let focus on the case $\tau\in\D$.  Up to the M\"obius involution 
 of the unit disk defined by $M_a(z)=\frac{a-z}{1-z\overline{a}},$ one can always assume that $\tau=0$, that is, $f(0) =0,$ or, what is the same, $\phi_t(0) =0$ for all $t\geq 0$. In this case $f\in \Ao$ is an infinitesimal generator if and only if  $\Re \frac{f(z)}{z}\geq 0, \   z\in \D \setminus \{0\}$, see Theorem~\ref{th-gen-prop} (a). So, it is natural to denote
\begin{equation*}
\G_0:=\mathcal{A}\cap \G= \left\{ f\in\Hol: \ \frac{f(z)}z\in\Pp\right\}.
\end{equation*}
Problems (1)--(4) discussed in Section~\ref{sect-int} are of special relevance for the class $\G_0$.

As an example for the first problem, let  consider the function $f\in\Ao$ defined by $f(z) =-z -2\log (1-z)$. Since this function is transcendental, it is not trivial to check whether either \eqref{b-p}, or \eqref{abate}, or other equivalent condition holds. However, by direct calculation one gets $\Re f'(z) = \Re\frac{1+z}{1-z}>0,z\in\D,$ so $f \in \RA$. Notice that \eqref{NW1} implies $\RA\subset\G_0$, thus $f$ is a generator.

Recall that every $f\in \RA$ is univalent by Theorem~\ref{thm-NW}. Since not all generators are univalent, the Noshiro--Warschawski condition is far from being a necessary condition for membership in $\G_0$.

As for the second problem, the rate of convergence of semigroups to zero can be estimated in terms of the Euclidean distance as follows.
\begin{theorem}\label{thm-G-P}
Let $\{\phi_{t}\}_{t\geq 0}$ be the semigroup generated by $f \in \G_0$. Then for
all $z \in \D$ and $t\geq 0$, the following estimates hold:
\begin{itemize}
\item[$(i)$] $|z|\cdot \exp\left(-t\,\Re f'(0) \dfrac{1+|z|}{1-|z|}\right)\leq\left|\phi_{t}(z)\right|\leq|z|\cdot \exp\left(-t\, \Re
f'(0) \dfrac{1-|z|}{1+|z|}\right)$;

\item[$(ii)$] $\exp(-t\, \Re f'(0) )\dfrac{|z|}{(1+|z|)^{2}}\leq \dfrac{\left|\phi_{t}(z)\right|}{\left(1-\left|\phi_{t}(z)\right|\right)^{2}}
\leq\exp(- t\,\Re f'(0) )\dfrac{|z|}{(1-|z|)^{2}}$.
\end{itemize}
\end{theorem}
Estimate (i) is due to Gurganus~\cite{GKR-75} and estimate (ii) was established by Poreda~\cite{Por}.

It is worth to mention that although the estimates of rate of convergence given in Theorem~\ref{thm-G-P} are very useful, they are not uniform on the whole disk $\D$. Indeed, inequality (i) says that the semigroup converges exponentially pointwise, and consequently, uniformly on every compact subset of $\D$, that is, $|\phi_{t}(z)| \le |z| \exp(-t a)$, where $a=a(z)$ depends on $z$. The point is that in general $a(z)\to0$ as $z$ approaches the boundary $\partial\D$. At the same time, there are specific generators in $\G_0$ for which the rate of convergence is exponential, uniformly on the whole disk. This leads us to the following concept.

\begin{definition}[see \cite{BCDES}]\label{def-squee}
Let $\left\{ \phi _{t}\right\} _{t\geq 0}\subset\HolD$ be a continuous semigroup. If there exists a constant $a> 0$ such that
$$|\phi _{t}(z)|\leq e^{-at}|z|\quad \text{ for all } \quad z\in \mathbb{D},$$ then $\left\{ \phi _{t}\right\} _{t\geq
0}$  is said to be \textsl{exponentially squeezing semigroup with squeezing ratio $a$.}
\end{definition}

A criterion for a semigroup to be exponentially squeezing is given in the next statement
\begin{proposition}[Proposition~2.7 in \cite{BCDES}]\label{charactexp copy(1)}
Let $f\in\G_0$ and  $a\in(0,1]$. Then the semigroup generating by $f$ is exponentially squeezing with squeezing ratio $a$ if and only if $\Re\frac{f(z)}{z} \ge a$ for all $z\in\D\setminus \{0\}$.
\end{proposition}
\noindent(For the infinite-dimensional case the reader can be referred to \cite{E-R-S-02, E-R-S-04}, see also \cite{E-R-S-19, R-S1}.)

Following the analogy with semigroups of linear operators, the possibility of analytic extension of semigroups of holomorphic mappings in a complex parameter has been studied. This relates to problem (3) above and consists of two components: to find conditions  that allow analytic extension of a semigroup with respect to its parameter into a sector in the complex plane, and to estimate the maximal size (angle of opening) of this sector.

 To be more specific, fix $\theta\in\left(0,\frac\pi2\right]$ and denote
\begin{equation*}
\Lambda(\theta) = \left\{\zeta\in\C: |\arg \zeta| <\theta  \right\}.
\end{equation*}

\begin{definition}[see \cite{E-R-S-19, ACP, E-S-T}]\label{def-an-sg}
A family $\{F_{\zeta}\}_{\zeta\in\Lambda}$ of holomorphic self-mappings of $\D$ indexed by a parameter $\zeta$ in the sector $\Lambda:=\Lambda(\theta)\cup\{0\}$ of the complex plane is said to be a one-parameter analytic semigroup if
\begin{itemize}
\item[(i)]  $\zeta\mapsto F_\zeta$ is analytic in $\Lambda$;

\item[(ii)] $\lim\limits_{\Lambda\ni\zeta\to0} F_{\zeta}=F_0 = I$;

\item[(iii)] $F_{\zeta_1 +\zeta_2} = F_{\zeta_1}\circ F_{\zeta_2}$
whenever $\zeta_1,\zeta_2\in\Lambda$.
\end{itemize}
\end{definition}

The approach to analytic extension of semigroups of holomorphic self-mappings is based on
the following necessary and sufficient condition established in \cite[Theorem~2.12]{E-S-T}.

\begin{proposition}\label{rotation}
Let $f\in\G_0$ and  $\alpha\in\left[0,1\right)$. Then the semigroup generating by $f$  extends analytically to the sector
$\Lambda\left(\frac{\pi(1-\al)}2\right)$ in $\C$ if and only if $\left|\arg \frac{f(z)}z \right|\leq \frac{\pi\alpha}2$ for all $z\in \D\setminus \{0\}$.
\end{proposition}

\vspace{2mm}
Assume that we study the {\it family} of dynamical systems determined by a subset $\G_1\subsetneq\G_0$. Due to Propositions~\ref{charactexp copy(1)} and~\ref{rotation}, one can conclude that the study of the asymptotic behavior of dynamical systems may relay on the distinction whether there exist $k,\beta\ge0$ so that $\Re\frac{f(z)}z\ge k$ and/or $\left|\arg[f(z)/z]\right|\le \pi\beta/2$ for all $f \in \G_1.$
Of course if the answer is positive, it is of a special interest to find such (sharp) $k$ and/or $\beta$.
(For example, it follows from Marx and Strohh\"acker's Theorem~\ref{thm-M-S}  that $k=\frac{1}{2}$ fits for the class $\Co$.)

To systematize subsequent reasoning in this direction, let introduce the two quantities:
\begin{definition}[\cite{ESS}]\label{def-K-B}
Let $\G_1$ be a subclass of $\G_0$, denote
\begin{eqnarray*}
K(\G_1)&=&\inf_{f\in\G_1}\inf_{z\in\D}\Re\frac{f(z)}{z} \qquad\mbox{and}\\
B(\G_1)&=&\frac2\pi \sup_{f\in\G_1}\sup_{z\in\D} \left| \arg\frac{f(z)}{z}\right|.
\end{eqnarray*}
\end{definition}

\begin{rem}\label{rem_27_02}
  Observe that
  \begin{itemize}
    \item if a function $f_*\in\G_0$ is totally extremal for a class $\G_1$, then $K(\G_1)=\inf\limits_{z\in\D}\Re\frac{f_*(z)}z$ and $B(\G_1)=\frac2\pi \sup\limits_{z\in\D} \left|\arg\frac{f_*(z)}z \right| $;
    \item if $\G_1$ and $\G_2$ are two subsets with $\G_1\subseteq\G_2\subseteq\G_0$, then $K(\G_1)\ge K(\G_2)$ and $B(\G_1)\le B(\G_2)$.
  \end{itemize}
 \end{rem}
Propositions~\ref{charactexp copy(1)} and~\ref{rotation}, in fact, express the connection between {\it geometric}  properties  of generators $f\in\G_1$ and {\it dynamic} properties of the semigroups they generate. This connection explains inter alia the importance of searching for  totally extremal functions (if exist), see Definitions~\ref{def-extrem} and~\ref{def-extrem1}.  Namely, if $K(\G_1)>0$, then semigroups generated by elements of $\G_1$ are exponentially squeezing and $K(\G_1)$ is the sharp squeezing ratio for them. If  $B(\G_1)<1$, then semigroups generated by any $f\in\G_1$ admit analytic extension to the sector $\Lambda\left(\frac\pi2 (1-B(\G_1) )\right)$  and this sector is the maximal one.

\vspace{2mm}

Further, we recall some notions from geometric function theory adapted to the restrictions in this paper. Let $f \in \G_0$ and $\{\phi_{t}\}_{t\geq 0}$ be the semigroup generated by $f$. A point $\zeta\in\partial\D$ is called
\begin{itemize}
  \item  {\it a repelling fixed point } of the  semigroup $\{\phi_{t}\}_{t\geq 0}$  if for some $t>0$ (and hence for all $t>0$) the angular derivative $\phi_t'(\zeta):=\angle\lim\limits_{z\to\zeta}\frac{\phi_t(z)-\zeta}{z-\zeta}$ exists finitely;
  \item {\it a boundary regular null point} of $f$ if $f'(\zeta):=\lim\limits_{r\to1^-}\frac{f(r\zeta)}{\zeta(r-1)}$ exists finitely. In another terminology such point is also called a {\it boundary regular critical point} of $f$.
\end{itemize}

\begin{theorem}\label{thm-boud-p}
  Let $f \in \G_0$ and $\{\phi_{t}\}_{t\geq 0}$ be the semigroup generated by $f$.
  \begin{itemize}
    \item [(i)] If $\zeta\in\partial\D$ is a boundary regular null point of $f$, then $f'(\zeta)$ is a negative real number.
    \item [(ii)] A point $\zeta\in\partial\D$ is a repelling fixed point of $\{\phi_{t}\}_{t\geq 0}$ if and only if it is a boundary regular null point of $f$. In this case $\phi_t'(\zeta)=e^{-tf'(\zeta)}.$
  \end{itemize}
\end{theorem}
It turns out that the existence of boundary regular null points of a generator is equivalent to the existence of so-called backward flow-invariant domains, see~\cite{ESZ}. For more details on these notions and Theorem~\ref{thm-boud-p} see \cite{DS-03, C-DM-P}, Sections 12.1--12.2 in \cite{B-C-DM-book} and Sections 2.2--2.4 in \cite{E-Sbook}.

\bigskip

\section{Basic concepts}\label{sect-basic}

\setcounter{equation}{0}

In this section we introduce basic notions of filtration theory and present two  filtrations that provide important dynamic properties for the semigroups generated by elements of their classes. The definitions and results presented here are mainly taken from \cite{BCDES} and~\cite{ESS}.\vspace{2mm}

\subsection{Notion of filtration}\label{subsect-notion-filt}
First we define filtration (or parametric embedding) of infinitesimal generators.

\begin{definition}\label{def-filt}
Let $J$ be a connected subset of $\R$. A {\sl filtration} of $\G_0$ is a family $\mathfrak{F}= \left \{\Ff_s\right\}_{s\in J},\
\mathfrak{F}_s\subseteq\G_0,$ such that $\mathfrak F_s\subseteq \mathfrak F_t$ whenever $s,t \in J$ and $s\le t$.
\end{definition}

To enable  a more detailed description of concrete filtrations, we introduce additional notions.

\begin{definition}\label{def-stric_filt}
 Let  $\mathfrak{F}= \left \{\Ff_s\right\}_{s\in J},\ \mathfrak{F}_s\subseteq\G_0,$ be a filtration of $\G_0$. We say that
the filtration $\left \{ \mathfrak F_s\right\}_{s\in J}$ is {\sl strict} if $\mathfrak F_s\subsetneq \mathfrak F_t$ for
$s<t,\ s,t\in J$. In this case, for $t\in J$ we define the boundary of $\mathfrak{F}_t$ by the formula
\begin{equation}\label{boundary}
\partial\mathfrak{F}_t=\mathfrak{F}_t \setminus\bigcup_{s\in J, s<t}\mathfrak{F}_s.
\end{equation}
\end{definition}

Note that if the boundaries defined by \eqref{boundary} are not empty for every $t\in J$, then the filtration $\mathfrak{F}$ is strict. At the same time, we do not know any strict filtration formed by sets with empty boundaries.

If an infinitesimal  generator is totally extremal in $\mathcal F$, the semigroup it generates has some extremal dynamical behavior among those semigroup generated by infinitesimal generators of the class $\mathcal F$, from which the reason of the name.

\begin{definition}\label{def-extrem1}
We say that a filtration $\mathfrak{F}=\{\mathfrak F_s\}_{s\in J}$ admits a net $\left\{f_s\right\}_{s\in J}$ of {\sl
totally extremal functions} if for every $s\in J $, the function $f_s$ is totally extremal for the class~$\mathfrak F_s$.
\end{definition}

The problems of finding a net of totally extremal functions as well as boundaries for a given  filtration is of interest and will also be considered for filtrations studied  along this survey.
We will see (in particular, in Subsection~\ref{subsect-pseu-anal}) that elements of a net of totally extremal functions for some filtration might not be boundary points of the corresponding classes. And boundary points might not be totally extremal. In this connection we ask:

\begin{question}
Is there a strict filtration that does not admit a net of totally extremal functions?
\end{question}

\begin{question}
Is there a strict filtration $\Ff$ such that $\partial\Ff_t=\emptyset$ for all $t\in J$?
\end{question}

The next question is important for theory of filtrations, but has not yet been studied.

Assume that  $\partial\Ff_t\neq\emptyset$ for all $t \in J$. Curious to see how far the element $\partial\Ff_t$ is from $\Ff_s$, $s<t$, $s \in J.$ To this aim consider a distance $d$ in $\G_0$ (it can be, for example, induced by a distance in $\Pp$).

\begin{question}
Given $s<t$ and $f^* \in \partial\Ff_t$, find $d(f^*, \Ff_s) := \inf\limits_{f\in \Ff_s} d(f,f^*)$.
\end{question}
\noindent Once one answered this question, it is natural to ask
\begin{question}
Find $d_{s,t} := \inf\limits_{f^*\in\partial\Ff_t} d(f^*,\Ff_s)$.
Further, find $\hat f_s\in\Ff_s$ and $\hat f_t\in\partial\Ff_t$ (if exist) such that $d_{s,t} =d(\hat f_s,  \hat f_t)$.
\end{question}

An additional problem can be posed as follows:
\begin{question}
Given a family $\left\{\Ao_{\al_1,\ldots,\al_n}\right\}$ of subsets of $\Ao$ depending on $n$ parameters, find conditions on the curve $J\ni s\mapsto (\al_1(s),\ldots,\al_n(s))$ providing the family $\left\{\Ao_{\al_1(s),\ldots,\al_n(s)}\right\}_ {s\in J}$ is a filtration of $\G_0$.
\end{question}
We address this issue several times throughout the text when relevant. Except for these cases, the question is open. 

\bigskip

\subsection{Exponentially squeezing filtration}\label{ssect-squeezing-filt}
\setcounter{question}{0}

To describe some important properties of exponentially squeezing semigroups (see Definition~\ref{def-squee}) and their connection with other classes of holomorphic functions, we use the family of sets $\Cf=\left\{\Cf_\al  \right\}_{\alpha \in[0,1]}$
 defined by
\begin{equation}\label{squeezing-filtr}
\Cf_\alpha:=\left\{f\in\Ao:\ \Re f(z)\overline{z} \geq (1-\alpha)
|z|^{2} \quad \forall z\in \D\right\}.
\end{equation}

The next theorem explains our interest to this family.
\begin{theorem}\label{th-squee}
The following assertions hold:
  \begin{itemize}
    \item[(a)] the family of sets $\Cf$ is a strict filtration such that $\Cf_0=\{\Id\}$ and $\Cf_1=\G_0$.
    \item[(b)] The filtration $\Cf$ admits a net $\left\{f_\alpha\right\}_{\alpha\in(0,1]}$ of totally extremal
functions, where
\[
f_\alpha(z)=z\,\frac{1+(1-2\alpha)z}{1+z} \, .
\]
    \item[(c)] The boundary $\partial\Cf_\al$ consists of such functions $f\in\Ao$ that satisfies $\inf\limits_{z\in\D}\Re \frac{f(z)}z =1-\al$. In particular, $f_\al\in\partial\Cf_\al $ for every $\al\in(0,1]$.
     \item[(d)]  Let $f\in\G_0$. Then $f\in\Cf_\alpha$ if and only if the semigroup $\{\phi_t\}_{t\ge0}$ generated by $f$ admits squeezing ratio $1-\alpha$, that is, $|\phi_t(z)| \leq e^{-(1-\alpha) t}|z|,\ z\in \D$.
  \end{itemize}
\end{theorem}

\begin{proof}
Obviously, $\Cf$ is a filtration by construction. In addition, the function $p_\alpha$ defined~by
\begin{equation*}
p_\alpha(z)= \frac{f_\alpha (z)}z= \frac{1+(1-2\alpha)z}{1+z}
\end{equation*}
maps the open unit disc $\D$ conformally onto the half-plane $\{w\in\C:\ \Re w> 1-\alpha \}$. So, a function $f(z)=zp(z)$ belongs to the class $\Cf_{\alpha},\ \alpha \in(0,1],$ if and only if $p(0)=1$ and $\Re p(z)\ge 1-\alpha$, that is, $p\prec p_\al$ and hence $f_\alpha$ is totally extremal for $ \Cf_{\alpha}$. Thus assertion (b) is proved.

Further, it is clear that $f\in\Cf_\alpha,\ \al\in[0,1],$ if and only if $\inf\limits_{z\in\D}\Re \frac{f(z)}z \ge1-\al$. If it is the case, $f\in\Cf_\beta$ for some $\beta<\al$ if and only if this infimum is greater than $1-\al$. Otherwise, $f\in\partial\Cf_\al$.   Thus $f_\alpha\in \Cf_\al$, what implies assertion (c) and strictness of the filtration (so that assertion (a) is completely proved).

Assertion (d) follows immediately from Proposition\ref{charactexp copy(1)}.
\end{proof}
Assertion (d) of this theorem implies that every element $f \in \Cf_\alpha$, $\alpha \in [0,1] $, generates the semigroup with no boundary repelling fixed point. This assertion
gives the necessary and sufficient condition on a function of class $\G_0$ to generate the exponentially squeezing semigroup and emphasizes the importance of filtration $\Cf$.

\begin{definition}\label{def-filt-squee}
The filtration  $\Cf=\left\{\Cf_\al  \right\}_{\alpha \in[0,1]}$  defined by \eqref{squeezing-filtr} is called the {\it squeezing filtration}.
\end{definition}

Let $\mathcal{F}\subset\G_0$. Then $K(\mathcal{F})\ge 1-\al$ if and only $\mathcal{F}\subset \Cf_\alpha$ by Theorem~\ref{th-squee}. In particular, it follows from Marx and Strohh\"acker's Theorem~\ref{thm-M-S} that if $f$ is a normalized convex function then for every $z\in \D$, $z\neq 0$,
\[
\Re \frac{f(z)}z>\frac12\,.
\]
Hence  every  normalized convex function is an infinitesimal generator and belongs to  the class $\Cf_{\frac12}$.
Therefore by assertion (d) of Theorem~\ref{th-squee}, each semigroup generated by a convex function is exponentially squeezing with squeezing ratio $\frac12\,$.

\begin{example}
{\rm Consider the Cauchy problem:
\[
\frac{\partial u(t,z)}{\partial t}+\log(1+u(t,z))=0,\qquad u(0,z)=
z\in \D.
\]
Since $f(z)=\log(1+z)$ is a convex function, its solution
satisfies the estimate
\[
|u(t,z)|\le e^{-t/2}|z|.
\]}
\end{example}

\bigskip

\subsection{Sectorial filtration}\label{ssect-sectorial-filt}
We have already discussed the feasibility of  analytic extension of a semigroup with respect to its parameter to a sector in Subsection~\ref{subsect-semig} (in particular, Definition~\ref{def-an-sg}). Proposition~\ref{rotation} suggests considering the family of sets $\Sf=\left\{\Sf_\al  \right\}_{\alpha \in[0,1]}$ defined~by
\begin{equation}\label{anal-in-sector}
\Sf_\al:=\left\{f\in\Ao:\ \left|\arg \frac{f(z)}z \right|\leq \frac{\pi\alpha}2 \quad\mbox{for all } \  z\in \D\setminus \{0\}\right\}.
\end{equation}

\begin{theorem}\label{th-sect} The following assertions hold:
  \begin{itemize}
\item[(a)] The family of sets $\Sf$ is a strict filtration such that $\Sf_0=\{\Id\}$ and $\Sf_1=\G_0$.
\item[(b)] The filtration $\Sf$ admits a net $\left\{f_\alpha\right\}_{\alpha\in(0,1]}$ of totally extremal functions, where
\[
f_\alpha(z)=z\left(\frac{1-z}{1+z} \right)^\alpha .
\]
    \item[(c)] The boundary $\partial\Sf_\al$ consists of such functions $f\in\Ao$ that satisfies ${\sup\limits_{z \in \D}\left|\arg \frac{f(z)}z \right|= \frac{\pi\alpha}2}$. In particular, $f_\al\in\partial\Sf_\al $ for every $\al\in(0,1].$

    \item[(d)] Let $f\in\G_0$. Then $f\in\Sf_\alpha$ if and only if the semigroup $\{\phi_t\}_{t\ge0}$ generated by $f$ can be analytically extended with respect to the parameter $t$ to the sector $\Lambda\left(\frac{\pi}{2}(1-\alpha)\right)$.
  \end{itemize}
\end{theorem}

\begin{proof}
The fact that $\Sf$ is a filtration follows immediately from \eqref{anal-in-sector}.
The function $p_\alpha$ defined by $p_ \alpha(z)=\frac{f_\alpha (z)}z= \left(\frac{1-z}{1+z} \right)^\alpha, $ maps the open unit disc $\D$ conformally onto the sector $\{\zeta\in\C:\ |\arg\zeta|< \frac{\pi\alpha}2 \}$. Hence a function $f(z)=zp(z)$ belongs to the class $\Sf_{\alpha},\ \alpha \in(0,1],$ if and only if $p(0)=1$ and $|\arg p(z)|\le \frac {\pi\alpha}2$. In other words, $p\prec p_\alpha$, that is, $f_\alpha$ is totally extremal for
$\Sf_{\alpha}$. This proves assertion (b).

Further, it is clear that $f\in\Sf_\alpha,\ \al\in[0,1],$ if and only if ${\sup\limits_{z \in \D}\left|\arg \frac{f(z)}z \right|\leq \frac{\pi\alpha}2}$. In this case, $f\in\Sf_\beta$ for some $\beta<\al$ if and only if this supremum is less than $\frac{\pi\alpha}2$. Otherwise, $f\in\partial\Sf_\al$.   Thus $f_\alpha\in \Sf_\al$. Therefore assertion (c) and strictness of the filtration (so, assertion (a) also) are proved.

Assertion (d) follows from Proposition~\ref{rotation}.
\end{proof}

The following fact was first noted in \cite[Corollary 63]{SD16} in a different context.
\begin{corollary}
If $f \in\G_0$ has a boundary regular null point, then $f \in \partial \Sf_1$.
\end{corollary}

Since assertion (d) of this theorem constitutes the necessary and sufficient condition on a function of class $\G_0$ to generate the semigroup that can be analytically extended to the sector $\Lambda\left(\frac{\pi}{2}(1-\alpha)\right)$ for a given $\alpha \in [0,1]$, it highlights the distinguishing feature of filtration~$\Sf$.

\begin{definition}\label{def-filt-sec}
The filtration  $\Sf=\left\{\Sf_\al  \right\}_{\alpha \in[0,1]}$ defined by \eqref{anal-in-sector} is called the {\it sectorial filtration}.
\end{definition}

Similarly to the above, we note that if $\mathcal{F}\subset\G_0$, then $B(\mathcal{F})\le \al$ if and only $\mathcal{F}\subset \Sf_\alpha$ by Theorem~\ref{th-sect}.

In the next sections,  for every given filtration  $\mathfrak{F}= \left \{\Ff_s\right\}_{s\in J}$ we try to find nets of totally extremal functions for $\Ff$ and boundary elements $f\in\partial\Ff_s$, as well as $K(\Ff_s)$ and $B(\Ff_s)$ for $s\in J.$

\bigskip

\section{Linear filtrations}\label{sect-_spec_linear}

\setcounter{equation}{0}

In this section, we deal with linear filtrations of first order, that is, those defined by expressions depending linearly on $f$ and $f'$.

\subsection{Generic linear filtration}\label{sect-linear}

Let $J$ be a connected subset of $\R$. The results of this subsection for $J=[0,1]$ were  obtained in \cite{BCDES}.

Denote by $\K$ the set of all functions $k: J\times [0,1)\to [-1,+\infty)$ such that for every $r\in [0,1)$ the function $k(\cdot,r)$ is non-decreasing on $J$.

\begin{definition}\label{def-linear-filtr}
Given $k\in\K$, we denote by $\mathfrak F[k] =\left\{ \mathfrak F_\al[k]  \right\}_{\alpha \in J}$ the family of sets defined~by
\begin{equation}\label{N-alpha}
\mathfrak F_\al[k] :=\left\{f\in\Ao:\  \Re \left[k(\alpha, |z|)\frac{f(z)}{z}+f'(z)  \right]\geq 0, \quad  z\in
\D\setminus \{0\}\right\}.
\end{equation}
\end{definition}

\begin{theorem}\label{filtra-N}
For any $k\in\K$, the family $\mathfrak F[k]$ is a filtration of $\G_0$. If, in addition, for some $\gamma\in J$, we have
\begin{equation}\label{eq-cresce}
k(\gamma,r)\ge \frac{1+r^2}{1-r^2},
\end{equation}
then
$\mathfrak F_{\gamma}[k] =\G_0$. 
\end{theorem}

\begin{proof}
Let $f\in \mathfrak F_\al[k]$ for some $\alpha\in J$. Let $p(z):=\frac{f(z)}{z}$. To see that $\mathfrak F_ \al[k] \subseteq
\G_0$, it is enough to show that $\Re p(z)\geq 0$ for all $z\in \D$. The inequality in formula \eqref{N-alpha} can be rewritten as
\begin{equation}\label{Eq2-p1}
\left(1+ k(\alpha, |z|)\right)\Re p(z)  +\Re zp'(z) \geq 0, \quad
 z\in \D.
\end{equation}
Then by the proof of \cite[Lemma 3.5.3]{SD}, $\Re p(z)\geq 0$ for all $z\in\D$, and, hence, $\mathfrak F_\al[k]\subset \G_0$ for all $\alpha\in J$.

Fix $\alpha,\beta\in J$ such that $\al< \beta$ and let $f(z)=zp(z)\in\mathfrak F_\al[k]$. By \eqref{Eq2-p1}, for all $z\in \D$,
\begin{equation*}
\Re zp'(z)\geq -(1+ k(\alpha, |z|))\Re p(z) .
\end{equation*}
Hence, since $k(\cdot, |z|)$ is non-decreasing,
for all $z\in \D$, $z\neq 0$, we have
\begin{equation*}
\begin{split}
\Re\left[k(\beta,|z|)\frac{f(z)}{z}+f'(z)\right]&=
\Re \left[ (1+ k(\beta,|z|))p(z)+zp'(z) \right]\\
&\geq \Re \left[ (1+k(\beta,|z|))p(z)-(1+k(\alpha,|z|)) p(z)
\right]
\\&=[k(\beta,|z|)-k(\alpha,|z|)] \Re p(z) \geq 0.
\end{split}
\end{equation*}
Therefore, $\mathfrak F_\al[k]\subseteq \mathfrak F_\beta[k]$,
that is, $\mathfrak F[k]$ is a filtration of $\G_0$.

Assume \eqref{eq-cresce} holds. We show that any $f\in\G_0$
belongs to $\mathfrak F_\gamma[k]$. Indeed, given $f\in\G_0$, let $p(z):=\frac{f(z)}{z},\ p(0)=1$. By \eqref{b-p}, $\Re p(z)\geq 0$ for all $z\in \D$. It is well-known that $\displaystyle
\left|zp'(z)\right|\le\frac{2\Re p(z)}{1-|z|^2}$ (see, for
example, \cite{Duren, GAW} or \cite{GI-KG}). Therefore, according to our
assumption,
\begin{eqnarray*}
\Re \left[k(\gamma, |z|)\frac{f(z)}{z}+ f'(z) \right]  =  \Re
\left[ \left(k(\gamma, |z|)+1\right) p(z) +zp'(z) \right] \\
\ge \left[ \left(k(\gamma, |z|)+1\right) -\frac{2}{1-|z|^2}
\right] \Re p(z)\ge0.
\end{eqnarray*}
The proof is complete.
\end{proof}

Setting $J=[0,1]$ and $k(\alpha,r):=\alpha$ in the above theorem, we immediately get:
\begin{corollary}\label{Noshiro}
$\displaystyle \RA\subset \left\{f\in\Ao:\  \Re \left[\frac{f(z)}{z}+f'(z)  \right]\geq 0 \quad\mbox{for all\ } z\in
\D\setminus \{0\}\right\} \subset \G_0.
$
\end{corollary}

Now we see how linear filtrations can be used to get some information on the boundary behavior of a semigroup.

Let $f\in \mathcal G_0$. We say that a boundary point $\zeta\in\partial\D$ is a {\sl boundary critical point of order
$\lambda\in (0,1]$} (cf. \cite{BraGum16, BCDES}) for $f$ if
\begin{eqnarray}\label{cond1}
\lim_{r\to 1^-} \frac{f(r\zeta)}{\zeta(1-r)^\lambda}= \omega \neq 0,\infty\qquad\mbox{and}\qquad
\lim_{r\to 1^-}\frac{f'(r\zeta)}{(1-r)^{\lambda-1}}=-\lambda\omega.
\end{eqnarray}
By \cite[Theorem 6.4]{BraGum16}, for every $\lambda\in (0,1]$ there exists $f_\lambda\in \mathcal G_0$ having a  boundary critical point of order $\lambda$. A simple example of a function that satisfies \eqref{cond1} with arbitrary $\lambda\in (0,1]$ is $f(z)=z(1-z\overline\zeta)^\lambda$. Since $\Re \frac{f(z)}{z}\geq 0$, it follows immediately that
\begin{equation}\label{omegagreat}
\Re \omega\geq 0.
\end{equation}

\begin{proposition}\label{regular-filtr}
Let  $k\in\K$ be such that for some $\alpha\in J$ the limit
\[
\ell(\alpha):=\lim_{r\to 1^-}k(\alpha, r)(1-r)
\]
exists finitely. If $f\in \mathfrak F_\al [k]$ has a boundary critical point $\zeta\in\partial\D$ of order $\lambda\in (0,1]$, then
\[
 \lambda\le \ell(\alpha).
\]
\end{proposition}

\begin{proof}
Without loss of generality, one can assume that $\zeta=1$. Setting $z=r\in (0,1)$ in \eqref{N-alpha} and dividing by $(1-r)^{\lambda-1}$,  we obtain
\[
\Re \left[ k(\alpha, r)(1-r)
\frac{f(r)}{r(1-r)^\lambda}+\frac{f'(r)}{(1-r)^{\lambda-1}
}\right]\geq 0.
\]
Taking limit as $r\to 1^-$, we obtain
\[
\Re [\ell(\alpha) \omega -\lambda\omega]\geq 0.
\]
Hence, $(\ell(\alpha)-\lambda) \Re \omega \geq 0$. By \eqref{omegagreat}, we the result follows.
\end{proof}

We do not know wether such filtrations are strict in general.  Below we prove the strictness of some particular filtrations and find their boundaries. Also we will see  some important filtrations that are not strict.

\bigskip

We have already proved in Theorem~\ref{filtra-N} that for every $k\in\K$, the family $\mathfrak F[k]$ is a filtration of $\G_0$. At the same time, particular choices of function $k$ may be of independent interest and provide useful results. Next we treat two such choices.

\subsection{Analytic filtration}\label{subsec-analytic}

The simplest choice of function $k\in\K$ is a constant function (and this is the only case where $k$ is analytic). First we concentrate on the filtration of type $\mathfrak F[k]$ obtained by choosing $\alpha \in J=(-\infty, 1]$ and
\[
k(\alpha, r)=k_0(\alpha, r):=\frac{\alpha}{1-\alpha}>-1\,.
\]
Thus the inequality in condition \eqref{N-alpha} can be rewritten as
\begin{equation}\label{N-2}
\Re \left[\alpha \frac{f(z)}{z} +(1-\alpha) f'(z) \right]\geq 0,
\quad  z\in \D\setminus \{0\}.
\end{equation}
Denote by $\Af_\al$ the set of all $f\in \Ao$ which satisfy \eqref{N-2}.
Note that the family of sets $\Af =\left\{ \Af_\al\right\}_{\alpha \in(-\infty,1]}$ forms a filtration by Theorem~\ref{filtra-N}.
\begin{definition}
The family $\Af =\left\{ \Af_\al\right\}_{\alpha \in(-\infty,1]}$ is called the {\sl analytic filtration}.
\end{definition}

Our first result, that is a particular case of Proposition~\ref{thm4-g} below, states that the Hadamard product with certain functions preserves the class  $\Af_\al.$

\begin{proposition}[cf. \cite{Gi-Ku}]\label{thm4}
    Let $f\in\Af_\alpha$ and $g\in\Ao$ be such that $\Re \frac{g(z)}z>\frac12\,.$ Then $f * g \in \Af_\alpha$.
\end{proposition}

We proceed with criteria of the membership of a function $f\in\Ao$ in the class $\Af_\alpha$ involving the Hadamard product \eqref{hadamar} obtained in \cite{Gi-Ku}, the Gau{\ss} hypergeometric function $\Gauss(a,b;c;z)$, see \eqref{gauss}, and Bernardi integral operator \eqref{I-ga}, see \cite{ESS}.

\begin{theorem}\label{so}
Let $\al\in(-\infty,1)$ and $F_1(z)=\frac{1+z}{1-z}$. For $\al<1$ we denote
\begin{equation}\label{f-alpha}
 F_\alpha(z)=I_{\frac{1}{1-\alpha}}[F_1](z)
=2\Gauss\left(1,\frac{1}{1-\alpha};\frac{2-\alpha}{1-\alpha};z\right)-1,
\end{equation}
and
\begin{equation}\label{f-alpha1}
 f_\alpha(z) :=zF_\al(-z).
\end{equation}

A function $f\in \Ao$ belongs to $\Af_\alpha$ if and only if one (hence all) of the following conditions holds:
\begin{itemize}
  \item [(1)] For all $\zeta$ with $\lvert \zeta \rvert =1$ we have  \[
    f(z) * z \left(\frac{ 1 -\alpha z}{(1  - z)^2}  - \frac{1+\zeta}{1-\zeta} \right) \neq 0.
\]
  \item [(2)] There is a function $q\in\Pp$ such that
\begin{equation*}
f(z)=zI_{\frac1{1-\al}}[q](z) = z\int_0^1q\left(t^{1-\alpha}z\right)dt.
\end{equation*}
Consequently, if $f\in\Af_\alpha$, then ${f(z)/z \prec F_{\alpha}(z).}$
\item[(3)] There is a probability measure $\mu$ on the unit circle $\partial\D$ such that
\begin{equation*}
f(z) = z\int\limits_{\partial\D} F_\alpha (z\overline{\zeta}) d\mu(\zeta).
\end{equation*}
\end{itemize}
\end{theorem}

This theorem can be obtained from Theorem~\ref{so-g} below taking $\lambda=1$.

It is worth mentioning that functions $F_\alpha$ defined by formula~\eqref{f-alpha} have important analytic properties.

Our next assertion follows from Theorem 5 in \cite{Rus75} with $n=1$ (see also \cite[\S 4]{SW16sp}). For an alternative proof suggested by P. Gumenyuk, see \cite{BCDES}.
\begin{proposition}\label{l-pasha}
For every $\al\leq 1$, the function $F_\al$ is a univalent  function.
\end{proposition}

The next two theorems present the main results of this subsection.

\begin{theorem}\label{th-F-alpha}
The following assertions hold:
  \begin{itemize}
\item[(a)] The family $\Af$ forms a filtration with $\Af_{-\infty}:=\bigcap_{\alpha<1}\Af_\alpha=\{\Id\}$, $\Af_0=\RA$ and $\Af_1=\G_0$.
\item[(b)] The filtration $\Af$ admits a net $\left\{f_\alpha\right\}_{\alpha\in(-\infty,1]}$ of totally extremal functions, where $f_\alpha$ is defined by \eqref{f-alpha1}.
 \item[(c)] For every $\al\in(-\infty,1]$ we have $f_\al\in\partial\Af_\al $. Consequently, the filtration $\Af$ is strict.
  \end{itemize}
\end{theorem}
\noindent This theorem is a particular case (namely,  $\lambda_{**}=1$) of Theorem~\ref{thm-gener-analy} below.

\vspace{3mm}

\indent To proceed we introduce (cf. \cite{BCDES, ESS}) the following three functions:
\begin{eqnarray}\label{kappa}
\varkappa(\alpha) &:=&  \int_0^1\frac{1-t^{1-\alpha}}{1+t^{1-\alpha}}\,dt = \Re F_\alpha(-1)= \inf_{z\in\D} \Re F_\alpha(z) , \\
  \delta(\al) &:=& \frac2\pi\cdot\max_{0<\theta<\pi} \arg F_\alpha(e^{i\theta}), \nonumber
\end{eqnarray}
where $F_\alpha$ is given in Theorem \ref{so},  and the function $b=b(\alpha)$ is the unique solution in $[0,1]$ to the equation
\[
\frac{\pi b}2 + \arctan\left((1-\alpha)b \right) =\frac\pi2\,.
\]

It can be seen (see~Fig.~\ref{fig:kappa}) that $\varkappa$ is a decreasing function with $\displaystyle \lim\limits_{\al\to-\infty}\varkappa(\alpha)=1$, $\varkappa(0)=2\ln 2-1$, $\varkappa(1)=0$. For more details see Lemma~\ref{lem-kappa} below. Functions $\delta$ and $b$ are increasing  with $\lim\limits_{\al\to-\infty}\delta(\al)=\lim\limits_{\al\to-\infty} b(\al)=0$ and $\delta(1)= b(1)=1$ (see Fig.~\ref{fig:B}).
Numerical calculations using Maple give $\varkappa(-5)\approx 0.8075$, $\delta(-5)\approx 0.2024$, $b(-5)\approx 0.3122$.
We will see that these functions play an important role in filtration theory.

\begin{figure}
\begin{center}
\includegraphics[angle=0,width=5cm,totalheight=5cm]{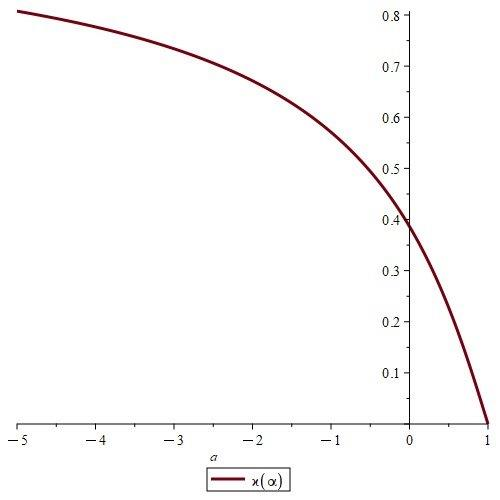}
\caption{The graph of $\varkappa(\alpha).$}\label{fig:kappa}
\end{center}
\end{figure}

\begin{figure}
\begin{center}
\includegraphics[angle=0,width=5cm,totalheight=5cm]{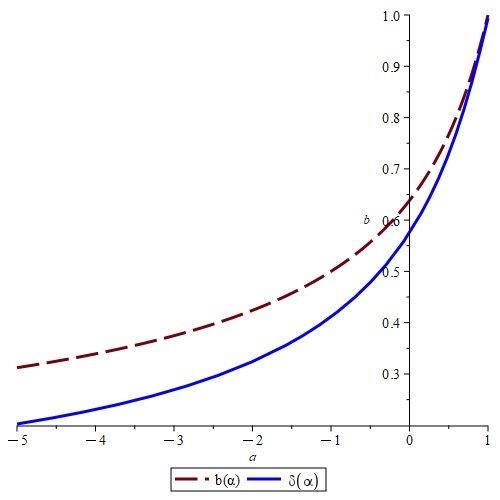}
\caption{The graphs of $b(\alpha)$ and $\delta(\alpha).$}\label{fig:B}
\end{center}
\end{figure}

\vspace{2mm}

Membership of a function $f$ to the analytic filtration
means that the semigroup generated by  $f$ owns certain dynamical properties, like the uniform rate of convergence, analytic extension of the semigroup with respect to its parameter (see Definitions~\ref{def-squee} and~\ref{def-an-sg} above). Indeed, Remark~\ref{rem-note}, assertion (d) in Theorem~\ref{th-squee}, assertion (d) in Theorem~\ref{th-sect} and inequality~\eqref{ineq-analy1-g} below imply:
\begin{theorem}\label{th-F-alpha-2}
Let  $\alpha<1$, $f\in\Af_ \alpha$ and $\left\{ \phi _{t}(z)\right\}_{t\ge0} $ be the semigroup generated by $f$. Then
\begin{itemize}
  \item [(i)] this semigroup has no repelling fixed points;
\item[(ii)]  this semigroup is exponentially squeezing with squeezing ratio $\varkappa (\alpha )$;
  \item [(iii)]  this semigroup  extends analytically to the sector
$\Lambda\left(\frac{\pi(1-\delta(\al))}2\right)$, where $\delta(\al) \le b(\al)$.
\end{itemize}
\end{theorem}

\begin{proof}
Indeed, the limit $\ell(\alpha)$ in Proposition~\ref{regular-filtr} is equal zero. Therefore $f$ has no boundary regular null points. Hence assertion (i) follows by Theorem~\ref{thm-boud-p}.

Further, assertion (ii) and the first part of assertion (iii) follows from Remark~\ref{rem-note} and Theorem~\ref{th-F-alpha}. So, we have to show that $\delta(\alpha)\le b(\al)$.

Consider the function $F_\alpha$ defined by \eqref{f-alpha1}. A straightforward calculation shows that
\[
F_\alpha(z)+(1-\alpha)zF_\alpha'(z)=\frac{1+z}{1-z}\,.
\]
Hence, by  \cite[Theorem~3.1c]{M-M} applied with $\lambda=1-\alpha$, we have $\displaystyle F_\alpha(z)\prec\left (
\frac{1+z}{1-z}\right)^b$, which completes the proof.
\end{proof}

As an immediate consequence of the last theorem and assertion (iii) in Proposition~\ref{l-pasha}, we conclude:
\begin{corollary}\label{cor-Aconseq}
The quantities $K(\Af_\alpha)$ and $B(\Af_\alpha)$ are given by
$$
K(\Af_\alpha)=\varkappa(\alpha)\quad\mbox{and}\quad B(\Af_\alpha)=\delta(\alpha).
$$
\end{corollary}
In particular, $K(\RA)=\varkappa(0)=2\ln2-1\approx 0.3863$ and $B(\RA)=\delta(0)\approx 0.5804$. Thus any function of the Noshiro--Warschawski class $\RA$ generates the  semigroup that satisfies $\left|\phi_t(z) \right|\le e^{-0.3862 t}|z|$ and can be analytically extended to the sector $\Lambda\left(0.2097\pi\right)$.

\bigskip

\subsection{Hyperbolic filtrations}\label{subsec-hyperb}

In this subsection we discuss two filtrations (very close one to another) that are linear combinations of Abate's formula and the inequality $\Re\frac{f(z)}{z}>0$ that follows from the Berkson--Porta representation (see Theorem~\ref{th-gen-prop}). Since the origin of both mentioned formulas is in the hyperbolic geometry of the unit disk, we call them {\it hyperbolic filtrations}. In this subsection we mainly follow \cite{BCDES}.

The first one is of type  $\mathfrak F[k]$ and can be obtained by taking $J=[0,1)$ and function $k_1\in \mathcal{K}$ defined
by
\[
k_1(\alpha, r):=\frac{\alpha r^2}{(1-\alpha)(1-r^2)}.
\]
In this case, condition \eqref{N-alpha} can be rewritten as
\begin{equation}\label{N-1}
\Re \left[\alpha |z|^2\frac{f(z)}{z}+(1-\alpha)(1-|z|^2) f'(z) \right]\geq 0, \quad  z\in \D\setminus \{0\}.
\end{equation}

\begin{definition}\label{def-hyp-filt}
We say that $f\in\Ao$ belongs to $\Hf_\al$ if $f$ satisfies \eqref{N-1}, and the family of the sets $\Hf =\left\{ \Hf_\al\right\}_{\alpha \in[0,1]}$ is called the {\sl hyperbolic filtration}.
\end{definition}

In fact,  the inequalities defining the filtration
are convex combinations of derivatives of inequalities containing the infinitesimal hyperbolic metric and  the hyperbolic distance:
\begin{equation*}
\omega (\phi_{t}(z),\phi_{t}(w) )\le \omega(z,w) \qquad \mbox{and}\qquad \kappa (\phi_t(z);f(\phi_t(z)))  \le  \kappa (z;f(z)).
\end{equation*}
The last two inequalities hold for every semigroup $(\phi_t(z))$ with infinitesimal generator $f$, where $\omega$ denotes the hyperbolic distance and $\kappa$ the infinitesimal hyperbolic metric on $\D$. This is the reason why we call this filtration {\sl hyperbolic}.

Note that $\Hf_0=\RA=\{f\in \Ao : \Re f'(z)>0 , \, z \in \D\}$ consists of univalent functions.  It is natural to ask: for a given $\alpha\in \left(0,1\right) $, how far away are the classes $\Hf_\alpha$  from the class $\RA$? The following statement answers this question.

\begin{proposition}[see \cite{SD16}]\label{der}
Let $\alpha\in \left( 0,1\right) $ and $f \in \Hf_\alpha$. Then
\begin{equation*}
\Re f^{\prime }(z)\geq -\frac{\alpha}{1-\alpha}\cdot\left(\frac{\left\vert z\right\vert }{1-\left\vert z\right\vert }\right)^2, \quad  z\in \D.
\end{equation*}%
\end{proposition}

\begin{proof}
It follows from Definition~\ref{def-hyp-filt} that
\begin{equation*}
(1-\alpha)\Re f^{\prime }(z)(1-\left\vert z\right\vert ^{2})\geq -\alpha\Re %
f(z)\overline{z}.
\end{equation*}
Since $\Re f(z)\overline{z}\geq 0,$ $z\in \D,$ we get by Harnack's inequality (see, for example, \cite{Duren, GI-KG, SD})  that
\begin{equation*}
\Re f(z)\overline{z}\leq \left\vert z\right\vert ^{2}\Re f^{\prime
}(0)\frac{1+\left\vert z\right\vert }{1-\left\vert z\right\vert }.
\end{equation*}%
So,
\begin{equation*}
\Re f^{\prime }(z)(1-\left\vert z\right\vert ^{2})\geq -\frac{\alpha}{1-\alpha}
\left\vert z\right\vert ^{2}\cdot\frac{1+\left\vert
z\right\vert }{1-\left\vert z\right\vert }.
\end{equation*}
Hence the assertion follows.
\end{proof}
Thus letting $\alpha \to 0$ we see that $\Hf_\alpha$ is `close' to the Noshiro--Warschawski class.
\vspace{2mm}

Observe further that $\Hf_1=\mathcal G_0$. At the same time, condition \eqref{eq-cresce} of Theorem~\ref{filtra-N} does not hold, so it is not a priori clear whether $\Hf_\alpha=\mathcal G_0$
for some $\alpha\in[0,1)$.
In fact, for $\alpha=\frac{2}{3}$\,, equation
\eqref{N-1} reduces to the Abate's formula \eqref{abate}.
 Therefore,
\[
\Hf_\alpha= \G_0\quad\mbox{for all } \alpha\in
\left[\frac{2}{3},\, 1\right].
\]

As for $\alpha \in \left[0, \frac{2}{3}\right)$, we state the following assertion that follows from Proposition~\ref{regular-filtr} similarly to the proof of Theorem~\ref{th-F-alpha-2} (i).
\begin{proposition}
Let $f \in \Hf_\alpha$, $\alpha \in \left[0, \frac{2}{3}\right)$. Then the semigroup generated by $f$ has no repelling fixed points.
\end{proposition}

Our next aim is to show that for $0<\alpha\leq\frac23$ the boundaries $\partial\Hf_\alpha$ are not empty, and hence the filtration $\{\Hf_\alpha\}_{\alpha\in \left[0,\frac23\right]}$ is strict. For $\al\in(0,\, 2/3]$ and $r\in (0,1)$, we denote
\begin{eqnarray*}
   \phi(\alpha, r) &:=& \frac{1-\al +r^2(2\al-1)}{2r(1-\al)+r^3(3\al-2)}\,, \\
   \psi(\alpha, r)&:=& \frac{(1-\alpha)+r^2(2\alpha -1)}{3r^2 (1-\alpha) + r^4(4\alpha-3) }\,.
\end{eqnarray*}

It can be shown that for every fixed $\al\in\left(0, \frac23\right]$ these functions attain their minima
\begin{equation}\label{eta}
\eta_1=\eta_1(\al):=\min_{r\in (0,1]}\phi(\alpha, r)>0\quad\mbox{and}\quad \eta_2=\eta_2(\al):=\min_{r\in (0,1]}\psi(\alpha, r)>0.
\end{equation}

\begin{theorem}\label{th-N-alpha}
Let $\al\in\left(0, \frac23\right]$ and $\eta_1,\eta_2$ be defined by \eqref{eta}. Then the boundary  $\partial\Hf_\alpha$ contains the functions $f_{\al,1},\ f_{\al,2}$ and $f_{\al,3}$ defined by
\[
f_{\al,1}(z):=z(1-z)^{\frac\alpha{2(1-\alpha)}},\quad  f_{\al,2}(z):= z(1+\eta_1z),\quad  f_{\al,3}(z):= z(1+\eta_2 z^2).
\]
Consequently, the filtration $\{\Hf_\alpha\}_{\alpha\in \left[0,\frac23\right]}$ is strict. Moreover, for any $f\in\G_0\setminus\Hf_\al$, let $r\in [0,1)$ be the largest number such that $f_r\in\Hf_\al,$ where $f_r(z):=\frac1r f(rz)$. Then $f_r\in\partial\Hf_\al.$
\end{theorem}
We omit the proof that can be found in the papers \cite{ESS} for $f_{\al,1}$, \cite{BCDES} for $f_{\al,3}$ and \cite{E-S-Tu1} for $f_{\al,2}$ and the last statement of the theorem.

\vspace{2mm}

Obviously, for each $\al\in(0,\frac23]$ we have $f_{\al,1}\notin\Cf_\beta$ for $\beta<1$ and $f_{\al,1}\in\Sf_{\frac{\al}{2(1-\al)}}$, where $\Cf$ and $\Sf$ are the squeezing and sectorial filtrations, respectively. Taking in mind Theorem~\ref{th-squee}, the inclusion $f_{\al,1}\in\Hf_\al$ implies the following fact:
\begin{corollary}\label{cor-K-hyperb}
  $K(\Hf_\al)=0$ and $B(\Hf_\al)\ge\frac{\al}{2(1-\al)}$  for every $\al\in\left(0,\frac23\right]$.
\end{corollary}

\vspace{2mm}

Another interesting filtration of type $\mathfrak F[k]$ was introduced in \cite{SD16}. It can be obtained by choosing $J=(-\infty,1)$ and function $k_2\in \mathcal{K}$ defined by
\[
k_2(\beta, r):=\frac{2 r^2}{(1-\beta)(1-r^2)}.
\]
In this case, condition \eqref{N-alpha} gets the form
\begin{equation}\label{N}
\Re \left(2f(z)\overline{z}+(1-|z|^{2})(1-\beta )\Re f' (z)\right)\geq 0.
\end{equation}
On the other hand, this inequality (as well as \eqref{N-1}) is a linear combination of Abate's formula and the Berkson--Porta condition.  Therefore it can be obtained by changing the parameter in \eqref{N-1}. Indeed, setting $\alpha=\frac{2}{3-\beta}$ there, one obtains~\eqref{N}.

\begin{definition}
We say that $f\in\Ao$ belongs to $\mathfrak{H}^1_\al$ if $f$ satisfies \eqref{N} and name $\Hf^1 =\left\{\Hf_\al^1\right\}_{\alpha \in(-\infty,1]}$ the {\sl modified hyperbolic filtration}.
\end{definition}

Due to the mentioned connection of the modified hyperbolic filtration $\Hf^1$ with filtration $\Hf$, we conclude the following:

\begin{theorem}\label{thm-gen 2}
  Filtration $\Hf^1$ has the following properties:
  \begin{itemize}
  \item [(i)]  $\Hf_\beta^1 =\G_0$ for all $\beta\in [0,1]$;
  \item [(ii)]  Filtration $\left\{\Hf_\beta^1\right\}_{\beta \in(-\infty,0]}$ is strict. Moreover, the boundaries $\partial\Hf_\beta^1$ are not empty.
  \item [(iii)] Let $f \in \Hf^1_\beta$, $\beta<0$. Then the semigroup generated by $f$ has no repelling fixed points.
  \item [(iv)] $K(\Hf^1_\beta)=0$ and $B(\Hf^1_\beta)\ge\frac{1}{1-\beta}$  for every $\beta<0$.
\end{itemize}
\end{theorem}
The proof follows immediately from the previous results in this section.

\bigskip

In the next subsection we consider linear filtrations that are not of  type  $\mathfrak F[k]$ defined in Subsection~\ref{sect-linear}.

\subsection{Generalized analytic filtrations}\label{subsec-gener-analytic}

In this subsection we present a family of filtrations for which the squeezing filtration $\Cf$ and the analytic filtration $\Af$ are, in a sense, its `border' cases. Moreover, some other filtrations considered here may serve a `homotopy' between $\Cf$ and $\Af$.

Let start with the two-parametric family of sets defined for parameters $\alpha\le1$ and $\lambda\in[0,1]$ as follows
\begin{equation}\label{2-param}
\Gf_{\al}^{\lambda}:=\left\{\Re \left[\alpha \frac{f(z)}{z} +(1-\alpha) f'(z) \right]\geq 1-\lambda,
\quad  z\in \D\setminus \{0\}\right\}.
\end{equation}
It was introduced in~\cite{K-R}, where an integral transform between different sets $\Gf_{\al}^{\lambda}$ was established. Clearly, $\Gf_\al^0=\{\Id\}$, $\Gf_{1}^\lambda=\Cf_\lambda$ and $\Gf_{\al}^1=\Af_\al$, while the sets $\Gf_0^\lambda$ were studied in~\cite{Hal}. It can be easily seen that the semigroup $\{\phi\}_{t\geq 0}$ generated by $f \in \Gf_0^\lambda$ satisfies $\left|\phi_t'(z)\right|\leq e^{(\lambda-1)t}$.

In the following assertion we show that these sets $\Gf_\al^\lambda$ consist of generators and the Hadamard product with certain functions preserves them.

\begin{proposition}\label{thm4-g}
    Let $\alpha\le1$ and $\lambda\in[0,1]$. Let $f\in\Gf_\alpha^\lambda$ and $g\in\Ao$ be such that $\Re \frac{g(z)}z>\frac12\,.$ Then $f\in\G_0$ and $f * g \in \Gf_\alpha^\lambda$.
\end{proposition}
\begin{proof}
By definition, $f\in\Gf_\al^\lambda\subset\Af_\al$. Hence $f\in\G_0$ as it is explained in Section~\ref{subsec-analytic} (see also  Theorem~\ref{filtra-N}).

 It follows from the Riesz--Herglotz representation (see Theorem~\ref{thm-RH}) that
  \begin{equation*}
   g(z)  = z\oint\limits_{\partial\D} \frac{1}{1- z\overline{\zeta}}\,d\mu(\zeta)
  \end{equation*}
  with some probability measure $\mu$ on the unit circle.

     Denote $F= f * g.$  Then $ z F'(z) = \left(z f'(z) \right) * g(z).$ Therefore
\begin{align*}
     &\frac{\alpha F(z)+ (1-\alpha)z F'(z)}{z} = \frac{\alpha(f(z)*g(z))+\left((1-\alpha)z f'(z)*g(z)\right)}{z} \\
= &  \frac{\alpha f(z)  + (1 - \alpha)z f'(z) }{z} * \frac{g(z)}z  =  \frac{\alpha f(z)  + (1 - \alpha)z f'(z) }{z} *\oint\limits_{\partial\D} \frac{1}{1- z\overline{\zeta}}\,d\mu(\zeta)  \\
=& \oint\limits_{\partial\D}  \frac{\alpha f(z\overline{\zeta})  + (1 - \alpha)z\overline{\zeta} f'(z\overline{\zeta}) }{z\overline{\zeta}} \,d\mu(\zeta) = \oint\limits_{\partial\D} \left( \al  \frac{f(z\overline{\zeta})}{z\overline{\zeta}} +(1-\al) f'(z\overline{\zeta}) \right)d\mu(\zeta).
\end{align*}
   Thus the function $\alpha \frac{F(z)}{z} +(1-\alpha) F'(z)$ takes values in the convex hull of the image of $\alpha \frac{f(z)}{z} +(1-\alpha) f'(z)$. So, the result follows.
  \end{proof}

 In addition to notations \eqref{f-alpha}--\eqref{f-alpha1}, we denote
\[
F_\al^\lambda(z)=\lambda F_\al(z)+1-\lambda\qquad \mbox{ and } \qquad f_\al^\lambda(z)=zF_\al^\lambda(z).
\]

Criteria  of membership of a function $f\in\Ao$ to the class $\Gf_\alpha^\lambda$ are described in the next theorem (cf. Theorem~\ref{so}).
\begin{theorem}\label{so-g}
A function $f\in \Ao$ belongs to $\Gf_\alpha^\lambda$ if and only if one (hence all) of the following conditions holds:
\begin{itemize}
  \item [(1)] For all $\zeta$ with $\lvert \zeta \rvert =1$ we have
  \[
    f(z) * z \left(\frac{ 1 -\alpha z}{(1  - z)^2}  -1- \frac{2\lambda\zeta}{1-\zeta} \right) \neq 0.
\]
  \item [(2)] There is a function $q\in\Pp$ such that $\Re q(z)\ge1-\lambda,$ $z\in\D$, and
\begin{equation}\label{q-g}
f(z)=z I_{\frac1{1-\al}}[q](z) = z\int_0^1q\left(t^{1-\alpha}z\right)dt.
\end{equation}
Consequently, if $f\in\Gf_\alpha^\lambda$, then $f(z)/z \prec F_{\alpha}^\lambda(z).$
\item[(3)] There is a probability measure $\mu$ on the unit circle $\partial\D$ such that
\begin{equation}\label{F-al-l-g}
f(z) = z\int\limits_{\partial\D} F_\alpha^\lambda (z\overline{\zeta}) d\mu(\zeta).
\end{equation}
\end{itemize}
\end{theorem}

\begin{proof}
According to formula \eqref{2-param}, the function $f$ belongs to the class $\Gf_\alpha^\lambda$ if and only if
\[
  \alpha  \frac{f(z)}{z}  + ( 1 - \alpha ) f'(z) \prec \frac{2\lambda z}{1-z}+1,
\]
       which ensure the existence of a function $\omega\in\Omega$ such that
\[
      \al  \frac{f(z)}{z}  + ( 1 - \al ) f'(z) = \frac{2\lambda \omega(z)}{1- \omega(z)}+1.
\]
In turn, this implies that
\[
     \alpha  f(z)   + ( 1 - \alpha ) zf'(z) \neq z\left( \frac{2\lambda \zeta}{1- \zeta}+1\right) \quad \mbox{for all }\  z\in \D \ \mbox{ and }\ \zeta\in\partial\D.
\]
Using properties \eqref{hadamar-pr} of the Hadamard product we obtain
\[
 f(z) * \left(  \frac{\alpha z}{1-z}  +  \frac{(1 - \alpha ) z}{(1-z)^2} - z\left( \frac{2\lambda\zeta}{1-\zeta} +1\right) \right) \neq 0,
 \]
       which completes the proof of (1).

To proceed, assume that $f\in\Gf_\al^\lambda$. Then there is a function $q$ of the Carath\'eodory class such that  $\Re q(z)\ge1-\lambda$ and
\begin{equation*}
\alpha \frac{f(z)}{z} +(1-\alpha) f'(z)=q(z).
\end{equation*}
Solving this differential equation, we get \eqref{q-g}.

Assume now that $f$ can be represented by formula \eqref{q-g} with a function $q\in\Hol$ that satisfies $\Re q(z)\ge1-\lambda$. It follows from the Riesz--Herglotz Theorem~\ref{thm-RH} that there is a probability measure $\mu$ on the unit circle such that
\[
q(z)= \lambda \oint\limits_{\partial\D} \frac{1+z\overline{\zeta}}{1-z\overline{\zeta}}\,d\mu(\zeta) +1-\lambda.
\]
Now formula \eqref{q-g} becomes
\begin{eqnarray*}
 f(z)  &=& z \int_0^1 \left( \lambda \oint\limits_{\partial\D} \frac{1+t^{1-\alpha}z\overline{\zeta}}{1-t^{1-\alpha}z\overline{\zeta}}\,  d\mu(\zeta) +1-\lambda \right)dt  \\
   &=& z \oint\limits_{\partial\D} \left(  \int_0^1\lambda \frac{1+t^{1-\alpha}z\overline{\zeta}}{1-t^{1-\alpha}z\overline{\zeta}}\,dt +1-\lambda \right) d\mu(\zeta)  =   z\int\limits_{\partial\D} F_\alpha^\lambda (z\overline{\zeta}) d\mu(\zeta),
\end{eqnarray*}
so \eqref{F-al-l-g}    holds.

Assume that condition (3) is fulfilled and hence
\[
f'(z)= \int\limits_{\partial\D} F_\alpha^\lambda (z\overline{\zeta}) d\mu(\zeta) + z \int\limits_{\partial\D} (F_\alpha^\lambda)' (z\overline{\zeta}) \overline{\zeta}d\mu(\zeta).
\]
 Therefore
 \begin{eqnarray*}
   \alpha \frac{f(z)}{z} +(1-\alpha) f'(z)  &=&  \int\limits_{\partial\D} F_\alpha^\lambda (z\overline{\zeta}) d\mu(\zeta) + (1-\al) \int\limits_{\partial\D} (F_\alpha^\lambda)' (z\overline{\zeta}) z\overline{\zeta}d\mu(\zeta) \\
    &=& \int\limits_{\partial\D}\left(\lambda \frac{1+ z\overline{\zeta}}{1- z\overline{\zeta}} +1-\lambda\right) d\mu(\zeta)
 \end{eqnarray*}
 by the definition of functions $F_\al^\lambda$. This implies that $f\in\Gf_\al^\lambda$.  The proof is complete.
\end{proof}

Following an idea in \cite{E-S-Tu2}, we present a property of functions $f_\al^\lambda$.
\begin{proposition}\label{prop-new1-g}
  Let $\alpha,\beta\le1$ and $\lambda,\nu\in(0,1]$. Then the function $f_{\al}^\lambda$
   \begin{itemize}
     \item [(a)] belongs to $\Gf_\al^\lambda$ and is totally extremal for this class;
     \item [(b)] does not belong to $\Gf_\beta^\nu$ whenever $\frac{1-\beta}\nu>\frac{1-\alpha}\lambda $. Consequently, taking $\nu=1$, we conclude that $f_1^\lambda\in \partial\Af_1$ for any $\lambda\in(0,1]$.
   \end{itemize}
\end{proposition}

\begin{proof}
  First we note that
  \[
  \Re \left(\al \frac{f_{\al}^\lambda(z)}{z} +(1-\al)(f_{\al}^\lambda)'(z) \right) =  \lambda \Re \left(\al \frac{f_\al(z)}{z} + (1-\al)(f_\al)'(z) \right) +1-\lambda\ge 1-\lambda\,,
  \]
  and hence $f_\al^\lambda\in \Gf_\al^\lambda$.

  If a function $f$ belongs to $\Gf_\al^\lambda$, it admits representation \eqref{q-g}. The subordination $q\prec F_1^\lambda$ implies that $\frac{f(z)}{z}\prec F_\al^\lambda$ by Hallenbeck--Ruscheweyh's Theorem~\ref{thm-H-R-75}. Therefore, for every $\lambda \in\C$,
\begin{equation}\label{ineq-analy1-g}
\min_{|z|=r} \Re \left(\lambda\frac {f(z)}{z}\right)\ge
\min_{|z|=r}\Re\left(\lambda F_\alpha^\lambda(z)\right),
\end{equation}
so the function $f_\alpha^\lambda$ is totally extremal for the set $\Gf_\al^\lambda$.

Further, the Taylor coefficients of $f_\al^\lambda$ are $a_n=\frac{2\lambda}{n(1-\al)+1}$. Assume that $f_\al^\lambda\in\Gf_\beta^\nu$, that is, $\Re \left(\beta \frac{f_{\al}^\lambda(z)}{z} +(1-\beta)(f_{\al}^\lambda)'(z) \right)\ge1-\nu$, or equivalently, $\Re\frac{F_\al^\lambda(z) +(1-\beta)z (F_\al^\lambda)'}{\nu}\ge0$. This implies that $\frac{a_n(n(1-\beta)+1)}{\nu}\le2$ for all $n\in\N$. This leads to
\[
\frac{2\lambda}{n(1-\al)+1} \le \frac{2\nu}{n(1-\beta)+1},\qquad n\in\N.
\]
Therefore $\frac{1-\beta}{\nu}\le\frac{1-\al}{\lambda}\,.$ So, the conclusion follows.
\end{proof}
Using Remark~\ref{rem_27_02} one sees that statement (a) in Proposition~\ref{prop-new1-g} implies the following
\begin{corollary}\label{cor-KB-g}
$ K(\Gf_\alpha^\lambda)=\lambda\varkappa(\alpha)+1-\lambda$\quad and \quad $B(\Gf_\alpha^\lambda)=\frac2\pi\cdot\max\limits_{0<\theta<\pi} \arg F_\alpha^\lambda(e^{i\theta}).$
\end{corollary}

\vspace{2mm}

It is clear that the two-parameter family $\left\{\Gf_\alpha^\lambda\right\}_{a\le1,\lambda\in[0,1]}$ contains numerious different  filtrations. In addition to the squeezing and analytic filtrations mentioned above (presented in Fig.~\ref{fig-a} by bold right vertical and upper horizontal lines, respectively), notice that for every given $\al_0\le1$, the family of sets $\left\{\Gf_{\al_*}^\lambda\right\}_{\lambda\in[0,1]}$ is obviously a filtration (bold left vertical line in Fig.~\ref{fig-a} for $\al_0=-2$; dash lines will be explained later).

To study more filtrations formed by these sets the following fact is needed.

\begin{lemma}[see Corollary 2.4.4 in \cite{E-R-S-19}]\label{lemma-g}
 Let $A>0,\ B<1$ and $p\in\Pp$ satisfy $\Re\left(p(z)+Azp'(z) \right)\ge B.$ Then for every $A_1\in[0,A]$ we have
 \[
 \Re\left(p(z)+A_1zp'(z) \right)\ge B_1:=(1-B)\frac{A-A_1}{A}\cdot\varkappa(1-A) +B.
 \]
\end{lemma}

Recall that function $\varkappa$ is defined by \eqref{kappa}, see Fig.~\ref{fig:kappa}. In addition to $\varkappa$, consider three functions defined as follows:
\[
\chi_1(\al):=(1-\alpha)(1-\varkappa(\alpha)),\quad \chi_2(\al):=\frac{\varkappa(\alpha)}{1-\alpha},\quad \chi_3(\al):=\frac{(1-\alpha)(1-\varkappa(\alpha))}{\varkappa(\alpha)}.
\]

Their graphs are presented in Fig.~\ref{figs_of_chis}. In fact, behavior of these functions reflects certain properties of $\varkappa$. We summarize them in the following technical lemma.

\begin{lemma}\label{lem-kappa}
We state that
\begin{itemize}
  \item [(i)] the function $\varkappa$ maps $(-\infty,1]$ onto $(0,1]$ with $\varkappa(0)=2 \ln 2-1$ and satisfies $\varkappa'(\al)<\varkappa''(\al)<0$ and $1-\varkappa(\al)<\frac{3-2\al}{2}\left|\varkappa'(\al)\right|$; consequently, it is decreasing and concave;
  \item [(ii)] the function $\chi_1$ is decreasing  and maps $(-\infty,1]$ onto $[0,2\ln2)$;
  \item [(iii)] the function $\chi_2$ is increasing and maps $(-\infty,1)$ onto $(0,\frac12)$.
\item [(iv)] the function $\chi_3$ is increasing and maps $(-\infty,1)$ onto $(2\ln2,2)$.
\end{itemize}
\end{lemma}

\begin{figure}[b]
\centering
\begin{subfigure}{.33\textwidth}
  \centering
  \includegraphics[width=.9\linewidth]{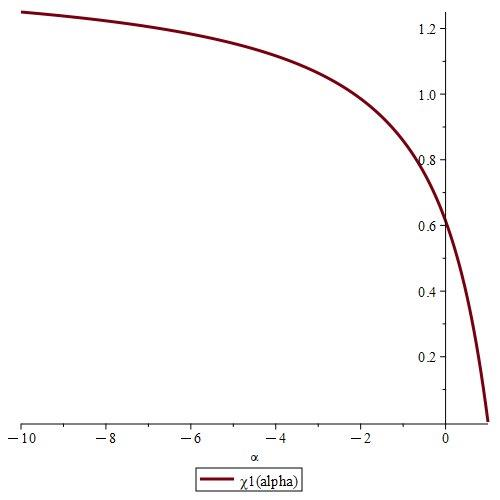}
  \label{fig-chi1}
\end{subfigure}%
\begin{subfigure}{.33\textwidth}
  \centering
  \includegraphics[width=.9\linewidth]{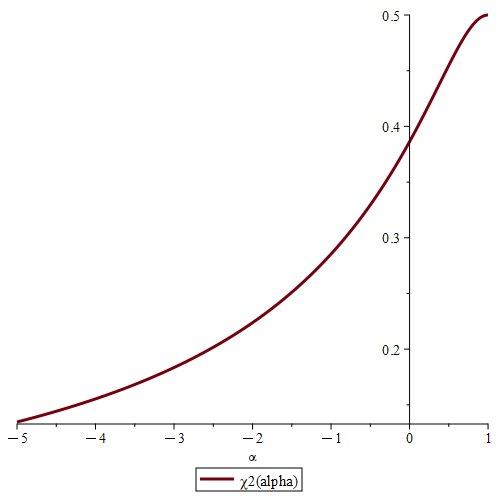}
  \label{fig-chi2}
\end{subfigure}
\begin{subfigure}{.33\textwidth}
  \centering
  \includegraphics[width=.9\linewidth]{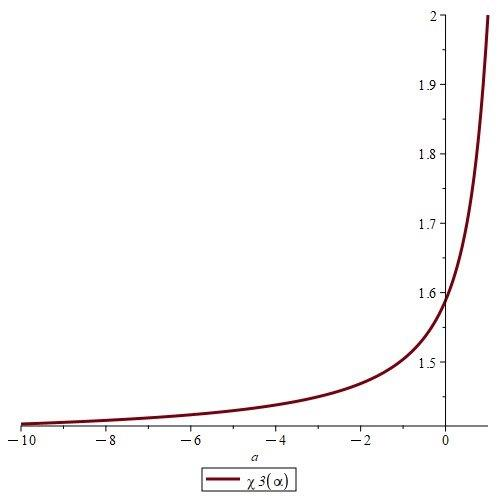}
  \label{fig-chi3}
\end{subfigure}
\caption{Graphs of functions $\chi_1,\chi_2,\chi_3$}
\label{figs_of_chis}
\end{figure}


For any $\lambda_*\in(0,1)$ by statement (iii) of Lemma~\ref{lem-kappa}, there is the unique  solution less than $1$ to the equation $\lambda_*=\frac{2\varkappa(\al)}{1-\al}$. We denote it by $\alpha_*$ and consider the sets
\[
  \Gf_\al^{(0)} := \Gf_\alpha ^ {\lambda_0(\al)}, \quad  \mbox{where}\quad   \alpha\in[\al_*,1]\quad \mbox{and} \quad \lambda_0(\al)=\frac{\lambda_*}{2\chi_2(\al)} .
\]
 (For illustration, in Fig.~\ref{fig-a}  the sets $\Gf_\al^{(0)}$ with $\lambda_*=0.19,\ \al_*\approx-8.11$ are presented by the descending dashed curve.)

\begin{figure}[b]
\begin{center}
\includegraphics[angle=0,width=5cm,totalheight=5cm]{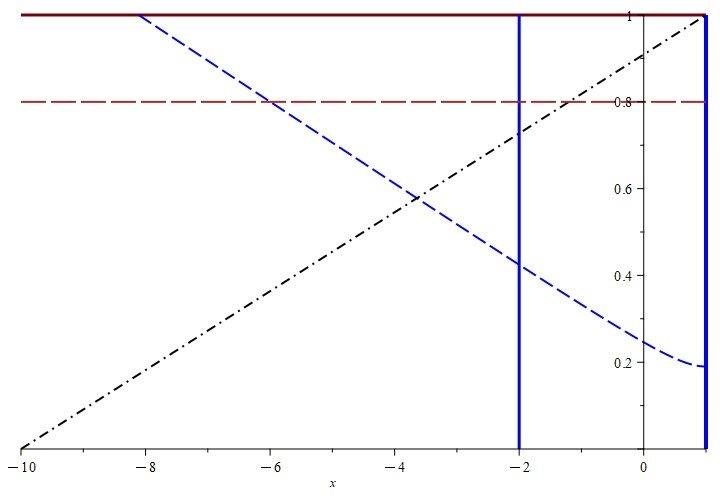}
\caption{Different filtrations in the range of parameters $(\al,\lambda)$.}\label{fig-a}
\end{center}
\end{figure}

\begin{theorem}\label{thm-gener-analy}
 Let $\lambda=\lambda(\al)$ be a continuous function on an interval $J$ such that the function $\lambda(\al)\chi_2(\al)$ is non-decreasing on $J$.
Then the family $\left\{\Gf_\al^{\lambda(\al)}\right\}_{\alpha\in J }$ is a strict filtration with non-empty boundaries of its  sets. Moreover, it admits a net of totally extremal functions.
\end{theorem}

\begin{proof}
First we prove this theorem for the family $\Gf^{(0)}$, for which $\lambda(\al)\chi_2(\al)$ is a constant.
Let $\al_*\le\alpha<\beta\le1$ and $f\in\Gf_\al^{(0)}$, that is, for all $z\in \D\setminus \{0\}$, the function $f$ satisfies
  \[
 \Re \left[\alpha \frac{f(z)}{z} +(1-\alpha) f'(z) \right]\geq 1-\lambda_0(\alpha)   .
  \]

  Function $f$ can be represented in the form $f(z)=zp(z)$. Substituting this in the inequality in~\eqref{2-param}, we get $\Re\left(p(z)+(1-\al)zp'(z) \right)\ge1-\lambda_0(\al).$ Therefore we can apply Lemma~\ref{lemma-g} with $A=1-\al$, $B=1-\lambda_0(\al)$ and $A_1=1-\beta$. This leads to
  \[
  \Re\left(p(z)+(1-\beta)zp'(z) \right)\ge \lambda_0(\al)\frac{\beta-\al}{1-\al}\cdot\varkappa(\al)+1-\lambda_0(\alpha).
  \]

  If we show that the term in the right-hand side of this inequality is not less than $1-\lambda_0(\beta)$, this implies $\Gf^{(0)}$ is a filtration.  Thus we have to show that
    \[
  \lambda_0(\al)\left(1-\frac{\beta-\al}{1-\al}\cdot\varkappa(\al) \right) \le\lambda_0(\beta),
  \]
  or, which is the same,
  \begin{eqnarray*}
     &&  \frac{1-\alpha}{\varkappa(\al)}\left(1-\frac{\beta-\al}{1-\al}\cdot\varkappa(\al) \right) \le\frac{1-\beta}{\varkappa(\beta)}, \\
     &&  \frac{1-\alpha}{\varkappa(\al)} -(1-\al) \le \frac{1-\beta}{\varkappa(\beta)} - (1-\beta).
 \end{eqnarray*}
 The last inequality follows from statement (iv) above. So, $\Gf^{(0)}$ is a filtration.

 Further, let $\al_*\le\gamma<\al$. By statement (i) of Lemma~\ref{lem-kappa},
 \[
  \frac{1-\gamma}{\lambda_0(\gamma)} = \frac2{\lambda_*}\varkappa(\gamma)> \frac2{\lambda_*}\varkappa(\al) =\frac{1-\al}{\lambda_0(\al)}.
\]
Therefore, it follows from Proposition~\ref{prop-new1-g} that the function $f_{\al}^{\lambda_0(\al)}$ belongs to $\Gf_\al^{(0)}$ and does not belong to $\Gf_\gamma^{(0)}$. Thus $f_{\al}^{\lambda_0(\al)}\in\partial \Gf_\al^{(0)}$ and the filtration is strict.

Turn now to a general case of a family  $\left\{\Gf_\al^{\lambda(\al)}\right\}_{\alpha\in J }$.  The fact that the function $\lambda(\al)\chi_2(\al)$ is non-decreasing implies that the curve $\lambda=\lambda(\al)$ either intersects every curve of the form $\lambda=\lambda_0(\al)$ at a single point, or (maybe, partially) coincides with it.

Let $\al<\beta\le1$ and $f \in \Gf_\al^{\lambda(\al)}$. Choose $\lambda_*$ such that $\lambda_0(\al)=\lambda(\al)$. Hence $f \in \Gf_\al^{(0)} \subset \Gf_\beta^{(0)}$ by  the proved part of the theorem.

So, $\lambda_0(\beta)\le \lambda(\beta)$ and one concludes that  $\Gf_\beta^{(0)}\subset \Gf_\beta^{\lambda(\beta)}.$
Thus the family $\left\{\Gf_\al^{\lambda(\al)}\right\}_{\alpha\in J }$ is a filtration. The strictness follows from the proved fact that filtration $\Gf^{(0)}$ is strict (cf. Proposition~\ref{prop-new1-g} (b)).
The existence of a net of totally extremal functions follows from Proposition~\ref{prop-new1-g} (a).
\end{proof}

\begin{example}
Take any $\lambda_{**}\in(0,1]$, $\alpha_{**}<1$ and consider the sets
\begin{equation*}
\begin{array}{lll}
   \Gf_\al^{(1)} := \Gf_\alpha ^ {\lambda_{**}},&  \mbox{ where } & \alpha\in(-\infty,1] ,\vspace{2mm} \\
  \Gf_\al^{(2)} := \Gf_\alpha ^ {\lambda_2(\al)}, &  \mbox{ where } &  \alpha\in[\al_{**},1]\mbox{ and } \lambda_2(\al)=\frac{\al-\al_{**}}{1-\al_{**}}.
\end{array}
 \end{equation*}
 Theorem~\ref{thm-gener-analy} implies that $\Gf^{(1)}$ and $\Gf^{(2)}$ are strict filtrations. In Fig.~\ref{fig-a}  the sets $\Gf_\al^{(1)}$ with $\lambda_{**}=0.8$ are presented by the long dash line, and $\Gf_\al^{(2)}$ with $\al_{**}=-10$ by the dash-dot line.
 \end{example}

\bigskip

\subsection{Open questions}\label{ssect-quest1}

Recall that generic linear filtrations  $\mathfrak F[k] =\left\{ \mathfrak F_\al[k]  \right\}_{\alpha \in J}$, $k\in\K$, were defined by \eqref{N-alpha}. We do not know whether such filtrations are strict in general.

\begin{question}
What conditions on a function $k(\alpha,|z|)$ imply that the filtration $\mathfrak F[k]$ is strict?  How to find boundaries of its sets?
\end{question}

Note that both analytic $\Af$ and hyperbolic filtrations $\Hf$ obtained by specific choices $k\in\K$ are strict.

\begin{question}
  What are boundaries $\partial\Af_\al$, $\al\le1$, and $\partial\Hf_\beta,\ \beta\in(0,1]$?
\end{question}

Concerning the analytic filtration $\Af$, we assume that the following conjecture is true.
\begin{conjecture}
The boundary $\partial\Af_\al$ coincides with the set $\left\{zF_{\alpha}(e^{i\theta}z): \theta\in\R  \right\}$ for every $\al\le1$.
\end{conjecture}

As for the hyperbolic filtration $\Hf$, the searching of the whole boundaries seems to be more complicated. It was shown in Theorem~\ref{th-N-alpha} that the functions
\[
f_{\al,1}(z):=z(1-z)^{\frac\alpha{2(1-\alpha)}},\quad  f_{\al,2}(z):= z(1+\eta_1z),\quad  f_{\al,3}(z):= z(1+\eta_2 z^2)
\]
 belong to the boundary $\partial\Hf_\alpha,\ \al\in(0,\frac23].$
We believe that
\begin{conjecture}\label{conj-beta}
  $f\in\partial\mathfrak{H}_\alpha$ whenever $\displaystyle f(z)=z\left( \frac{1+z}{1-z}\right) ^{\frac\alpha {2(1-\alpha)}}$ for $\alpha\in(0,2/3]$.
\end{conjecture}
For more explanation regarding this conjecture see \cite{ESS}. \vspace{2mm}

Remind that $\Af_0=\Hf_0=\RA$ and so, this set consists of univalent functions.

\begin{question}
  Does there exist $\beta>0$ such that $\Af_\beta$ (respectively, $\Hf_\beta$) consists of univalent functions? Find suprema of such $\beta$.
\end{question}
In connection to this question, we recall that the set $\Af_1=\Hf_{\frac23}=\G_0$ contains not only univalent functions.\vspace{2mm}

Recall also that the `starting' set for the analytic filtration is the singleton $\{\Id\}$ while for the hyperbolic one $\Hf_0=\RA$.
\begin{question}
   Is it possible to complete the filtration $\Hf$ by a `natural way' to negative values of the parameter, say $\alpha\ge-a,\ a>0,$ such that $\Hf_{-a}=\{\Id\}$?
\end{question}

In addition, the following question is still open.
\begin{question}
Does the hyperbolic filtration $\Hf$ admit a net of totally extremal functions?
\end{question}

We have already shown in Corollary~\ref{cor-K-hyperb} that for each $\al$ there is no common squeezing ratio for semigroups generated by elements of $\Hf_\al$ and that $B(\Hf_\al)\ge\frac{\al}{2(1-\al)}$. The question of finding $B(\Hf_\al)$ or at least its upper bound remains open.
\begin{question}
  Calculate $B(\Hf_\al).$
\end{question}

As for generalized analytic filtrations presented in Section~\ref{subsec-gener-analytic} the most important problem is

\begin{question}
Prove statement (iv) after Lemma~\ref{lem-kappa}
\end{question}
This will complete the proof of Theorem~\ref{thm-gener-analy}.
If this is done, we turn to filtration  $\Gf^{(0)}$. It is of special interest since in its definition the parameter $\lambda_0(\al)= \frac{\lambda_*}2\frac{1-\alpha}{\varkappa(\al)}$ is {\it decreasing}.

\begin{question}
Find boundaries $\partial\Gf^{(0)}_\al$.
\end{question}

We already know that semigroups generated by elements of squeezing and analytic filtrations have no  boundary repelling fixed points. So the following conjecture is natural:
\begin{conjecture}
Every element $f \in \Gf^{(0)}_\alpha$, $\alpha \in [\alpha_*,1] $, generates the semigroup with no boundary repelling fixed point.
\end{conjecture}

\bigskip

\section{Non-linear filtrations}\label{sect-non-l}

\setcounter{equation}{0}

In Section \ref{sect-_spec_linear}, our focus was on filtrations depending linearly on $f$ and $f'$. In this section we present results on several filtrations determined by non-linear expressions depending on $f, f'$ and $f''$. The study of such filtrations is usually more complicated. To make it easier to read, in the first subsection we present filtrations formed by classes of functions well-known in geometric function theory. Of course, in general these classes consist not only of generators and not always are embedded one in another. Therefore we need to specify the conditions under which a family of some sets is a filtration of $\G_0$.

\subsection{Starlike and prestarlike filtrations}\label{subsec-prestar}

We begin with a refitment of Marx--Strohh\"{a}cker's Theorem~\ref{thm-M-S}. Namely, for $0<\alpha \leq\frac{1}{2}$ consider the class (cf. \cite{Duren, GI-KG})
\begin{equation*}
\mathcal{S}^*(1-\alpha)=\left\{f \in \Ao : \Re \frac{zf'(z)}{f(z)}\geq 1-\alpha\right\}.
\end{equation*}

\begin{theorem}\label{thm-st-filtr}
The following assertion hold:
\begin{itemize}
  \item [(a)] The family of the sets $\{\mathcal{S}^*(1-\alpha)\}_{\al\in\left[0,\frac{1}{2}\right]}$ is a strict filtration of $\G_0$;
  \item [(b)] Let $f_\alpha,\, \alpha\in\left(0,\frac12\right]$, be defined by $f_\alpha(z)=\frac{z}{(1-z)^{2\alpha}}$, then $f_\al\in\partial \mathcal{S}^*(1-\alpha)$;
  \item [(c)]  The net $\left\{f_\alpha\right\}_{\alpha\in\left(0,\frac12\right]}$ is a net of totally extremal functions for this filtration;
  \item [(d)] For every $\al\in\left[0,\frac{1}{2}\right]$ we have $K(\mathcal{S}^*(1-\alpha))=\frac{1}{2^{2\alpha}}$ and $B(\mathcal{S}^*(1-\alpha))=2\al$.
\end{itemize}
\end{theorem}

\begin{proof}
Obviously $\mathcal{S}^*(1-\beta)\subseteq\mathcal{S}^*(1-\al)$ for  $\beta\le\al$ by construction. In addition, it is known (see for instance, Example~\ref{exa-star}) that the function $f_\al$ defined by $f_\alpha(z)=\frac{z}{(1-z)^{2\alpha}}$ is totally extremal for the set $\mathcal{S}^*(1-\alpha)$. The direct calculation gives $\frac{zf_\alpha'(z)}{f_\alpha(z)} = g_\alpha(z)$, where the function $g_\alpha(z)=\frac{1-(1-2\alpha)z}{1-z}$
maps the open unit disk $\D$ conformally onto the half-plane $\{w\in\C:\ \Re w> 1-\alpha \}$. So, $f_\alpha\notin\mathcal{S}^*(1-\beta)$ for $\beta<\al$.
Thus assertions (a), (b) and (c) are proven.

Consider now continuous extension of the function $f_\alpha$ on $\partial\D\setminus\{1\}$. Substituting $z=e^{i\phi},\phi\in(0,2\pi)$,  we get
  \begin{equation*}
    \Re\frac{f_\alpha(e^{i\phi})}{e^{i\phi}}=\Re \frac{1}{(1-e^{i\phi})^{2\alpha}} =\Re\left(\frac{e^{i(\pi-\phi)/2}}{2\sin(\phi/2)}\right)^{2\alpha}=\frac{\cos \left((\pi-\phi)\alpha\right)}{(2\sin(\phi/2))^{2\alpha}}.
  \end{equation*}
  When $\alpha>\frac12$ and $\phi$ is close enough to zero (or, to $2\pi$), the nominator takes negative values while the denominator is positive; hence, $f_\alpha\not\in\G_0$. Thus $\mathcal{S}^*(1-\alpha)\not\subset\G_0$.

  Otherwise,  when $\alpha\le1/2$, the nominator and denominator are positive for all $\phi$; hence, $f_\alpha\in\G_0$. Moreover, for $\alpha=1/2$, $\Re\left[{f_\alpha(e^{i\phi})}/{e^{i\phi}}\right]=1/2$ for all $\phi$. For $\alpha<1/2$, using standard methods of analysis, one can see that the minimum of $\Re\left[{f_\alpha(e^{i\phi})}/{e^{i\phi}}\right]$ is attained at $\phi=\pi$ and equals ${1}/{2^{2\alpha}}$.

 In addition, by Example~\ref{exa-star}, for any $f\in \mathcal{S}^*(1-\alpha)$ with $\alpha\le1/2$  we have
  \begin{equation*}
     \inf_{z\in\D} \Re\frac{f(z)}{z} \ge\inf_{z\in\D}\Re  \frac{1}{(1-z)^{2\alpha}}=\frac{1}{2^{2\alpha}}>0.
  \end{equation*}
  So,  $\mathcal{S}^*(1-\alpha)\subset\G_0$ by Theorem~\ref{th-gen-prop} and $f$ generates the exponentially squeezing semigroup with squeezing ratio $\frac{1}{2^{2\alpha}}$.

  Further, for $\phi\in(0,2\pi)$ we calculate
   \begin{equation*}
     \sup_{z\in\D} \arg \frac{f(z)}{z}\le \sup_{z\in\D}\arg  \frac{1}{(1-z)^{2\alpha}}
     =\sup_{z\in\D}\arg \left(\frac{e^{i(\pi-\phi)/2}}{2\sin(\phi/2)}\right)^{2\alpha}=\alpha \pi.
\end{equation*}
 Thus $f$ generates a semigroup that can be analiticaly extended to the sector $\Lambda\left(\frac{\pi(1-2\alpha)}{2}\right)$ by Proposition~\ref{rotation}. This proves assertion (d).
\end{proof}
A part of assertion (d) is contained in \cite[Theorem 5]{E-S-Tu2}.  \vspace{3mm}

Observe that filtration $\{\mathcal{S}^*(1-\alpha)\}_{\al\in\left[0,\frac{1}{2}\right]}$  is a particular case of a general scheme based on Janowski functions $\So^*(a,b)$, see  \eqref{Jano}. Simple geometric considerations regarding Janowski functions lead to the following fact.
\begin{theorem}\label{thm-Jano}
Let $-1\le b<a\le1$. Then the following assertions hold:
\begin{itemize}
  \item [(a)] $B(\So^*(a,b))=\frac{2(a-b)}{\pi b}\arcsin b$;
  \item [(b)] If, in addition, $a\leq b+ \frac{\pi b}{2\arcsin b}$, then $\So^*(a,b)\subset \G_0$;
  \item [(c)] If, in addition, $0<\frac ab\le2$, then and $K(\So^*(a,b))=(1-|b|)^{\frac{a-b}{|b|}}$.
\end{itemize}
\end{theorem}
\begin{proof}
Let function $h$ be the solution of the differential equation $\frac{z h'(z)}{h(z)}=\frac{1+az}{1+bz}$, that is, $h(z)=z(1+bz)^{\frac{a-b}{b}}$, $z\in \D$. Thus, each $f\in\So^*(a,b)$ satisfies $\frac{f(z)}{z}\prec \frac{h(z)}{z}$ by~\cite[Theorem~1']{Ma-Mi}.
Since $\{1+bz : z\in\D\}$ is a disk of radius $|b|$ centered at $1$, the smallest sector containing the image of $\frac{h(z)}{z}$ has angle of opening $\frac{a-b}{|b|}\arcsin |b|$.
Hence, $B(\So^*(a,b))=\frac{2(a-b)}{\pi b}\arcsin b$.

Moreover, if $a\leq b+\frac{\pi}{2}\cdot\frac{b}{\arcsin b}$, then $\frac{a-b}{b}\arcsin b \leq \frac{\pi}{2}$. Therefore $\So^*(a,b)\subset \G_0$.

In the case $0<\frac ab\le2$,  we have $\Re \frac{f(z)}{z}\geq \inf\limits_{z \in \D} \Re \frac{h(z)}{z}=(1-|b|)^{\frac{a-b}{|b|}}$, which completes the proof.
\end{proof}

To link this theorem to filtrations, we note that by the definition $f\in\So(a,b)$, $-1\le b<a\le1,$ if and only if the range of $\frac{zf'(z)}{f(z)}$ is contained in the disk with diameter $\left( \frac{1-a}{1-b},\frac{1+a}{1+b} \right)$. Therefore, if we choose $a$ and $b$ be two continuous functions on $J\subset \R$, $a$ an increasing one and $b$ decreasing, then the family $\left\{\So^*(a(\al),b(\al))\right\}_{\al\in J}$ is a strict filtration of $\G_0$.

An interesting particular case of Janowski functions can be described as follows. Let $M\ge0$ and $\lambda\in\left(M,\sqrt{M^2+1}\right]$. Denote
\[
\So^*_M(\lambda)=\left\{ f\in\Ao:\ \left|\frac{zf'(z)}{f(z)}-M-1\right|<\lambda,\ z\in\D \right\}.
\]
In the following statement we combine results of Theorems 3 and 6 in \cite{E-S-Tu2}.
\begin{theorem}\label{c2-b}
Let $0\le M<\lambda\le\sqrt{M^2+1}$.  The following assertions hold:
\begin{itemize}
  \item [(i)] $\So^*_M(\lambda) \subset \Af_\alpha\subset\G_0$, where $\alpha\ge \alpha(M,\lambda):= 1- \frac{\cos\left(\frac{\pi}{2}\sqrt{\lambda^2-M^2} \right)}{\lambda - M\cos\left(\frac{\pi}{2}\sqrt{\lambda^2-M^2} \right)}$;
  \item [(ii)] $B(\So^*_M(\lambda)) <\sqrt{\lambda^2-M^2}$;
  \item [(iii)] If, in addition, $M\in\left[0,\frac12\right)$ and $\lambda<1-M$, then $\left| \frac{f(z)}z -1  \right|  <\frac{\lambda+M}{1-\lambda-M}$ for all $z\in\D.$ Consequently, $K(\So^*_M(\lambda))> \frac{1-2(\lambda+M)}{1-(\lambda+M)}$.
\end{itemize}
\end{theorem}

\begin{example}[see \cite{E-S-Tu2}]\label{ex-janoM}
Consider the function $f $ defined by
\[
f(z)=\left\{\begin{array}{lcc}
ze^{\lambda z} & \mbox{if} &  M=0, \vspace{1mm}\\
z\left(\frac\lambda{\lambda-Mz}\right)^{{(\lambda^2-M^2)}/{M}} & \mbox{if} & M>0
\end{array}\right.
\]
with $\lambda\in\left(M,\sqrt{M^2+1}\right]$.
Then  $f\in \So^*_M(\lambda)\subset\Af_{\alpha(M,\lambda)}$.

Additionally, if $M\in\left[0,\frac12\right)$ and $\lambda<1-M$, then
$
\left| \frac{f(z)}z -1  \right| <\frac{\lambda+M}{1-\lambda-M}.
$
\end{example}

\begin{remark}\label{remark-janoM}
Denote by $\lambda_0(M)$ the unique fixed point of the function
$$\psi(\lambda):=\left(M+1\right)\cos\left(\frac{\pi}{2}\sqrt{\lambda^2-M^2}
\right),$$ that is, the root of the equation
$
(M+1)\cos\left(\frac{\pi}{2}\sqrt{\lambda^2-M^2} \right) = \lambda
$
on the interval $\left(M ,\sqrt{M^2+1}\right]$.
In particular, for $M=0$ we have $\lambda_0(0)\approx0.5946$.

It is easy to verify that $0\le \alpha(M,\lambda)\le1$ if and only
if $\lambda_0(M)\le\lambda\le \sqrt{M^2+1}$.
Therefore $f\in\Af_0=\RA$ if  $\lambda\le \lambda_0(M)$.
\end{remark}

 Similarly to the above, we notice that if $M(\al)$ and $\lambda(\al)$ are two differentiable functions such that $|M'(\al)|<\lambda'(\al)$, then the family of the sets $\left\{\So^*_{M(\al)}(\lambda(\al))\right\}_{\al\in J}$ is a strict filtration.

\vspace{3mm}

Other interesting classes related in some way to both $\So^*$  and $\G_0$ are
\begin{equation*}
\Uf_\alpha:=\left\{f \in \Ao : \left|\frac{zf'(z)}{f(z)}\cdot\frac{z}{f(z)}-1\right|<\alpha\right\},\quad 0<\al\le1.
\end{equation*}
These classes were introduced by Aksentiev \cite{Ak58} and studied by many mathematicians (see, for example, \cite{OPW} and references therein). One also sets $\Uf_0=\{\Id\}$.

\begin{proposition}[see \cite{OPW}]\label{propo-obra}
For any $0\le \alpha\leq 1$ if $f\in\Uf_\alpha$, then $f$ is a univalent function and
\[
\frac{f(z)}{z} \prec \frac{1}{1+(1+\alpha)z+\alpha z^2}.
\]
\end{proposition}

It turns out that for a certain range of parameters these classes form a filtration of $\G_0$.

\begin{theorem}\label{thm-U-filtr}
The following assertion hold:
\begin{itemize}
  \item [(a)] The family of sets $\Uf=\{\Uf_\al\}_{\al\in[0,\frac{1}{3}]}$ is a strict filtration of $\G_0$;
  \item [(b)] Let $f_\al,\, \al\left(0,\frac13\right],$ be defined by
  \begin{equation}\label{extremal-U}
f_\alpha(z)=\frac{z}{1+(1+\alpha)z+\alpha z^2}.
\end{equation}
Then $f_\al\in\partial\Uf_\al$.

  \item [(c)] The net $\left\{f_\alpha\right\}_{\alpha\in\left(0,\frac13\right]}$ is a net of totally extremal functions for $\Uf$;

  \item [(d)] for every $\al\in \left[0,\frac{1}{3}\right]$ we have $K(\Uf_\al)=\frac{1-3\alpha}{2 \alpha^2-4\alpha +2}$ and $B(\Uf_\al)=1$.
\end{itemize}
\end{theorem}

\begin{proof}
By construction $\Uf_\al\subset\Uf_\beta$ as $0\le\al\le\beta\le1$. So, assuming that $\Uf_\al\subset\G_0$ (that will  be shown later), we conclude that $\Uf$ is a filtration.

Consider now functions $f_\alpha,\, \alpha\in \left(0,1\right],$ defined by~\eqref{extremal-U}.  A direct calculation gives
\[
\sup_{z\in\D} \left|\frac{zf_\al'(z)}{f_\al(z)}\cdot\frac{z}{f_\al(z)}-1\right|=\sup_{z\in\D}|-\alpha z^2|=\alpha.
\]
This means that $f_\al\in\partial\Uf_\al$. So, it proves assertions (a) and (b).

Assertion (c) follows from Proposition~\ref{propo-obra} since $f_\al\in\Uf_\al$. Moreover, this assertion implies that to prove assertion (d), we have to find $\inf\limits_{z\in\D} \Re\frac{f_\al(z)}{z}$ and $\sup\limits_{z\in\D}\left| \arg \frac{f_\al(z)}{z}\right|$.

We now calculate
\begin{eqnarray*}
\inf_{z\in\D} \Re\frac{f_\al(z)}{z}&=&  \inf_{z=re^{i\phi}\in\D} \Re \frac{1}{(1+re^{i \phi})(1+\alpha re^{i \phi})} =\inf_{x\in(-1,1]}g(x),
\end{eqnarray*}
where  $g(x):=\frac{2\al x-\al+1}{2(\al^2+2\al x+1)}$. Since function $g$ is increasing, one easily finds
\begin{equation*}
\inf_{z\in\D} \Re\frac{f_\al(z)}{z}  =\inf_{x\in(-1,1]}g(x) =\frac{1-3\alpha}{2\alpha^2-4\alpha  +2}.
\end{equation*}
Thus $f_\al \in\G_0$ if and only if $\alpha\le \frac{1}{3}.$ So, each  $f \in\Uf_\al, \, \al\in\left[0,\frac{1}{3}\right],$ generates the exponentially squeezing semigroup with squeezing ratio $\frac{1-3\alpha}{2 \alpha^2-4\alpha +2}$.

In addition,
\[
\sup\limits_{z\in\D}\left| \arg \frac{f_\al(z)}{z}\right|  =\sup\limits_{z\in\D}\left| \arg \frac{1}{(1+re^{i \phi})(1+\alpha re^{i \phi})} \right|  =\frac\pi2\,.
\]
The proof is complete.
\end{proof}
Assertion (d) and partially assertion (a) were recently proved in \cite{Gi-Ku}.

\vspace{3mm}


We complete this subsection with a filtration formed by prestarlike functions. The term `prestarlike function' was introduce by Ruscheweyh in \cite{Rus77} after the previous works \cite{Rus75} and \cite{Suff76}.

\begin{definition}[\cite{Rus77}]\label{def-prestar}
  One says that a function $f\in\Ao$ is prestarlike of order $\alpha\in[0,1)$~if
  \[
  f(z)* \frac{z}{(1-z)^{2-2\al}} \in \So^*(\al).
  \]
  We denote by $\Rf_\al$ the class of all prestarlike functions of order $\al$. In addition, we set
  $$
  \Rf_1= \left\{f\in\Ao:\ \Re \frac{f(z)}{z} >\frac12, \quad  z\in \D\setminus \{0\}\right\}.
  $$
\end{definition}

The following description of prestarlike functions follows from \cite{Rus77}.
\begin{proposition}
  Let $f\in\Ao$. Then $f\in\Rf_\al$ if and only if $\displaystyle \Re\left(\frac{f(z)* \frac{z^2}{(1-z)^{3-2\al}}} {f(z)* \frac{z}{(1-z)^{2-2\al}}} \right)>-\frac12\,,$ $z\in\D\setminus\{0\},$ and if and only if for any $\zeta\in\partial\D$ we have $\displaystyle f(z)* \frac{z(1-z\zeta)}{(1-z)^{3-2\al}} \neq0,$ $z\in\D\setminus\{0\}.$
\end{proposition}

The following theorem describes the sets of prestarlike functions as a filtration of infinitesimal generators.
\begin{theorem}\label{thm-filtr-prestar}
The following assertion hold:
\begin{itemize}
  \item [(a)] The family of sets $\Rf=\{\Rf_\al\}_{\al\in[0,1]}$ is a strict filtration of $\G_0$;
  \item [(b)] Let $f_\al,\ \al\in(0,1],$ be defined by $f_\al(z)=z+\frac{1}{2(2-\al)}z^2$, then $f_\al \in \partial\Rf_\alpha$;
  \item [(c)] $K(\Rf_\al)=\frac12$ and $B(\Rf_\al)=1$ for any $\al\in[0,1]$.
\end{itemize}

\end{theorem}
\begin{proof}
Let $0\le\beta<\alpha\le1$. The inclusion $\Rf_\beta\subseteq\Rf_\alpha$ was proved by Suffridge in \cite{Suff76} (see also \cite{Rus77}). Furthermore, $\Rf_1\subset\G_0$ by definition. Thus $\Rf$ is a filtration of $\G_0$.

Consider the functions $f_\al,\ \al\in(0,1],$ defined by $f_\al(z)=z+\frac{1}{2(2-\al)}z^2$. If $\al=1,$ then $\Re f_1(z)= \Re \left(z+\frac{z^2}{2} \right)>\frac12$ as $z\in\D$. So, $f_1\in\Rf_1$. If $\al<1$, then $f_{\al}(z)*\frac{z}{(1-z)^{2-2\al}}= z+\frac{1-\al}{2-\al}z^2$. It is easy to see that the last function belongs to $\So^*(\al)$, hence $f_\al\in\Rf_\al$.

We show now that $f_\al\notin\Rf_\beta$ for $\beta<\alpha$. Indeed, let denote
$$h(z):=f_{\al}(z)*\frac{z}{(1-z)^{2-2\beta}}= z+\frac{1-\beta}{2-\al}z^2.$$
Then the value of the function $\frac{zh'(z)}{h(z)}$ at the point $z=-1$ is $\frac{2\beta-\alpha}{1+\beta-\alpha}.$ This value is less than $\beta$ because $\beta<\alpha$. So, $h\notin \So^*(\beta)$ and $f_\al\not\in\Rf_\beta.$ Therefore $f_\al\in\partial\Rf_\al$ and the filtration is strict. This completes the proof of assertions (a) and~(b).

Further, remind that for any $\al\in[0,1]$ we have $K(\Rf_\al)\ge K(\Rf_1)=K(\Lf_1)=\frac12$ and $B(\Rf_\al)\le B(\Rf_1)=B(\Lf_1)=1$. On the other hand,
\[
\displaystyle \frac{z}{1-z}*\frac{z}{(1-z)^{2(1-\al)}} =\frac{z}{(1-z)^{2(1-\al)}} \in \So^*(\al).
\]
Hence $\frac{z}{1-z}\in\Rf_\al.$ Therefore Definition~\ref{def-K-B} implies that $K(\Rf_\al)\le \inf\limits_{z\in\D} \Re \frac1{1-z} =\frac12$ and $B(\Rf_\al)\ge \frac2\pi \sup\limits_{z\in\D} \left| \arg\frac1{1-z}\right|=1$. The proof is complete.
\end{proof}

\bigskip

 We now continue this section with a filtration defined by certain nonlinear restriction on the same linear expression that was used for the analytic filtration $\Af$ in Subsection~\ref{subsec-analytic}.

\subsection{Pseudo-analytic filtration}\label{subsect-pseu-anal}

In this subsection we follow \cite{E-J-S}.
Consider the classes of functions defined as follows:
\begin{equation}\label{A1-classes}
\Af^1_\alpha:=\left\{f\in \Ao: \left|\alpha\frac{f(z)}{z}+(1-\alpha)f'(z)-1\right|\leq 2-\alpha\right\}, \quad \alpha\leq 1.
\end{equation}

For each $\al\in(-\infty,1]$ the set $\Af_\al^1$ is non-empty. Moreover, it follows immediately from \eqref{A1-classes} and the Schwarz lemma that $f\in\Af_1^1$ if and only if it admits the representation $f(z)=z(1+\omega(z)),$ $\omega\in\Omega$. Our first result presents a one-to-one correspondence between the sets $\Af^1_\alpha$ and $\Omega$.
\begin{proposition}\label{th-A1-prop}
Let $\alpha<1$. For each $f \in \Af^1_\alpha$ there is the unique $\omega\in \Omega$ such that
\begin{equation}\label{sol-diff-eq}
f(z)=z\left(1+\frac{2-\alpha}{1-\alpha}\cdot \int_{0}^{1} s^{\frac{\alpha}{1-\alpha}}\omega(sz)ds\right)\!.
\end{equation}
Consequently, if $f \in \Af^1_\alpha$ then $\frac{f(z)}{z}-1\in \Omega.$

Conversely, for any $\omega\in\Omega$, the function $f$ defined by \eqref{sol-diff-eq} belongs to $\Af_\al^1$.
\end{proposition}

\begin{proof}
Let $f \in \Af^1_\alpha$.
Denote $g(z):=\alpha\frac{f(z)}{z}+(1-\alpha)f'(z)-1$. Since $g(0)=0$, the Schwarz Lemma implies that $|g(z)|\leq(2-\alpha)|z|$, which is equivalent to $g(z)\prec(2-\alpha)z$. Therefore there exists $\omega \in \Omega$ such that
\begin{equation}\label{diff-eq-ps}
\alpha\frac{f(z)}{z}+(1-\alpha)f'(z)-1=(2-\alpha)\omega(z).
\end{equation}
Solving this differential equation we get \eqref{sol-diff-eq}.

Therefore,
\begin{eqnarray*}\label{calc1}
\left|\frac{f(z)}{z}-1\right|&=&
\left|\frac{2-\alpha}{1-\alpha}\cdot\int_{0}^{1}s^{\frac{\alpha}{1-\alpha}}\omega(sz)ds\right|\leq
\frac{2-\alpha}{1-\alpha}\cdot\int_{0}^{1}s^{\frac{\alpha}{1-\alpha}}|\omega(sz)|ds\\
&\leq&\frac{2-\alpha}{1-\alpha}\cdot\int_{0}^{1}s^{\frac{\alpha}{1-\alpha}}ds =1.
\end{eqnarray*}

The converse assertion follows from the equivalence of \eqref{sol-diff-eq}  and~\eqref{diff-eq-ps}.
\end{proof}

\begin{corollary}
$\Af^1_\alpha\subseteq\Af^1_1\subsetneq \G_0$ for any $\alpha \leq 1.$
\end{corollary}

\begin{lemma}\label{exa-4.1}
 Let  $\alpha\leq 1$. For any $\theta \in \R$ and $n\in \N$,  the function
\begin{equation*}
f_{\theta,n,\alpha}(z)=z+\frac{2-\alpha}{1+n(1-\alpha)}e^{i\theta} z^{n+1}
\end{equation*}
belongs to the class $\Af^1_\alpha$. Moreover, $f_{\theta,n,\alpha}$ is not totally extremal for this set when $n\ge2$.
\end{lemma}
\begin{proof}
Indeed,
\begin{eqnarray*}\label{ex-bound-A1-calc}
&&\left|\alpha\frac{f_{\theta,n,\alpha}(z)}{z}+(1-\alpha)f'_{\theta,n,\alpha}(z)-1\right|\\
&=&\left|\alpha\left(1+\frac{(2-\alpha)e^{i\theta}z^{n}}{1+n(1-\alpha)}\right)+(1-\alpha)\left(1+\frac{(n+1)(2-\alpha)e^{i\theta}z^{n} }{1+n(1-\alpha)} \right)-1\right|\\
&=&|(2-\alpha)ze^{i\theta}|\leq 2-\alpha.
\end{eqnarray*}
So, $f_{\theta,n,\alpha}\in\Af^1_\alpha. $

Let $n\ge2$. Assume to the contrary that $f_*:=f_{\theta,n,\alpha}$ is totally extremal for $\Af^1_\alpha$. Then it follows from Definition~\ref{def-extrem} that for any $f\in \Af^1_\alpha$ inequality
\[
\min_{|z|=r} \Re \left(\lambda\frac {f(z)}{z}\right)\ge \min_{|z|=r}\Re\left(\lambda \frac {f_*(z)}{z}\right).
\]
holds. Let $f$ be represented by \eqref{sol-diff-eq} with $\omega(z)=z$. Then
\[
\min_{|z|=r} \Re \left(\lambda\cdot \frac{2-\al}{1-\al}\int_0^1 s^{\frac {\al}{1-\al}+1}zds\right)\ge \min_{|z|=r}\Re\left(\lambda \cdot \frac{2-\alpha}{1+n(1-\alpha)}e^{i\theta} z^n \right).
\]
Equivalently,
\[
-|\lambda|r  \geq - \frac{2-\alpha}{1+n(1-\alpha)}|\lambda|r^n .
\]
Since the last inequality  obviously fails for $r$ small enough, the contradiction completes the proof.
\end{proof}
Note in passing, that $f_{\theta}:=f_{\theta,1,\alpha}$, in fact, does not depend on $\alpha$.

 Now we show that the sets $\Af^1_\alpha$, $\al\le1$, form a filtration and establish its main properties. In the next theorem we appeal to the functions $f_{\theta,n,\alpha}$ defined in Lemma~\ref{exa-4.1}.

\begin{theorem}\label{th-pseudo-prop} The family $\Af^{1}=\{\Af^1_\alpha\}_{\alpha \leq 1 }$ is a filtration of $\G_0$. Moreover, for any $\alpha \leq 1$ the following assertions hold:
\begin{itemize}
  \item[(i)] Each function $f_\theta (z)=z+e^{i\theta}z^2$, $\theta\in\R$, belongs to the set $\Af^1_\alpha$ (for any $\al\le1$) and is totally extremal for it.
  \item [(ii)] Each function $f_{\theta,n,\alpha}$, $\theta \in \R$, $n\in \N\setminus\{1\}$, belongs to $\partial \Af^1_\alpha$. Hence the filtration $\Af^{1}$ is strict.
  \item [(iii)] $K(\Af^1_\alpha)=0$ and $B(\Af^1_\alpha)=1$.
\end{itemize}
\end{theorem}

\begin{proof}
Note that  $f \in \Af^1_\alpha$ if and only if for every fixed $z\in\D$ the value $f'(z)$ belongs to the disk of radius $\frac{2-\alpha}{1-\alpha}$ and centered at $\frac{1}{1-\alpha}\left(1-\alpha\frac{f(z)}{z}\right).$ Hence, in order to prove that $\Af^1_\alpha\subseteq \Af^1_\beta$ whenever $\alpha\leq\beta\leq 1$, we need to ensure that such disks are included one into another. To this end we denote $w=\frac{1}{1-\alpha}\left(1-\alpha\frac{f(z)}{z}\right)+e^{i\theta}\frac{2-\alpha}{1-\alpha}$ and show that $\left|w-\frac{1}{1-\beta}\left(1-\beta\frac{f(z)}{z}\right)\right|\leq \frac{2-\beta}{1-\beta}$.
Indeed,
\begin{eqnarray*}\label{calc2}
&&\left|\frac{1}{1-\alpha}\left(1-\alpha\frac{f(z)}{z}\right)+e^{i\theta}\frac{2-\alpha}{1-\alpha}-\frac{1}{1-\beta}\left(1-\beta\frac{f(z)}{z}\right)\right|=\\
&=& \left|\left(\frac{f(z)}{z}-1\right)\frac{\alpha-\beta}{(1-\alpha)(1-\beta)}+e^{i\theta}\frac{2-\alpha}{1-\alpha}\right|\\
&\leq& \left|\frac{f(z)}{z}-1\right|\cdot\frac{\beta-\alpha}{(1-\alpha)(1-\beta)}+\frac{2-\alpha}{1-\alpha}\,.
\end{eqnarray*}
By Proposition~\ref{th-A1-prop},
\begin{equation*}\label{calc3}
\left|\frac{f(z)}{z}-1\right|\cdot\frac{\beta-\alpha}{(1-\alpha)(1-\beta)}+\frac{2-\alpha}{1-\alpha}\leq \frac{\beta-\alpha}{(1-\alpha)(1-\beta)}+\frac{2-\alpha}{1-\alpha}=\frac{2-\beta}{1-\beta}\,.
\end{equation*}
So, the family $\Af^{1}$ is a filtration of $\G_0$.

It follows from Proposition~\ref{th-A1-prop} that for each $\lambda=e^{i\phi}$ we have
\begin{eqnarray*}
\Re \left(\lambda\frac{f(z)}{z}\right)&=&\cos\phi + \frac{2-\alpha}{1-\alpha}\cdot \int_{0}^{1} s^{\frac{\alpha}{1-\alpha}}\Re \left(e^{i\phi}\omega(sz)\right)ds\\
  &\geq& \cos\phi -\frac{2-\alpha}{1-\alpha}\cdot \int_{0}^{1} s^{\frac{\alpha}{1-\alpha}}s|z|ds= \cos\phi - |z|.
\end{eqnarray*}

At the same time, $$\Re \left(\lambda\displaystyle\frac{f_\theta(z)}{z}\right)=\cos\phi+\Re \left( e^{i(\phi+\theta)}z\right)=\cos\phi - |z|$$ whenever $\arg z=-(\phi+\theta).$
This proves assertion (i).

Assume now that $\alpha<\beta\leq 1$ and calculate
\begin{eqnarray*}
&&\left|\alpha\frac{f_{\theta,n,\beta}(z)}{z}+(1-\alpha)f'_{\theta,n,\beta}(z)-1\right|\\
&=&\left|\alpha\left(1+\frac{(2-\beta)e^{i\theta}z^{n}}{1+n(1-\beta)}\right)+(1-\alpha)\left(1+\frac{(n+1)(2-\beta)e^{i\theta}z^{n} }{1+n(1-\beta)} \right)-1\right|\\
&=&(2-\beta)\frac{1+n(1-\alpha)}{1+n(1-\beta)}|z|^n.\vspace{-1mm}
\end{eqnarray*}
Since $(2-\beta)\frac{1+n(1-\alpha)}{1+n(1-\beta)}>2-\alpha$ for every natural $n>1$, then there exists $z \in \D$ such that $\left|\alpha\frac{f_{\theta,n,\beta}(z)}{z}+(1-\alpha)f'_{\theta,n,\beta}(z)-1\right|>2-\alpha$. Thus, $f_{\theta,n,\beta}\in \Af^1_\beta\setminus \Af^1_\alpha$, which implies $f_{\theta,n,\beta}\in \partial\Af^1_\beta$. Then assertion (ii) follows.

Since $f_{\theta,1,\alpha}(z)=z+e^{i\theta}z^2 \in \Af^1_\alpha$ for every $\alpha\leq 1$ by Lemma~\ref{exa-4.1}, one concludes
\begin{equation*}
\inf\limits_{z\in \R} \Re \frac{f_{\theta,1,\alpha}(z)}{z}=\inf\limits_{z\in \R}\Re(1+e^{i\theta})=0
\end{equation*}
and
\begin{equation*}
\sup\limits_{z\in \R} \arg \frac{f_{\theta,1,\alpha}(z)}{z}=\sup\limits_{z\in \R}\left|\arg(1+e^{i\theta})\right|=\frac{\pi}{2}.
\end{equation*}
Thus assertion (iii) holds. The proof is complete.
\end{proof}

\begin{definition}
We denote $\Af^1 =\left\{\Af^1_\alpha \right\}_{\alpha \in(-\infty,1]}$ and call it the {\it pseudo-analytic filtration}.
\end{definition}

\begin{rem}\label{rem-exstr-bou}
We have proved that function $f_0$ is totally extremal for {\it every} $\Af^1_\alpha$. Hence it is not a boundary point. At the same time, functions $f_{\theta,n,\alpha}$, $n\ge2$, belong to the boundary $\partial \Af^1_\alpha$, $\al\le1$, and are not totally extremal for $\Af^1_\alpha$ by Lemma~\ref{exa-4.1}.
\end{rem}

Although the definitions of analytic and pseudo-analytic filtrations are similar, we have already seen that their properties are different. In addition, they have different smallest sets: $\Af_{-\infty}=\{\Id\}$, while $\Af^1_{-\infty}$ is not a singleton.
\begin{corollary}
The set  $\Af^1_{-\infty}:=\bigcap_{\alpha<1}\Af^1_\alpha=\{f\in\Ao: |f(z)-zf'(z)|\leq |z|\}$ contains all convex combinations of functions $\left\{z+\frac{1}{n}e^{i\theta}z^{n+1},\, n\in \N , \, \theta \in \R\right\}$.
\end{corollary}

\bigskip

Subsections~\ref{sSect-Qstar} and~\ref{subsec-mocanu1} are mainly based on the work \cite{EJT}. Some of the results are new.

\subsection{Quasi-starlike filtration}\label{sSect-Qstar}

Another  `simple' non-linear expression is a linear combination of inequalities that determine the class $\G_0$ and the class of starlike functions. More precisely, for a given $\alpha\in\R$ we consider the class of functions that satisfies
\begin{equation}\label{alm-star}
\displaystyle \Re\left[ \alpha\, \frac{f(z)}{z}+(1-\alpha)\, \frac{zf'(z)}{f(z)}\right]>\frac12, \quad  z\in \D\setminus \{0\}.
\end{equation}
We denote by $\Lf_\al$ the family of all $f\in\Ao$ that satisfy \eqref{alm-star}.
It is not clear whether elements of the class $\Lf_\al$ are starlike functions for arbitrary $\alpha$. Therefore the following remark is of certain interest.
\begin{rem}\label{rem-alm-star1}
It follows from Marx and Strohh\"{a}cker's Theorem~\ref{thm-M-S} that
\[
\Lf_0=\mathcal{S}^*\left(\frac 12\right)\subset \Lf_\alpha\quad\mbox{for any}\quad \alpha\in[0,1] .
\]
\end{rem}
 In general, we cannot state that the whole family $\left\{ \Lf_\al\right\}_{\al\in  \R}$, or even $\left\{ \Lf_\al\right\}_{\al\ge0}$ is a filtration, while we prove that $\{\Lf_\alpha\}_{\al\in[0,1]}$ is.
\begin{theorem}\label{filtration-alpha}
The following assertions hold:
\begin{itemize}
  \item [(a)] The family of sets $\{\Lf_\alpha\}_{\al\in[0,1]}$ is a filtration of $\G_0$ such that ${K(\Af_\al)\!=\!\frac12}$ and ${B(\Af_\al)\!=\!0}$ for any $\al\in[0,1]$.
  \item [(b)] For any $a>1$ the family of sets $\{\Lf_\alpha\}_{\al\in[0,a]}$ is not a filtration, although $\Lf_\alpha\subset \G_0$ for any $\alpha\ge 0$.
\end{itemize}
\end{theorem}

\begin{proof}
Clearly, $\Lf_1\subset\G_0$. Assume that $0\le\alpha <\beta\le1$ and $f\in\Lf_\alpha$.

We have to show that $f\in \Lf_\beta$. For this aim, define the function $\omega_\alpha\in\Omega$  such that
\begin{equation}\label{w-al-simply3}
 \omega_\alpha(z)\cdot \left[{\alpha \frac{f(z)}{z}+(1-\alpha)\frac{z f'(z)}{f(z)}}\right]= {\alpha \cdot \frac{f(z)}{z}+(1-\alpha)\cdot \frac{z f'(z)}{f(z)}}-1 \, .
\end{equation}
In addition, define $\omega_\beta(z)$,  replacing $\alpha$ by $\beta$ in \eqref{w-al-simply3}. We have to show that $\omega_\beta\in\Omega$ as well.

To this end denote $h(z)=\frac{f(z)}{z}-\frac{z f'(z)}{f(z)}$. Equality~\eqref{w-al-simply3} implies
\begin{equation*}\label{h(z)-al}
-h(z)=\frac{-1+[1-\omega_{\alpha}(z)]\cdot \frac{zf'(z)}{f(z)}}{\alpha\cdot [1-\omega_{\alpha}(z)]}=-\frac{1}{\alpha\cdot [1-\omega_{\alpha}(z)]}+\frac{1}{\alpha}\cdot \frac{zf'(z)}{f(z)}\,.
\end{equation*}
 By definition of function $h$, we have
\begin{equation*}\label{heqh-al}
-\frac{1}{\alpha\cdot [1-\omega_{\alpha}(z)]}+\frac{1}{\alpha}\cdot\frac{zf'(z)}{f(z)}=-\frac{1}{\beta\cdot [1-\omega_\beta(z)]}+\frac{1}\beta\cdot\frac{zf'(z)}{f(z)}\,,
\end{equation*}
or, equivalently,
\begin{equation}\label{heqh1-al}
-\left(\frac{1}{\alpha}-\frac{1}{\beta}\right)\cdot \frac{zf'(z)}{f(z)}+\frac{1}{\alpha}\cdot\frac{1}{1-\omega_\alpha(z)}=\frac{1}{\beta}\cdot\frac{1}{1-\omega_{\beta}(z)}. \end{equation}

Obviously $\omega_\beta(0)=0$. Assume  to the contrary that $\omega_\beta $ is not a self-mapping of the unit disk. Then there exists a point $z_0 \in \D$ such that $|\omega_\beta(z)|<1$ for all $|z|<|z_0|$ while $|\omega_\beta(z_0)|=1.$
Substitute $z=z_0$ in the right-hand side of \eqref{heqh1-al} and get
\[
\frac{1}{\beta}\cdot\Re \frac{1}{1-\omega_{\beta}(z_0)}=\frac{1}{2\beta }\,.
\]
Thus the left-hand side of \eqref{heqh1-al} is
\begin{eqnarray*}\label{heqh3-al}
&&\!-\left(\frac{1}{\alpha}-\frac{1}{\beta}\right)\cdot \Re \frac{zf'(z)}{f(z)}+ \frac{1}{\alpha}\cdot \Re\frac{1}{1-\omega_\alpha(z)}>\!-\left(\frac{1}{\alpha}-\frac{1}{\beta}\right)\cdot \Re \frac{zf'(z)}{f(z)}+\frac{1}{2\alpha}. 
\end{eqnarray*}
Thus \eqref{heqh1-al} implies that
\begin{equation*}
\frac{1}{2\beta}>-\left(\frac{1}{\alpha}-\frac{1}{\beta}\right)\cdot \Re \frac{zf'(z)}{f(z)}+\frac{1}{2\alpha}, \qquad  z\in\D.
\end{equation*}
In its turn, this means
\begin{equation*}
 \Re \frac{zf'(z)}{f(z)}>\frac{1}{2}, \qquad  z\in\D,
\end{equation*}
that is, $f\in \mathcal{S}^*(\frac{1}{2})$. Hence $f\in\Lf_\beta$ according to Remark~\ref{rem-alm-star1}. This contradiction proves that  $\{\Lf_\alpha\}_{\al\in[0,1]}$ is a filtration.

It follows from the definition that $\Lf_1=\Cf_{\frac12}$ (see Section~\ref{ssect-squeezing-filt}). Therefore, $K(\Lf_1)=\frac12$ and $K(\Lf_\al)\ge\frac12$. Consider the function $f(z)=\frac{z}{1-z}$. Being a starlike function of order $\frac12$, it belongs to $\Lf_0$. A direct calculation shows that $f\not\in\Cf_\al$ with $\al<\frac12.$ Thus  $K(\Lf_0)=\frac12$. The same function shows that $B(\Lf_0)=0$. This completes the proof of assertion (a).

We now show that for every $\al>1$ there is a function of the class $\Lf_1$ that does not belong to $\Lf_\al$. Indeed, let $n$ be a natural number greater than $\frac\al{\al-1}$ and consider the function $f(z):=\displaystyle\frac{z}{1-z^n}.$ Clearly, $f\in\Lf_1$. At the same time, it satisfies
\[
\alpha\, \frac{f(z)}{z}+(1-\alpha)\, \frac{zf'(z)}{f(z)}=\frac{1+(1-\al)(n-1)z^n} {1-z^n}.
\]
This expression takes negative values for $z$ close enough to $1$, so $f\notin\Lf_\al$.

Taking in mind assertion (a), it is enough to show that $\Lf_\al\subset\G_0$ for $\al>1.$ Let $f\in\Lf_\al\setminus\G_0$. Consider the function $\omega $ defined by $\omega(z)=\frac{f(z)-z}{f(z)+z}$. For a fixed $z \in \D$, the value $\omega(z)$  lies in $\D$ if and only if $\Re\frac{f(z)}{z}>0$. According to our assumption, there is $z_0\in\D$ such that $\Re\frac{f(z)}{z}<0$ as $|z|<|z_0|$ while $\Re\frac{f(z_0)}{z_0}=0$ and, consequently, $|\omega_0|=1$, where $\omega_0:=\omega(z_0)$. By Lemma~\ref{lem-Jack} there is $k\ge1$ such that $z_0\omega'(z_0)=k\omega_0$. A straightforward calculation gives
\[
  \alpha\, \frac{f(z_0)}{z_0}+(1-\alpha)\, \frac{z_0f'(z_0)}{f(z_0)} = \alpha\frac{1+\omega_0}{1-\omega_0} +(1-\alpha)\cdot \left( 1+\frac{k\omega_0}{1+\omega_0}+\frac{k\omega_0}{1-\omega_0} \right).
\]
The real part of the last expression is $1-\alpha<0<\frac12$, hence $f\notin\Lf_\alpha$. Contradiction.
\end{proof}

Since $\Lf_1=\left\{f:\, f(z)=zp(z),\, \Re p(z)>\frac12,\, z\in\D\right\}$ this theorem together with Theorem~\ref{thm-boud-p} immediately imply
\begin{corollary}
Every element $f \in \Lf_\alpha$, $\alpha \in [0,1] $, generates the semigroup with no boundary repelling fixed point.
\end{corollary}

\begin{definition}
The filtration $\{\Lf_\alpha\}_{\al\in[0,1]}$  is called {\it quasi-starlike}.
\end{definition}

Although we do not know whether this filtration is strict, we prove the strictness for a restricted range of $\alpha$.
\begin{theorem}[cf. Theorem 5.2 \cite{EJT}]\label{M-alp-1-alp}
For every $\al\in\left(\frac12,1\right]$, the boundary $\partial\Lf_\al$ is not empty. So, the filtration $\left\{\Lf_\alpha\right\}_{\al\in[\frac12,1]}$ is strict.
\end{theorem}

\begin{proof}
It follows from Theorem~\ref{filtration-alpha} that the family $\left\{\Lf_\alpha\right\}_{\alpha\in[\frac12,1]}$ is a filtration of $\G_0$.

To construct an element $f\in\partial\Lf_\al,\ \al\in\left(\frac12,1\right)$, let choose $\beta=\frac{\alpha}{1-\alpha},\ \gamma=0,\ c=1$ and $h(z)=1+\frac1\alpha \cdot \frac{z^2}{1-z^2}$ in \cite[Theorem~1]{MiMo-85}.   Then
  \[
\Re\left[\beta h(z)+\gamma \right]=\frac{\alpha}{1-\alpha} +\frac{1}{1-\alpha}\cdot \frac{z^2}{1-z^2}> \frac{\alpha-\frac12}{1-\alpha}>0.
\]
 Therefore, the mentioned Theorem~1 \cite{MiMo-85} can be applied. Then the solution $q$ of the corresponding (Briot--Bouquet) differential equation
 \[
 q(z)+\frac{zq'(z)}{\beta q(z)+\gamma} = h(z), \quad h(0)=q(0),
\]
exists and belongs to $\Hol$. By construction, the function $f(z)=zq(z)$ satisfies the differential equation
  \begin{equation}\label{dif-eq}
     \alpha\, \frac{f(z)}{z}+(1-\alpha)\, \frac{z f'(z)}{f(z)}= \frac{1}{1-z^2}
  \end{equation}
  and hence belongs to $\Af_\alpha$. Further, using \eqref{dif-eq} we calculate the first Taylor coefficients of $f$ and get $a_2=0$ and $a_3=\frac1{2-\alpha}.$ Thus, $a_3-a_2^2=\frac1{2-\al}.$

 On the other hand, it follows from Theorem~\ref{th-FS-estim} below with $\beta=1-\al$ that for every function from  $\Lf_\al$ with the Taylor coefficients $\{a_n\}$ the inequality $\left|a_3 -a_2^2\right| \le \frac{1}{2-\alpha}$ holds. Therefore the function $f$ we constructed cannot belongs to $\Lf_{\al'}$ with $\al'<\al$.

Hence $f\in\partial\Lf_\alpha$. The fact that $\partial\Lf_1\not=\emptyset$ follows from \cite{EJ-est}. The proof is complete.
\end{proof}

\subsection{Mocanu's filtration}\label{subsec-mocanu1}

In this subsection we consider classes of functions defined using a non-linear expression of {\it second order}. Fix $\beta\in\R$ and consider functions for which
\begin{equation}\label{Moc}
\Re\left[\beta\frac{zf'(z)}{f(z)}+(1-\beta)\left(1+\frac{zf''(z)}{f'(z)}\right)\right] >\frac\beta2\,, \ \quad  z\in\D.
\end{equation}
The class of functions that satisfy $\Re\left[(1-\alpha)\frac{zf'(z)}{f(z)}+\alpha\left(1+\frac{zf''(z)}{f'(z)}\right)\right] >0$  was introduced by Mocanu \cite{M-69}. After him such functions are called Mocanu's $\alpha$-convex functions. Later on, this class was studied by many mathematicians (see for example, references in \cite{EJT}). In particular, it was proved in \cite{MMR} that $\alpha$-convex functions are starlike; Fukui \cite{Fu} replaced zero in the right-hand side by a real number. In his terminology functions that satisfy \eqref{Moc} can be called $(1-\beta)$-convex functions of order~$\frac\beta2$.

\begin{definition}\label{def-M}
We say that $f\in\Ao$ belongs to $\Mf_\beta$ if $f(z)f'(z)\neq 0$ for all $z \in \D \setminus \{0\}$ and $f$ satisfies  \eqref{Moc}.
\end{definition}

Note that $\Mf_0=\mathcal{C}$ and $\Mf_1=\mathcal{S}^*\left(\frac12 \right)$.  We start the study of the classes $\Mf_\beta$ by establishing a representation for their elements.
\begin{proposition}[see Lemma 3.1 in \cite{EJT}]\label{prop-int-Moc}
Let $\beta<2$ and $\beta \neq 1$. Then, $f\in\Mf_\beta$ if and only if the function $g$ defined by
\begin{equation*}
g(z)=z \left(f'(z)\right)^{\frac{2-2\beta}{2-\beta}} \left[\frac{f(z)}{z}\right]^{\frac{2\beta}{2-\beta}}
\end{equation*}
is starlike. In this case
\begin{equation*}
f(z)= z\left\{\frac{1}{1-\beta}\int_0^1 t^{\frac{\beta}{1-\beta}} \left[\frac{g(tz)}{tz} \right]^{\frac{2-\beta}{2(1-\beta)}}dt \right\}^{1-\beta}.
\end{equation*}
\end{proposition}

As in the case of the classes $\Lf_\alpha$, we do not know whether the whole family $\left\{ \Mf_\beta\right\}_{\beta\in\R}$ is a filtration. The answer is affirmative for a restricted range of the parameter $\beta$.

\begin{theorem}\label{thm-filtr-beta-new}
The following assertions hold:
\begin{itemize}
  \item [(i)] the family of sets $\{\Mf_\beta\}_{\beta\in(-\infty,1]}$ is a filtration of $\G_0$;
  \item [(ii)] the filtration $\Mf=\left\{\Mf_\beta\right\} _{\beta\in[0,1]}$ is strict.
\end{itemize}
\end{theorem}
This result is an immediate consequence of the subsequent three statements.

\vspace{2mm}

{\sl\bf Statement 1.}  {\it $\Mf_\beta\subseteq  \mathcal{S}^*\left(\frac12\right)\subsetneq\G_0$ for any $\beta \leq 1$.}
\vspace{1mm}

To explain this statement, let us recall that every  $\beta$-convex function is starlike. Since in our settings $\beta$ might be negative, such conclusion for elements of the class $\Mf_\beta$ is not implied directly. It follows from \cite[Theorem 1]{Fu} that for $\beta<1$ elements of the class $\Mf_\beta$ are starlike functions of order $\frac 12.$ For a complete proof see \cite[Lemma 4.2]{EJT}.

\vspace{3mm}

The next statement shows that  $\{\Mf_\beta\}_{\beta \in (-\infty,1]}$ is a filtration.
\vspace{1mm}

{\sl\bf Statement 2.}  {\it If $\alpha <\beta\le1$, then $\Mf_\alpha\subset \Mf_\beta$.}
\begin{proof}
Let $\alpha <\beta\le1$ and $f\in\Mf_\alpha$. We have to show that $f\in \Mf_\beta$.
Denote $g_\al(z):=\alpha\frac{zf'(z)}{f(z)}+(1-\alpha)\left(1+\frac{zf''(z)}{f'(z)}\right)$. Then $\Re g_\al(z)>\al/2$, $z\in\D$, that is, there exists a function $\omega_\al \in \Omega$ such that
\[
g_\al(z) = \left(1-\frac{\al}{2}\right)\cdot \frac{1+\omega_{\al}(z)}{1-\omega_{\al}(z)}+\frac{\al}{2}\,,
\]
and hence
\begin{equation}\label{w1}
\omega_\al(z)=\frac{g_\al(z)-\frac{\al}{2}-1+\frac{\al}{2}}{g_\al(z)-\frac{\al}{2}+1-\frac{\al}{2}} \, .
\end{equation}
Now, let consider function $\omega_{\beta}$ defined in a similar way by
\begin{equation*}\label{w}
 \omega_{\beta}(z)=\frac{g_\beta(z)-\frac{\beta}{2} -1+\frac{\beta}{2}} {g_\beta(z)- \frac{\beta}{2}+1- \frac{\beta}{2}} \, .
 \end{equation*}
It is analytic in $\D$ and vanishes at the origin. So, in order to establish that $f\in \Mf_\beta$, it is enough to show that $|\omega_\beta(z)|<1$ for all $z\in\D$.
To this end, rewrite \eqref{w1} as
\begin{equation*}\label{w-simply}
\omega_{{\al}}(z)=\frac{g_{\al}(z)-1}{g_{\al}(z)+1-{{\al}}}=\frac{{{\al}}\frac{z f'(z)}{f(z)}+(1-{{\al}})\left[1+\frac{zf''(z)}{f'(z)}\right]-1}{{{\al}}\frac{z f'(z)}{f(z)}+(1-{{\al}})\left[1+\frac{zf''(z)}{f'(z)}\right]+1-{{\al}}},
\end{equation*}
and then
\begin{eqnarray}\label{w-simply3}
\nonumber \omega_{{\al}}(z)\left\{(1-{{\al}})\left[1+\frac{zf''(z)}{f'(z)}-\frac{z f'(z)}{f(z)}\right]+  \frac{z f'(z)}{f(z)}+1-{{\al}}\right\}\\
= (1-{{\al}})\left[1+\frac{zf''(z)}{f'(z)}-\frac{z f'(z)}{f(z)}\right]+  \frac{z f'(z)}{f(z)}-1.
\end{eqnarray}
Similarly, we have
\begin{eqnarray}\label{w-simply4}
\nonumber \omega_\beta(z) \left\{(1-\beta)\left[1+\frac{zf''(z)}{f'(z)}-\frac{z f'(z)}{f(z)}\right]+  \frac{z f'(z)}{f(z)}+1- \beta\right\}\\
= (1-\beta)\left[1+\frac{zf''(z)}{f'(z)}-\frac{z f'(z)}{f(z)}\right]+  \frac{z f'(z)}{f(z)}-1 \ .
\end{eqnarray}
Further, let us define the function $h$ by $h(z)=1+\frac{zf''(z)}{f'(z)}-\frac{z f'(z)}{f(z)}$. Equation~\eqref{w-simply3} implies
\begin{eqnarray*}\label{h(z)}
h(z)&=&\frac{(1-\al)\cdot \omega_{\al}(z)+[\omega_{\al}(z)-1]\cdot \frac{zf'(z)}{f(z)}+1}{(1-\al)\cdot [1-\omega_{\al}(z)]}\\
&=&\frac{(1-\al)\cdot [\omega_{\al}(z)-1]+[\omega_{\al}(z)-1]\cdot \frac{zf'(z)}{f(z)}+2-\al}{(1-\al)\cdot [1-\omega_{\al}(z)]}\\
&=&-1-\frac{1}{(1-\al)}\cdot \frac{zf'(z)}{f(z)}+\frac{2-\al}{(1-\al)\cdot [1-\omega_{\al}(z)]}\,,
\end{eqnarray*}
while \eqref{w-simply4} leads  to
\[
h(z)=-1-\frac{1}{(1- \beta)}\cdot \frac{zf'(z)}{f(z)}+\frac{2- \beta}{(1-\beta)\cdot [1-\omega_\beta(z)]}\,.
\]
Therefore, 
\begin{equation*}\label{heqh}
\begin{split}
&\quad \!-\frac{1}{1-\al}\cdot\frac{zf'(z)}{f(z)}+\frac{2-\al}{(1-\al)\cdot [1-\omega_{\al}(z)]}\\
&=-\frac{1}{1-\beta}\cdot\frac{zf'(z)}{f(z)}+\frac{2-\beta}{(1-\beta)\cdot [1-\omega_{\beta}(z)]},
\end{split}
\end{equation*}
or, equivalently,
\begin{equation}\label{heqh1}
\left(\frac{1}{1-\beta}-\frac{1}{1-\al}\right) \cdot \frac{zf'(z)}{f(z)}+ \frac{2-\al}{1-\al}\cdot \frac{1}{1-\omega_\al(z)}= \frac{2-\beta}{1-\beta}\cdot\frac{1}{1-\omega_\beta(z)} \,.
\end{equation}

Now, assume on contrary that $\omega_\beta$ is not a self-mapping of the unit disk.  Then there exists a point $z_0 \in \D$ such that $|\omega_\beta(z)|<1$ for all $|z|<|z_0|$ and $|\omega_\beta(z_0)|=1.$ Let substitute $z=z_0$ in the right-hand side of \eqref{heqh1}, then we get
\[
\begin{split}
&\quad \left(\frac{1}{1-\beta}-\frac{1}{1-\al}\right)\cdot \Re\frac{z_0f'(z_0)}{f(z_0)}+\frac{2-\al}{1-\al}\cdot \Re\frac{1}{1-\omega_\al(z_0)} \\
&=\Re \left[\frac{2-\beta}{(1-\beta)}\cdot \frac{1}{1-\omega_\beta(z_0)}\right]=\frac{1}{2}\cdot\frac{2-\beta}{1-\beta}.
\end{split}
\]
At the same time, by the Statement~1, the expression in the left-hand side of \eqref{heqh1} is
\begin{eqnarray*}\label{heqh3}
&&\left(\frac{1}{1-\beta}-\frac{1}{1-\al}\right)\cdot \Re\frac{z_0f'(z_0)}{f(z_0)}+ \frac{2-\al}{1-\al}\cdot \Re\frac{1}{1-\omega_\al(z_0)}>\\
&&\frac{1}{2}\left(\frac{1}{1-\beta}-\frac{1}{1-\al} +\frac{2-\al}{1-\al}\right)=\frac{1}{2}\cdot \frac{2-\beta}{1-\beta},
\end{eqnarray*}
which contradicts our assumption. So, $\{\Mf_\beta\}_{\al\in(-\infty,1]}$ is a filtration.
\end{proof}

\vspace{1mm}

We do not know whether the filtration $\{\Mf_\beta\}_{\beta\in(-\infty,1]}$ is strict. We prove strictness for $\beta$ restricted to interval $(0,1]$.

\vspace{2mm}
{\sl\bf Statement 3.}   For every $\beta\in(0,1]$,  the boundary $\partial\Mf_\beta$ is not empty.
\begin{proof}
Similarly to the proof of Theorem~\ref{M-alp-1-alp}, let consider the (Briot--Bouquet) differential equation
\begin{equation}\label{dif-eq2}
  q(z)+(1-\beta)\cdot \frac{zq'(z)}{q(z)} = \frac{\beta}{2} +\left( 1- \frac{\beta}{2} \right) \cdot \frac{1+z^2}{1-z^2} \,.
\end{equation}
 It follows from \cite[Theorem~1]{MiMo-85} that this equation has a holomorphic in $\D$ solution $q$.

 We now restore the function $f\in\Ao$ solving $\frac{zf'(z)}{f(z)} =q(z)$. By construction, it belongs to the class $\Mf_\beta$.  Equality \eqref{dif-eq2} enables us to calculate the Taylor coefficients of $f$: $a_2=0$ and $a_3=\frac{2-\beta}{6-4\beta}.$ Thus $a_3-a_2^2=\frac{2-\beta}{6-4\beta}.$

At the same time, according to Theorem~\ref{th-FS-estim} applied to the case $\al=0$ we have that for every function from  $\Mf_\beta$ with the Taylor coefficients $\{a_n\}$ the inequality $|a_3-a_2^2|\le \frac{2-\beta}{6-4\beta}$ holds. Therefore the above constructed function $f$ cannot belong to $\Mf_{\beta'}$ as $\beta'<\beta$.  Hence $f\in\partial\Mf_\beta$. This completes the proof.
\end{proof}

As we have already mentioned, for $\beta<1$ every $(1-\beta)$-convex function is starlike (see \cite{MMR}). Since the classes $\Mf_\beta$ for $\beta<0$ are wider than the class of $(1-\beta)$-convex functions, the result in~\cite{MMR} does not imply that elements of $\Mf_\beta$ are starlike. On the other hand, $\Mf_0=\mathcal{C}$. Hence Theorem~\ref{thm-filtr-beta-new} implies
\begin{corollary}
  For any $\beta<0$ the class $\Mf_\beta$ consists of convex functions. Furthermore, the set $\bigcap_{\beta<0}\Mf_\beta=\left\{ f\in\Ao: \,\Re\left[\frac{zf'(z)}{f(z)} - \left(1+\frac{zf''(z)}{f'(z)}  \right)  \right]<\frac12 \right\}$ is not a singleton.
\end{corollary}

Note that $\Mf_1=\So^*(\frac12),$ and hence Marx--Strohh\"acker's Theorem~\ref{thm-M-S} implies that  $K(\Mf_1)=\frac12$. In addition, it follows from formula \eqref{Moc} (or Proposition~\ref{prop-int-Moc}) that $f\in\Mf_\beta,\ \beta<2,$ where $f(z)=\frac{z}{1-z}$.
Therefore we conclude
\begin{corollary}
\begin{itemize}
  \item [(i)] $K(\Mf_\beta)=\frac12$ and $B(\Mf_\beta)=1$ for every $\beta\le1$;
  \item [(ii)] Every element $f \in \Mf_\beta$, $\beta<1$, generates the semigroup with no boundary repelling fixed point.
\end{itemize}
\end{corollary}

\bigskip

\subsection{Open questions }\label{ssect-quest2}

\setcounter{question}{0}

We start with a question raised in \cite{ESS} including the explanation presented there.

As we have already mentioned, the Noshiro--Warschawski class $\RA$ is a subclass of $\G_0$, hence for any $z\in\D$ both the points $w_0=f'(z)$ and $w_1=\frac{f(z)}z$ lie in the right half-plane. In the definition of analytic filtration $\Af,$ we considered the Euclidean convex combination $\alpha w_1+(1-\alpha)w_0$, see \eqref{N-2}. Since the values $w_0$ and $w_1$ are usually non-vanishing, one may apply the quasihyperbolic geometry on the punctured plane $\C^*=\C\setminus\{0\}.$ Then the quasihyperbolic geodesic of $w_0$ and $w_1$ can be expressed by
$$
w_0^{1-\alpha}w_1^\alpha=w_0\left(\frac{w_1}{w_0}\right)^\alpha
=f'(z)\left(\frac{f(z)}{zf'(z)}\right)^\alpha
$$
(see \cite{MO86} for example). Note that one can choose a single-valued analytic branch of $h(z)=[zf'(z)/f(z)]^{-\alpha}$ on $\D$ with $h(0)=1$ whenever $f$ is non-vanishing except for $z=0$ and locally univalent; namely, $f(z)\ne0$ and $f'(z)\ne0$ for $0<|z|<1.$ Consider the class $\mathfrak{F}_\alpha$ consisting of such functions $f\in\G_0$ that satisfy
$$
\Re\left[f'(z)\left(\frac{f(z)}{zf'(z)}\right)^\alpha\right]>0,\quad z\in\D.
$$
We have now the following fundamental problems.

\begin{question}
Is the family $\{\mathfrak{F}_\alpha\}_{\alpha\in J}$ a filtration of $\G_0$? If the answer is affirmative, is this filtration strict? \end{question}

\begin{question}
What are common dynamical properties of semigroups generated by elements of $\mathfrak{F}_\alpha$, $\alpha\in J$?
\end{question}

\vspace{2mm}

We now turn to filtrations formed by starlike and prestarlike functions as they are presented in Subsection~\ref{subsec-prestar}.
The first question is of general interest from the point of view of geometric function theory.

\begin{question}
What is the topological structure of the boundaries  $\partial \So^*(1-\alpha)$, $\alpha \in [0,1]$?
\end{question}

Recall that for two continuous functions $a$ and $b$ on $J\in\R$, such that $a$ increasing and $b$ decreasing one, the family $\left\{\So^*(a(\al),b(\al))\right\}_{\al\in J}$ is a strict filtration of $\G_0$.

\begin{question}
Does the family $\left\{z(1+b(\alpha)z)^{\frac{a(\alpha)-b(\alpha)}{b(\alpha)}} \right\}_{\alpha \in J}$ form a net of totally extremal functions for this filtration?
\end{question}

\begin{question}
  Find $B(\So^*_M(\lambda))$ and $K(\So^*_M(\lambda))$.
\end{question}
Recall that in Theorem~\ref{c2-b} we establish some estimates for these quantities.

\vspace{2mm}

Regarding the filtration $\Rf=\left\{\Rf_\al \right\}_{\al\in[0,1]}$, we notice that the `starting' set $\Rf_0=\mathcal{C}$ is not a singleton, $\Rf_{\frac12}=\So^*\left(\frac{1}{2}\right)$ consists of univalent functions and $\Rf_1\subsetneq\G_0$. Therefore the following questions naturally arise.

\begin{question}
   How to complete `naturally' the filtration $\Rf$ by negative values of the parameter, say $\alpha\ge-a,\ a>0,$ to get $\Rf_{-a}=\{\Id\}$?
\end{question}

\begin{question}
  Does there exist $\beta>\frac{1}{2}$ such that $\Rf_\beta$ consists of univalent functions? Find supremum of such $\beta$.
\end{question}

\begin{question}
Do there exist $a<0$ and/or $b> 1$ such that the family  $\{\Rf_\al\}_{\al \in [a, b]}$ is a filtration of $\G_0$? In the case of affirmative answer, is this filtration strict?
\end{question}
\vspace{2mm}

Concerning filtration $\Af^1$, it is known from Lemma~\ref{exa-4.1} with $n=1$ that semigroups generated by $f \in\{\Af^1 _\al\}_{\al \leq 1}$ can have one repelling fixed point.
\begin{question}
Can the semigroup generated by some $f \in\{\Af^1 _\al\}_{\al \leq 1}$ have more than one repelling fixed point?
\end{question}

\vspace{2mm}

Next questions concern the classes $\Lf_\al$ defined by \eqref{alm-star}. Recall that by Theorem~\ref{filtration-alpha} the family $\{\Lf_\al\}_{\al\in [0,1]}$ is a filtration of $\G_0$, while for any $a>1$ the family $\{\Lf_\al\}_{\al\in [0,a]}$ is not.
\begin{question}
Does there exist $a<0$ such that the family $\{\Lf_\al\}_{\al\in(a,1]}$ is a filtration of $\G_0$?
\end{question}

We know that  $\Lf_0=\mathcal{S}^*(\frac 12)\subset \Lf_\alpha$ for all $\alpha\in [0,1]$ (see Remark~\ref{rem-alm-star1}). One asks
\begin{question}
For which $\al>0$ does the class $\Lf_\al$ consist of univalent (starlike) functions?
\end{question}

We have proved in Theorem~\ref{M-alp-1-alp} that the boundaries $\partial\Lf_\al$ are not empty as $\al>\frac12$.
\begin{conjecture}
  The filtration $\{\Lf_\al\}_{\al\in[0,1]}$ is strict.
\end{conjecture}

As for the classes $\Mf_\beta$, we proved that the family  $\{\Mf_\beta\}_{\beta \in (-\infty, 1]}$ is a filtration of $\G_0$ and that this filtration is strict for $\beta$ restricted to interval $(0,1]$, see Theorem~\ref{thm-filtr-beta-new}. Therefore the next questions remain open.
\begin{question}
Does there exist $b>1$ such that the family  $\{\Mf_\beta\}_{\beta \in (-\infty, b]}$ is a filtration?
Do there exist $a<0$ and $b\geq 1$ such that the family  $\{\Mf_\beta\}_{\beta \in [a, b]}$ is a strict filtration?
\end{question}

\begin{question}
Is the filtration $\{\Mf_\beta\}_{\beta\in(-\infty,1]}$ strict?
\end{question}

\bigskip

\section{Applications to estimates of coefficient functionals}\label{sect-appl}
\setcounter{equation}{0}

This section is devoted to one of the classical problems in geometric function theory, namely, to estimation of different functionals over various classes of analytic functions. Let $f\in \Hol$ have Taylor expansion $f(z)=\sum\limits_{n=0}^{\infty}a_nz^n$. Estimates of the Taylor's coefficients $a_n$ including the Bieberbach conjecture (see for example, \cite{Duren}) have a long history. Among other significant functionals we emphasize the generalized Zalcman functional $\Phi_{m,n}(f,\lambda) =a_1a_{m+n-1} -\lambda a_ma_n$ and its particular case the Fekete--Szeg\"o functional $\Phi(f,\lambda):=\Phi_{2,2}(f,\lambda) =a_1a_3 -\lambda a_2^2.$ The Fekete--Szeg\"o problem for a class of analytic functions is to find the sharp estimate on $|\Phi(f,\lambda)|$ over this class.

\subsection{Estimates for the analytic, prestarlike and pseudo-analytic filtrations}\label{subsect-an-pre-pse}
We begin with the estimates of some coefficient functionals over the classes $\Af_\al$,  $\al\leq 1$, defined by the inequality
\begin{equation*}
\Re \left[\alpha \frac{f(z)}{z} +(1-\alpha) f'(z) \right]\geq 0
\end{equation*}
and studied in Section~\ref{subsec-analytic}.
\begin{theorem}\label{thm-coef-anal}
  Let $f\in\Af_\alpha$ have the Taylor expansion $f(z)\!=\!z+\sum\limits_{n=2}^\infty a_nz^n.$ For ${2\!\le\!k\!\le\!m\!-\!1}$ denote
 \begin{equation}\label{Lambda}
 \Lambda_\alpha(m,k):=\frac{\left(\alpha +(1-\al)k\right) \left(\alpha + (1-\al)(m-k+1)\right)} {2 \left(\alpha +(1-\al)m\right)}.
 \end{equation}
Then
  \begin{equation}\label{estim1}
    \left| a_{m} - \Lambda_\alpha(m,k)a_k a_{m-k+1} \right|  +  \Lambda_\alpha(m,k) \left|a_k a_{m-k+1} \right|  \le\frac2{\alpha +(1-\al)m}.
  \end{equation}
  Consequently, $\displaystyle\left|a_m\right|\le \frac2{\alpha +(1-\al)m}$ and for any $\nu\in\C,$
  \begin{equation}\label{estim2}
    \left| a_{m} - \nu a_k a_{m-k+1} \right|\le  \frac2{\alpha +(1-\al)m} \, \max\left\{1,\left| 1-\frac{\nu } {\Lambda_\alpha(m,k)} \right|\right\}.\
  \end{equation}
  Inequalities \eqref{estim1}--\eqref{estim2} are sharp.
 \end{theorem}
\begin{proof}
  Since $f\in\Af_\al$, the function $q$ defined by $q(z)=\alpha \frac{f(z)}{z} +(1-\alpha) f'(z) $ has positive real part. Its Taylor coefficients are $p_n=\left[\alpha +(1-\al)(n+1)\right]a_{n+1},\ n\ge1.$
 According to \cite[Proposition~2.2]{E-V}, we have $\left| p_n - \frac12 p_kp_{n-k} \right| + \frac12  \left| p_kp_{n-k} \right|\le2$,
 $1\le k\le n- 1$, or,  which is the same,
  \begin{align*}
& \left| a_{n+1} - \frac{\left(\alpha +(1-\al)(k+1)\right) \left(\alpha + (1-\al)(n-k+1)\right)} {2 \left(\alpha +(1-\al)(n+1)\right)}a_{k+1} a_{n-k+1} \right|  \\
 & +   \frac{\left(\alpha +(1-\al)(k+1)\right) \left(\alpha + (1-\al)(n-k+1)\right)} {2 \left(\alpha +(1-\al)(n+1)\right)}
  \left|a_{k+1} a_{n-k+1} \right| \\
  & \le\frac2{\left(\alpha +(1-\al)(n+1)\right)} .
 \end{align*}
 This inequality is equivalent to \eqref{estim1} by \eqref{Lambda}. In its turn, \eqref{estim1} implies $\left|a_m\right|\le \frac2{\alpha +(1-\al)m}$ by the triangle inequality.

To proceed, we note that Lemma~2.1 in \cite{E-V} can be reformulated as follows:

{\it Let $a,b\in\C$ and $D>0$. Then $|a|+|b|\le D$ if and only if $|a+\mu b|\le D\max\left\{1,|\mu|\right\}$ for all $\mu\in\C$. }

Applying this conclusion to inequality \eqref{estim1} we get
\begin{equation*}
  \left| a_{m} - \Lambda_\alpha(m,k)(1-\mu) a_k a_{m-k+1} \right| \le \frac2{\alpha +(1-\al)m} \,\max\left\{1,|\mu|\right\}.
\end{equation*}

Denote $\nu:= \Lambda_\alpha(m,k)(1-\mu)$. Then the last displayed formula is equivalent to \eqref{estim2}.

The sharpness follows from the sharpness of the used estimates from  \cite[Proposition~2.2]{E-V}.
 \end{proof}
\begin{corollary}\label{corol-estim1}
 Let $f\in\Af_\alpha$ have the Taylor expansion $f(z)=z+\sum\limits_{n=2}^\infty a_nz^n$. Then  for any $\nu\in\C,$
  \[
     \left|\Phi(f,\nu)\right|=\left| a_3 - \nu a_2^2 \right|  \le  \frac2{3-2\alpha} \, \max\left\{1,\left| 1-\frac{2\nu (3-2\alpha)} {\left(2-\alpha\right)^2 } \right|\right\},
  \]
and this estimate is sharp for every $\alpha \leq 1$.
\end{corollary}

Note that estimates to the generalized Zalcman and Fekete--Szeg\"o functionals over generalized analytic filtrations can be treated similarly.

\bigskip

Now we turn to the prestarlike filtration and present the solution of the Fekete--Szeg\"o problem for classes~$\Rf_\al$.

\begin{theorem}\label{thm-prest-FS}
For every $\al\in[0,1]$, we have $\displaystyle \max_{f\in\Rf_\al}\left| \Phi(f,\lambda) \right| = \max \left\{ \frac1{3-2\al}, \left| \lambda -1 \right|   \right\}$.
\end{theorem}
\begin{proof}
  Let $f\in\Rf_\al$ have the Taylor expansion $f(z)=z+\sum\limits_{n=2}^\infty a_nz^n.$ Then
\[
\So^*(\al)\ni\frac{z}{(1-z)^{2(1-\al)}} *f(z) =z+\sum\limits_{n=2}^\infty c_nz^n,
\]
where
\[
c_2=2(1-\al)a_2\qquad\mbox{and}\qquad c_3=(1-\al)(3-2\al)a_3.
\]
It follows from \cite[Theorem~1]{Ke-Me} that $\left|c_3-\mu c_2^2 \right|\le(1-\al)\max\left\{1, \left| 2(1-\al)(2\mu-1)-1 \right| \right\}$ for any $\mu\in\C$ or, equivalently,
\[
\left|(3-2\al)a_3 -4\mu (1-\al)a_2^2  \right| \le \max\left\{ 1, \left| 4\mu(1-\al) - (3-2\al) \right|   \right\}.
\]
Denote $\lambda=\displaystyle\frac{4\mu(1-\al)}{3-2\al}$. The last displayed inequality means that
\[
\left| \Phi(f, \lambda ) \right| \le \max\left\{ \frac1{3-2\al}, \left| \lambda -1 \right|   \right\}.
\]

It remains to show that this estimate is sharp, more precisely, that for every $\lambda$ there is a function $f\in\Rf_\al$ for which it becomes equality.

First assume that $|\lambda-1|\ge  \frac1{3-2\al}$. As we already saw $\frac{1}{1-z}\in\Rf_\al$ for every $\al$. Since the Taylor coefficients of this functions are $a_2=a_3=1$, we have $a_3-\lambda a_2^2=1-\lambda$, so the equality holds.

Otherwise, in tne case where $|\lambda-1|< \frac1{3-2\al}$, let consider the function $f_\al\in\Ao$ defined by its Taylor expansion:
\[
f_\al(z)=z+ \sum_{k=1}^\infty \frac{(1-\al)_k(2k)!}{(2-2\al)_{2k}k!}\, z^{2k+1}.
\]
It follows from \eqref{hadamar} that
\[
f_\al(z)* \frac{z}{(1-z)^{2-2\al}} =z+\sum_{k=1}^{\infty} \frac{(1-\al)_k}{k!}\, z^{2k+1} =\frac{z}{(1-z^2)^{1-\al}}.
\]
Since the last function is starlike of order $\al$, we conclude that $f_\al\in\Rf_\al$. In this case $a_2=0$ and $a_3=\frac1{3-2\al}$. Thus $a_3-\lambda a_2^2=\frac1{3-2\al}$ which completes the proof.
\end{proof}

Note that for $\alpha=1$ this result is known (see, for example, Theorem~2.2 in \cite{EJ-est}).

\bigskip

Finally, we deal with estimates on the coefficient functionals over the classes $\Af^1_\al,\ \al\leq 1,$ defined by the inequality
\begin{equation*}
\left|\alpha\frac{f(z)}{z}+(1-\alpha)f'(z)-1\right|\leq 2-\alpha,
\end{equation*}
see Section~\ref{subsect-pseu-anal}.

\begin{theorem}\label{th-A1-coeff-estim}
Let $f \in \Af^1_\alpha$ have the Taylor expansion $f(z)=z+ \sum\limits_{k=2}^{\infty}a_kz^k$. The following assertions hold:
\begin{itemize}
\item [(i)] $\displaystyle|a_{k+1}|\leq \frac{2-\alpha}{k(1-\alpha)+1}$ for all $k \in \N$. Moreover, $\displaystyle|a_3|\leq \frac{2-\alpha}{3-2\alpha}\cdot(1-|a_2|^2)$ and $\displaystyle|a_4|\leq \frac{2-\alpha}{4-3\alpha}\cdot(1-|a_2|^2)$.
\item [(ii)] For every $\lambda \in \C$ we have $\displaystyle\left|\Phi(f,\lambda)\right| \leq  \max \left(\frac{2-\alpha}{3-2\alpha}\,, \,\, |\lambda|\right)$.
\end{itemize}
\end{theorem}
\begin{proof}
Let $f \in \Af^1_\alpha$, then by Proposition~\ref{th-A1-prop} there exists a unique function $\omega \in \Omega$, $\omega(z)=\sum\limits_{k=1}^{\infty}\beta_kz^k,$ such that \eqref{sol-diff-eq} holds. Then for any fixed $z \in \D$,
\begin{eqnarray}\label{f-coeff-A1-calc}
\nonumber f(z)&=&z\left(1+\frac{2-\alpha}{1-\alpha}\cdot \int_{0}^{1} s^{\frac{\alpha}{1-\alpha}}\left(\sum_{k=1}^{\infty}\beta_kz^ks^k \right)ds\right)\\
\nonumber &=&z+(2-\alpha)\cdot\sum_{k=1}^{\infty}\frac{\beta_k}{k+1-k\alpha}z^{k+1} \!.
\end{eqnarray}
Thus
 \begin{equation}\label{a_k-calc}
a_{k+1}= \frac{(2-\alpha)\beta_k}{k+1-k\alpha}, \qquad k\in \N.
 \end{equation}

Since $|\beta_k|\leq 1$ for all $k \in \N$, the estimates on $|a_{k+1}|$ follow.
Assume that $\gamma_k$, $k \in \N$, with $|\gamma_k|\leq 1$ are the Schur parameters of $\omega$, then by Schur’s recurrence relation (see for details  \cite{Sim}, see also \cite{EJ-coeff20}) we have
\begin{eqnarray}\label{beta_k}
\left\{ \begin{array}{ll}
          |\beta_1|  =&|\gamma_1|;\vspace{2mm}\\
          |\beta_2| =&(1-|\gamma_1|^2)|\gamma_2|\leq (1-|\gamma_1|^2);\vspace{2mm}\\
          |\beta_3| =&(1-|\gamma_1|^2)\left|(1-|\gamma_2|^2)\gamma_3-\gamma_1\gamma_2^2\right| \\
        \hspace{7mm}  \leq & (1-|\gamma_1|^2)\max\left(|\gamma_3|,|\gamma_1|\right)\leq (1-|\gamma_1|^2).
        \end{array}\right.
\end{eqnarray}
For $k=1$ formula~\eqref{a_k-calc} implies $a_2=\gamma_1$. For $k=2$ it follows from \eqref{a_k-calc}--\eqref{beta_k} that
\begin{equation*}
\displaystyle|a_3|\leq\frac{2-\alpha}{3-2\alpha}\cdot(1-|a_2|^2).
\end{equation*}
For $k=3$ formula~\eqref{beta_k} implies $|\beta_3|\leq (1-|a_2|^2)$.
Similarly to the above,
\begin{equation*}
\displaystyle|a_4|\leq\frac{2-\alpha}{4-3\alpha}\cdot(1-|a_2|^2).
\end{equation*}
Thus assertion (i) is proven.

Next we use formulas \eqref{a_k-calc} and \eqref{beta_k} to estimate the Fekete--Szeg\"o functional over $\Af^1_\alpha$ and get
\begin{eqnarray}\label{F-S-A1}
\nonumber|a_3-\lambda a_2^2|&=&\left|\frac{(2-\alpha)\beta_2}{3-2\alpha}-\lambda\frac{(2-\alpha)^2\beta_1^2}{(2-\alpha)^2}\right|
\leq\frac{2-\alpha}{3-2\alpha}\left|(1-|\gamma_1|^2)\gamma_2\right|+|\lambda|\cdot|\gamma_1|^2\\
\nonumber&\leq & (1-|\gamma_1|^2)\cdot \frac{2-\alpha}{3-2\alpha}+|\gamma_1|^2\cdot|\lambda|\leq \max \left(\frac{2-\alpha}{3-2\alpha},|\lambda|\right)\!.
\end{eqnarray}
This proves assertion (ii).
\end{proof}

\bigskip

\subsection{Fekete--Szeg\"o functional over quasi-starlike and Mocanu's filtrations}\label{subsec-FS}
 A~close look at the quasi-starlike and Mocanu's filtrations suggests the possibility of studying their properties together considering a two-parameter family  of functions containing both filtrations. For this aim assume that $\alpha,\beta \in \R$ and that $f\in\Ao$ satisfies $f(z)f'(z)\neq 0$ for all $z \in \D \setminus \{0\}$. We use the range of the expression (see \cite{EJT})
\begin{equation}\label{g}
g_{\alpha,\beta}(z)=\alpha\frac{f(z)}{z}+\beta\frac{zf'(z)}{f(z)}+(1-\alpha-\beta)\left(1+\frac{zf''(z)}{f'(z)}\right)
\end{equation}
and define the classes
\begin{equation}\label{class-M}
M_{\alpha,\beta}=\left\{f \in \Ao:\ f(z)f'(z)\neq0\ \mbox{and}\ \Re \left[g_{\alpha,\beta}(z)\right] > \frac{\alpha+\beta}{2},\  z\in\D\setminus\{0\} \right\}.
\end{equation}
Note that
\begin{itemize}
   \item $M_{\alpha,1-\alpha}=\Lf_\al$ and \vspace{1mm}
  \item $M_{0,\beta}=\Mf_\beta$. \vspace{1mm}
\end{itemize}
 In particular, $M_{0,1}=\mathcal{S}^*\left(\frac{1}{2}\right)$ and $M_{0,0}=\mathcal{C}$. In addition, $M_{\alpha,\beta}=\emptyset$ when $\alpha+\beta\geq 2$.  \vspace{2mm}

The results in this subsection can be found in \cite{EJT}.
The first theorem stratifies the range of $(\alpha,\beta)$ according to the values of the Fekete--Szeg\"{o} functional $\Phi(\cdot, \lambda)$.
\begin{theorem}\label{th-FS-estim}
Let $\alpha, \beta \in\R $ satisfy $\alpha+\beta<2$ and $5\alpha+4\beta<6$. Denote
\begin{equation*}
\mu:=\frac{2-\alpha-\beta}{6-5\alpha-4\beta}>0.
\end{equation*}
Then $|\Phi(f, \lambda)| \leq \max \{\mu,|1-\lambda|\}$,  $\lambda \in \C,$ over the class $M_{\alpha,\beta}$.
\end{theorem}

\begin{proof}
Let $f \in M_{\alpha,\beta}$ has  the Taylor expansion $f(z)=z+\sum\limits_{k=2}^\infty a_k z^k$.
Then
$$f'(z)=1+2a_2z+3a_3z^2+\ldots$$
and
$$f''(z)=2a_2+6a_3z+12a_4z^2+\ldots.$$
Let $g_{\alpha,\beta}$ be defined by formula \eqref{g}. We construct function $\omega$ substituting $g:=g_{\alpha,\beta}$ in formula \eqref{Psi}. One can see that the Taylor expansion of $\omega$ is
\begin{equation*}
\omega(z)=a_2z+(a_3-a_2^2)\mu z^2+ z^3\sum_{k=3}^{\infty}b_k z^{k-1}.
\end{equation*}
Now, for every $s \in \C$ we have
\begin{eqnarray*}
\left|b_2-sb_1^2 \right|= \left|\mu( a_3-a_2^2)-s a_2^2\right| = \mu \cdot\left|a_3-\left(1+\frac{s}{\mu}\right)a_2^2\right|.
\end{eqnarray*}
Denoting
\begin{equation}\label{mu-lambda}
 \lambda:= 1+s\mu,
\end{equation}
we get
\begin{equation*}
\left|a_3-\lambda a_2^2\right|=|\mu| \cdot \left|b_2-sb_1^2\right|
\end{equation*}
It is well known (see, for example, \cite{Ke-Me}) that if $\omega\in\Omega, \ \omega(z)=\sum\limits_{k=1}^{\infty} b_k z^k$ then $|b_2| \leq 1-|b_1|^2$.
Consequently, $|b_2-sb_1^2|\leq \max \{1,|s|\}$ for every  $s \in \C$. So
\begin{equation*}
\left|a_3-\lambda a_2^2\right|\leq \max \{\mu,\mu |s|\}.
\end{equation*}
 Finally, by \eqref{mu-lambda}, the  result follows.
\end{proof}

In connection with this theorem, we note that the level sets of the function \linebreak ${\mu(\alpha,\beta)=\displaystyle\frac{2-\alpha-\beta}{6-5\alpha-4\beta}}$ are rays starting at the point $(-2,4)$ and lying under the lines $\alpha+\beta=2$ and $5\alpha+4\beta=6.$ See Fig.~\ref{fig-mu-level}, where level sets of the function $\mu$ are represented by dot lines.
Assume that estimate in Theorem~\ref{th-FS-estim} is sharp and a curve in the plane $(\alpha,\beta)$ represents a filtration of $\G_0$. Then if it intersects each such ray at most at one point.
 In particular, the bold segments at the bottom of Fig.~\ref{fig-mu-level} represent the quasi-starlike filtration $\Lf$ and Mocanu's filtration $\Mf$.

\begin{figure}
\begin{center}
\includegraphics[angle=0,width=5cm,totalheight=5cm]{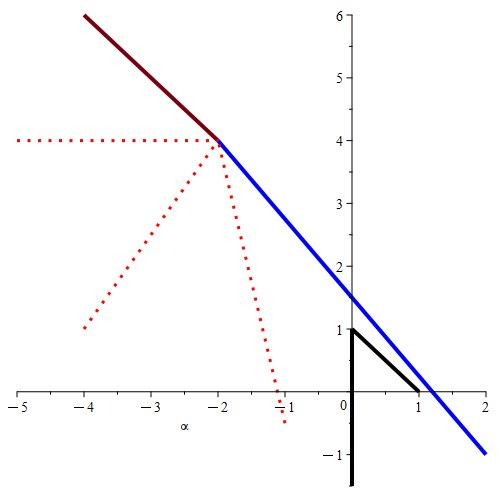}
\caption{Level sets of $\mu(\alpha, \beta)$}\label{fig-mu-level}
\end{center}
\end{figure}

\vspace{2mm}

To prove subsequent results we use Theorem~1 and 2 from \cite{MiMo-85} that can be combined as follows.
\begin{lemma}\label{lem-mm-2}
Let $\beta,\gamma \in \C$,  $\beta\neq0$. Let $h\in\Hol(\D,\C)$ with $h'(0)\neq0$ and $P(z)=\beta h(z)+\gamma$. Consider the Briot--Bouquet differential equation
\[ q(z)+\frac{zq'(z)}{\beta q(z)+\gamma} = h(z), \quad h(0)=q(0).
\]
\begin{itemize}
  \item[(i)] If  $\Re P(z)>0$ for $z\in\D$ then its solution $q$ is analytic in $\D.$
  \item[(ii)] If, in addition, functions $Q(z):=\log P(z)$ and $R(z):= 1/P(z)$ are convex in $\D$, then $q$ is univalent in $\D$.
\end{itemize}
\end{lemma}

Now we are ready to present the solution of the Fekete--Szeg\"o problem for the classes $\Lf_\alpha$  of the quasi-starlike filtration.
\begin{theorem}\label{th-F-S-quasist}
 Let $\alpha<2$ and  $f \in \Lf_\alpha$. Then
\begin{equation}\label{F-S-quasist}
|\Phi(f, \lambda)| \leq \max \left\{\frac{1}{2-\alpha},|1-\lambda|\right\},\quad \lambda \in \C.
\end{equation}
 Moreover, if $\alpha \in \left[\frac{1}{2},1\right]$, then this estimate is sharp.
\end{theorem}
\begin{proof}
  Theorem \ref{th-FS-estim} applied to the case $\beta=1-\alpha$  gives
  \[
   \sup_{f\in M_{\alpha,1-\alpha}}|\Phi(f, \lambda)| \le \max \left\{\frac{1}{2-\alpha},|1-\lambda|\right\},\quad \lambda \in \C.
   \]
   For $\alpha=0$ it follows from a classical result of Koegh and Merkes \cite{Ke-Me} that this supremum attains at the value in the right-hand side, while for $\alpha=1$ it follows from \cite{EJ-est, E-J-21a}.
To prove sharpness of estimate \eqref{F-S-quasist} for every $\alpha\in\left(\frac{1}{2},1\right)$, we show that there are two functions $f^{(1)},f^{(2)}\in \Lf_\alpha$  such that the functions $g_{\alpha,1-\alpha}^{(1)},g_{\alpha,1-\alpha}^{(2)}$ constructed for them by \eqref{g} are
  \begin{eqnarray*}\label{g-1-n}
    g_{\alpha,1-\alpha}^{(k)}(z) = \frac{1}{1-z^k},\quad k=1,2.
  \end{eqnarray*}

Choose $\beta=\frac{\alpha}{1-\alpha},\ \gamma=0,\ c=1$ and $h^{(k)}(z)=1+\frac1\alpha \cdot \frac{z^k}{1-z^k}\,,\ k=1,2,$ to use Lemma~\ref{lem-mm-2}. Then
\[
\Re\left[\beta h^{(k)}(z)+\gamma \right]=\frac{\alpha}{1-\alpha} +\frac{1}{1-\alpha}\cdot \frac{z^k}{1-z^k}>\frac{\alpha-\frac12}{1-\alpha}>0 \, \text{ for } \, \frac{1}{2}<\alpha<1 .
\]
Therefore, this lemma can be applied, that is, the solutions $q^{(k)}$ of the corresponding (Briot--Bouquet) differential equation belong to $\Hol$. By construction, the functions $f^{(k)}(z)=zq^{(k)}(z)$ belong to $\Lf_\alpha$.

  Further, using the fact that $g_{\alpha,1-\alpha}(z)=\alpha\cdot \frac{f(z)}{z}+(1-\alpha)\cdot \frac{zf'(z)}{f(z)}$, one calculates first Taylor coefficients of $f^{(k)}$. In particular, for $f^{(1)}$ we can see that $a_2=a_3=1$. Hence, $\Phi(f^{(1)}, \lambda) = 1-\lambda$. Repeating such calculation for $f^{(2)}$ we get $a_2=0, \ a_3=\frac1{2-\alpha}.$ So, $\Phi(f^{(2)}, \lambda) = \frac1{2-\alpha}.$  The proof is complete.
\end{proof}

\vspace{2mm}

Next we present the solution of the Fekete--Szeg\"o problem for classes $\Mf_\alpha$ of Mocanu's filtration.
 \begin{theorem}\label{th-F-S-mocanu}
   Let $\beta<\frac{3}{2}$ and  $f \in \Mf_\beta$. Then
   \begin{equation*}
   |\Phi(f, \lambda)| \leq \max \left\{\frac{2-\beta}{6-4\beta},|1-\lambda|\right\},\quad \lambda \in \C.
   \end{equation*}
 Moreover, if $\beta \in [0,1]$, then this estimate is sharp.
 \end{theorem}

 \begin{proof}
Note that Theorem \ref{th-FS-estim} applied to the case $\alpha=0$  gives
  \[
   \sup_{f\in M_{0,\beta}}|\Phi(f, \lambda)| \le \max \left\{\frac{2-\beta}{6-4\beta},|1-\lambda|\right\},\quad \lambda \in \C.
   \]
According to the result of Keogh and Merkes in \cite{Ke-Me}, this supremum is attained at the right-hand side value for $\beta=0$ and $\beta=1$.  Thus we have to prove that the estimate is sharp whenever $\beta\in(0,1)$.

   For this purpose, we show that there are two functions $f^{(1)},f^{(2)}\in M_{0,\beta}$  such that the functions $g_{0,\beta}^{(1)},g_{0,\beta}^{(2)}$ constructed for them by \eqref{g} are
  \begin{eqnarray}\label{g-1}
    g_{0,\beta}^{(1)}(z) &=& \frac{\beta}{2} +\left( 1- \frac{\beta}{2} \right) \cdot \frac{1+z}{1-z} \,,\\
    g_{0,\beta}^{(2)}(z) &=& \frac{\beta}{2} +\left( 1- \frac{\beta}{2} \right) \cdot \frac{1+z^2}{1-z^2} \,,\nonumber
  \end{eqnarray}
  respectively. Denote $q(z):=\frac{z {f^{(1)}}'(z)}{f^{(1)}(z)}$. Then equality \eqref{g-1} coincides with
  \[
  q(z)+(1-\beta)\cdot \frac{zq'(z)}{q(z)} = \frac{\beta}{2} +\left( 1- \frac{\beta}{2} \right) \cdot \frac{1+z}{1-z} \,.
  \]
  It follows from Lemma~\ref{lem-mm-2} that the solution $q$ of this (Briot--Bouquet) differential equation is analytic in the unit disk $\D$. Using this solution we conclude that $f^{(1)}$ is analytic in $\D$ too. By construction, it belongs to the class $M_{0,\beta}$. Similarly, one considers the case of $f^{(2)}$.

  Further, equality \eqref{g-1} enables to calculate early Taylor coefficients of $f^{(1)}$, in particular, we can see that $a_2=a_3=1$. Hence, $\Phi(f^{(1)}, \lambda) = 1-\lambda$. Repeating such calculation for $f^{(2)}$, we get $a_2=0, \ a_3=\frac{2-\beta}{6-4\beta}.$ Thus $\Phi(f^{(2)}, \lambda) = \frac{2-\beta}{6-4\beta}.$
  This completes the proof.
   \end{proof}

\bigskip

\subsection{Coefficient estimates for non-linear resolvents}\label{subsec-resolv}

One of main properties of generators is the following result of Reich and Shoikhet (\cite{R-S-96}, see also \cite{E-R-S-19, R-S1, SD}):

\begin{theorem}\label{thm-RS-resolv}
  Let $f\in\Hol$. Then $f\in\G$ if and only if for every $z\in\D$ and every $r>0$ the functional equation $w+rf(w)=z$ has the unique solution  $w=G_r(z)$ such that $G_r\in\HolD$.
\end{theorem}

The functions $G_r:=\left(\Id +rf \right)^{-1},\ r>0,$ are called the {\it non-linear resolvents} of $f\in\G$. They play essential role in dynamical systems (see, for example, \cite{E-R-S-19, R-S1, SD}). Coefficient inequalities over the classes of nonlinear resolvents were recently studied in \cite{EJ-est, EJ-coeff20}. Here we complete the previous results by more qualified estimates.

\begin{theorem}\label{thm-estim-resol}
  Let $\lambda\in\C,$ $f\in\G_0$ and let $G_r$ be its non-linear resolvent with $r>0$.  Then
  \begin{equation}\label{ineq-nu}
   \left|\Phi(G_r,\lambda)\right| = \frac{r}{(1+r)^5}  \left| \Phi(f, \nu) \right|,\qquad \nu= (2-\lambda)\frac{r}{1+r}.
  \end{equation}
  Consequently,
  \begin{itemize}
    \item [(i)] If $f\in\Af_\al$ with $\al \leq 1$, then
    \[
  \left|\Phi(G_r,\lambda)\right| \le  \frac{r}{(1+r)^5} \, \max\left\{\frac{2}{3-2\alpha},\left| \frac{2}{3-2\alpha}- \frac{4 } {\left(2-\alpha\right)^2 }\cdot\frac{r}{1+r}(2-\lambda) \right|\right\}.
  \]
  \item [(ii)] If $f\in\Af^1_\al$ with $\al \leq 1$, then
   \[
   \displaystyle\left|\Phi(G_r,\lambda)\right| \leq \frac{r}{(1+r)^5}\max \left(\frac{2-\alpha}{3-2\alpha}\,,\,\, \frac{r}{1+r}|2-\lambda|\right).
   \]
 \item [(iii)] If $f\in\Rf_\al$ with $\al \in[0,1]$, then
      \[
  \left|\Phi(G_r,\lambda)\right| \le  \frac{r}{(1+r)^5} \max\left\{ \frac{1}{3-2\al}, \left|\frac{1+r(\lambda-1)}{1+r} \right| \right\} .
  \]
  \item [(iv)] If $f\in\Lf_\al$ with $\al <2$, then
      \[
  \left|\Phi(G_r,\lambda)\right| \le \frac{r}{(1+r)^5}  \max \left\{\frac{1}{2-\alpha},\left|\frac{1+r(\lambda-1)}{1+r}\right|\right\}  .
  \]
    \item [(v)] If $f\in\Mf_\beta$ with $\beta <\frac32$, then
      \[
  \left|\Phi(G_r,\lambda)\right| \le \frac{r}{(1+r)^5} \max \left\{\frac{2-\beta}{6-4\beta},\left| \frac{1+r(\lambda-1)}{1+r} \right|\right\}  .
  \]

  \end{itemize}
 \end{theorem}

\begin{proof}
  Let $f$ have the Taylor expansion $f(z)=z+\sum\limits_{n=2}^\infty a_nz^n$. Assume that for some $r>0$ its resolvent has the Taylor expansion $G_r(z)=\sum\limits_{n=1}^\infty b_nz^n$. Differentiating the equality $G_r(z)+rf(G_r(z))=z$  at $z=0$, we find $b_1=\frac{1}{1+r},\ b_2=-\frac{r a_2}{(1+r)^3}$ and $b_3=\frac{2r^2a_2^2}{(1+r)^5} - \frac{r a_3}{(1+r)^4}$ (see also \cite{EJ-coeff20, EJ-est}). Therefore
  \[
  b_1b_3-\lambda b_2^2 = - \frac{r}{(1+r)^5} \left[ a_3 - (2-\lambda)\frac{r}{1+r}\,a_2^2  \right]
  \]
which proves \eqref{ineq-nu}.

To verify assertion (i), let $f\in\Af_\al$. Then  $  \left|\Phi(f,\nu)\right| \le  \frac2{3-2\alpha} \, \max\left\{1,\left| 1-\frac{2\nu (3-2\alpha)} {\left(2-\alpha\right)^2 } \right|\right\}$  by Corollary~\ref{corol-estim1}. Thus formula~\eqref{ineq-nu} implies
  \[
      \left|\Phi(G_r,\lambda)\right|  \le  \frac{2r}{(3-2\alpha)(1+r)^5} \, \max\left\{1,\left| 1-\frac{2\nu (3-2\alpha)} {\left(2-\alpha\right)^2 } \right|\right\}   .
  \]
This is equivalent to the required inequality in assertion (i).

Assertion (ii) follows directly from assertion (ii) in Theorem~\ref{th-A1-coeff-estim} by formula~\eqref{ineq-nu}.

To proceed, we notice that $1-\nu =\frac{1+r(\lambda-1)}{1+r}\,.$

Assume that $f \in \Lf_\alpha,\ \alpha<2$. Then $|\Phi(f, \lambda)| \leq \max \left\{\frac{1}{2-\alpha},|1-\lambda|\right\}$ by Theorem~\ref{th-F-S-quasist}. Therefore relation~\eqref{ineq-nu} implies
    \begin{align*}
   \left|\Phi(G_r,\lambda)\right| \le \frac{r}{(1+r)^5}  \max \left\{\frac{1}{2-\alpha},|1-\nu|\right\} & \\
    =  \frac{r}{(1+r)^5}  \max \left\{\frac{1}{2-\alpha},\left|\frac{1+r(\lambda-1)}{1+r}\right|\right\}&,
    \end{align*}
   so, assertion (iii) is proven.  The proof of assertions (iv) and (v) is similar.
\end{proof}

\bigskip

\subsection{Open questions and further discussion}\label{ssect-quest3}

\setcounter{question}{0}

Since the results of this section (excepting Theorem~\ref{thm-coef-anal} and Corollary~\ref{corol-estim1}) relate to non-linear filtrations, the following problems are actual:

\begin{question}
  Find estimates on coefficient functionals over the hyperbolic filtration $\Hf$. Find estimates on the Fekete--Szeg\"o functional over {\it resolvents} of generators $f\in\Hf_\al,\, \al\in\left[0,\frac23\right)$.
\end{question}

\begin{question}
  Find estimates to the generalized Zalcman and Fekete--Szeg\"o functionals over the sets $\Gf_\alpha^\lambda$.
\end{question}
\vspace{2mm}

Now we turn to the classes $M_{\alpha,\beta}$ defined by \eqref{class-M}.  The study of properties of these classes and their elements includes answers to the next questions:

\begin{question}
What values $(\alpha,\beta)$ provide the existence of a totally extremal function for the class $M_{\alpha,\beta}$?
\end{question}

\begin{question}
What values $(\alpha,\beta)$ provide the class $M_{\alpha,\beta}$ consists of univalent functions?
\end{question}
The only case we know the affirmative answer is $\alpha=0,\ \beta\leq 1$, see Statement~1 in Subsection~\ref{subsec-mocanu1}.

Since our special interest is to dynamical systems, let recall  that $M_{\alpha,\beta} \subset \mathcal{G}_0$ if either  $\beta=1-\alpha, \ \alpha>0$ (Theorem~\ref{filtration-alpha}), or $\alpha=0,\ \beta \leq 1$ (Theorem~\ref{thm-filtr-beta-new}). In general, the next question remains open.
\begin{question}
What values $(\alpha,\beta)$ provide $M_{\alpha,\beta} \subseteq \G_0$?
\end{question}
If it is the case for some pair $(\alpha,\beta)$ (including the cases mentioned here), it is important to answer
\begin{question}
  \begin{itemize}
    \item [(a)] Can the semigroup generated by some $f\in M_{\al,\beta}$ have repelling fixed points?
    \item [(b)] Are all semigroups generated by $f\in M_{\al,\beta}$ exponentially squeezing? If yes, find $K(M_{\al,\beta})$.
    \item [(c)]Do all semigroups generated by $f\in M_{\al,\beta}$ admit analytic extension with respect to the semigroup parameter? If yes, find $B(M_{\al,\beta})$.
  \end{itemize}
\end{question}
\noindent Questions (b) and (c) are open even for the classes $\Lf_\alpha$ and $\Mf_\beta$ studied above.

\vspace{2mm}

To proceed we recall that by Theorem~\ref{th-FS-estim} we have
\[
|\Phi(\cdot, \lambda)| \leq \max \{\mu,|1-\lambda|\}, \text{ where } \mu=\left\{ \begin{array}{ll}
                                               \frac{1}{2-\alpha} & \mbox{ for }\ f\in\Lf_\al,\vspace{2mm} \\
                                               \frac{2-\beta}{6-4\beta} &  \mbox{ for }\  f\in \Mf_\beta\vspace{2mm}, \end{array}                       \right.
\]
whenever $\alpha\leq 2$, $\beta<\frac{3}{2}$.
We also know that this estimate is sharp if either  $\frac12\leq \alpha \leq1$, or  $0\leq \beta \leq 1$ (see Theorems~\ref{th-F-S-quasist} and \ref{th-F-S-mocanu}).
\begin{question}
Is this estimate sharp for all $\alpha\leq 2$ and $\beta<\frac{3}{2}$?
\end{question}
If the answer is negative,  one asks
\begin{question}
Find \,  $\sup\limits_{f \in\Lf_\al} |\Phi(f, \lambda)|$ for $\alpha\leq 2$ and $\sup\limits_{f \in \Mf_\beta} |\Phi(f, \lambda)|$ for $\beta<\frac{3}{2}$.
\end{question}

\vspace{2mm}

Analysing solutions of the Fekete--Szeg\"{o} problem for different classes of analytic functions, one notice some similarity of estimates for different classes of functions. For instance,
  \[
    \left| \Phi( f,\lambda)\right|\le \max(\mu,|1-\lambda|) \quad \text{with} \quad \mu=\left\{ \begin{array}{ll}
                                               \frac13 & \mbox{ for }\ f\in\mathcal{C},\vspace{2mm} \\
                                               \frac12 &  \mbox{ for }\  f\in \mathcal{S}^*\left(\frac12\right),\vspace{2mm}\\
                                               1 &  \mbox{ for }\  f\in \left\{f\in\Ao:\ \Re\frac{f(z)}{z} > \frac 12\right\}. \end{array}                       \right.
  \]
Keogh and Merkes proved the results for the classes $\mathcal{C}$  and  $\mathcal{S}^*\left(\frac12\right)$ in \cite{Ke-Me}. As for the last class, see, for example,  \cite[Theorem~2.2]{EJ-est}.

The structure of the bound on $\left| \Phi( f,\lambda)\right|$ in the above inequality prompts to study the questions:

\begin{question}
Let $\Psi$ be a positive function defined on the set $\{(\lambda,\al):\, \lambda \in \C, \,\al \in J\}$.  \linebreak Under what conditions on $\Psi$ there exists a filtration $\left\{\Ff_\al\right\}_{\al\in J}$ of $\G_0$ such that \linebreak $\sup\limits_{f \in \Ff_\al} |\Phi(f,\lambda)|=\Psi(\lambda,\al)$? Is this filtration unique? How to construct it?
\end{question}

A necessary condition is clear: for any fixed $\lambda\in\C$, the function $\Psi(\lambda,\cdot)$ should be non-decreasing. Sufficient conditions are unknown in general.
Nevertheless, the mentioned estimate involves $\Psi(\lambda,\al) =  \max(\mu(\al),|1-\lambda|),$ where $\mu$ is an increasing function from the set $J$ onto $\left[\frac13,1\right]$. The above results present different solutions. Namely, Theorems~\ref{th-F-S-mocanu} and \ref{th-F-S-quasist} give us filtrations of generators that cover the cases $\mu\in\left[\frac13,\frac12\right] $ and $\mu\in\left[\frac23,1\right]$, while in Theorem~\ref{thm-prest-FS} we establish that filtration $\Rf$ covers the whole range $\mu\in\left[\frac13,1\right]$.

Moreover, the concrete solutions presented in the above theorems give particular answers to another question that is of certain interest. Observe that each function of the class $\Ao$ is locally invertible around the origin. Moreover, it can be easily verified that $|\Phi(f^{-1},\lambda)| = |\Phi(f,2-\lambda)|$ (see, for example, \cite{E-J-21a}). So, we ask:
\begin{question}
 Describe  all classes $\Ff \subset \Ao$ that satisfy $\sup\limits_{f \in \Ff} |\Phi(f^{-1},\lambda)|=\sup\limits_{f \in \Ff} |\Phi(f,\lambda)|$.
 \end{question}

To explain the importance of the above questions, notice that it can be interpreted as an inverse Fekete--Szeg\"{o} problem. In principle, inverse problems are necessary for the integrity of any theory. Although the Fekete--Szeg\"{o} problems have been studied for almost one hundred years, inverse problems for it have been posed only recently in \cite{EJT}.


\vspace{2mm}



\bigskip

\begin{thebibliography}{999}



\bibitem{Ababook89} M. Abate,
{\sl Iteration theory of holomorphic maps on taut manifolds}, Mediterranean Press, Rende, 1989.


\bibitem{AM-92} M. Abate,
The infinitesimal generators of semigroups of holomorphic maps, \textit{Ann. Mat. Pura Appl.} \textbf{161} (1992), 167--180.


\bibitem{A-E-R-S} D. Aharonov, M. Elin, S. Reich and D. Shoikhet,
Parametric representations of semi-complete vector fields on the unit balls in $\C^n$ and in Hilbert space, \textit{Atti Accad.
Naz. Lincei} \textbf{10} (1999), 229--253.


\bibitem{A-R-S} D. Aharonov, S. Reich and D. Shoikhet,
Flow invariance conditions for holomorphic mappings in Banach spaces, \textit{Math. Proc. R. Ir. Acad.} \textbf{99A} (1999),
93--104.

\bibitem{Ak58} L. A. Aksentiev, Sufficient conditions for univalence of regular functions, {\it Izv.Vyssh. Uchebn. Zaved. Mat.} {\bf 3} (1958), 3--7 (Russian).

\bibitem{ACP} C. Avicou, I. Chalendar, J. R. Partington,
Analyticity and compactness of semigroups of composition operators, {\it J. Math. Anal. Appl.} {\bf 437} (2016), 545--560.


\bibitem{BE-PH} E. Berkson and H. Porta,
Semigroups of analytic functions and composition operators, \textit{Michigan Math. J.} \textbf{25} (1978), 101--115.





\bibitem{B-C-DM-book} F. Bracci, M.~D. Contreras and S. D\'{\i}az-Madrigal,
{\sl Continuous Semigroups of holomorphic self-maps of the unit disc}, Springer Monographs in Mathematics, Springer, 2020.


\bibitem{BCDES} F. Bracci, M. D. Contreras, S. Díaz-Madrigal, M. Elin and D. Shoikhet,
Filtrations of infinitesimal generators,  \textit{Funct. Approx. Comment. Math.}  \textbf{59} (2018), 99–115.


\bibitem{BraGum16} F. Bracci, P. Gumenyuk,
Contact points and fractional singularities for semigroups of holomorphic self-maps in the unit disc, {\it J. Anal. Math.}, \textbf{130} (2016),  185--217.


\bibitem{BHMW} L. Brickman, D.J. Hallenbeck, T.H. MacGregor and D.R. Wilken, Convex hulls and extreme points of
families of starlike and convex mappings, {\it Trans. Am. Math. Soc.} {\bf 185} (1973), 413--428.


\bibitem{Duren} P. Duren,
{\sl Univalent Functions}, Springer, New York, 1983.


\bibitem{C-DM-P} M. D. Contreras, S. Díaz-Madrigal and Ch. Pommerenke,
On boundary critical points for semigroups of analytic functions,  \textit{Math. Scand.}  \textbf{98} (2006), 125--142.


\bibitem{E-V} I. Efraimidis and D. Vukoti\'{c},
Applications of Livingston-type inequalities to the generalized Zalcman functional, {\it Math. Nach.} \textbf{291} (2018), 1502--1513. https://doi.org/10.1002/mana.201700022.


\bibitem{E-J} M. Elin and F. Jacobzon,
Analyticity of semigroups in the right half-plane, {\it J. Math. Anal. Appl.} {\bf 448} (2017), 750--766.


\bibitem{EJ-coeff20} M. Elin and F. Jacobzon,
Coefficient body for nonlinear resolvents, {\it Annales UMCS}, {\bf 2} (2020), 41--53.


\bibitem{EJ-spiral} M. Elin and F. Jacobzon,
The Fekete--Szeg\"o problem for spirallike mappings and non-linear resolvents in Banach spaces, {\it Stud. Univ. Babe\c{s}--Bolyai Math.} {\bf 67} (2022), 329--344.  http://dx.doi.org/10.24193/subbmath.2022.2.09.


\bibitem{E-J-21a} M. Elin and F. Jacobzon,
Families of inverse functions: coefficient bodies and the Fekete--Szeg\"{o} problem,  {\it Mediterr. J. Math.}  {\bf 19}   (2022).  https://doi.org/10.1007/s00009-022-02017-2.


\bibitem{EJ-est} M. Elin and F. Jacobzon,
Estimates on some functionals over non-linear resolvents, {\it Filomat} {\bf 37}, (2023), 797--808.


\bibitem{E-J-S} M.Elin, F. Jacobzon and D. Shoikhet,
Filtration families of semigroup generators. The Fekete-Szegö problem, {\it Applied Set-Valued Analysis and Optimization}, accepted for publication.


\bibitem{EJT} M.Elin, F. Jacobzon and N. Tuneski,
The Fekete--Szego functional and filtration of generators, {\it Rend. Circ. Mat. Palermo, II. Ser} (2022).  https://doi.org/10.1007/s12215-022-00824-w


\bibitem{E-R-S-02} M. Elin, S. Reich and D. Shoikhet,
Asymptotic behavior of semigroups of $\rho$-non-expansive and holomorphic mappings on the Hilbert ball, {\it Ann. Mat. Pura Appl.} {\bf 181} (2002), 501--526.


\bibitem{E-R-S-04} M. Elin, S. Reich and D. Shoikhet,
Complex Dynamical Systems and the Geometry of Domains in Banach Spaces, {\it Dissertationes Math. (Rozprawy Mat.)} {\bf 427} (2004), 62 pp.


\bibitem{E-R-S-19} M. Elin, S. Reich and D. Shoikhet,
{\sl Numerical range of holomorphic mappings and applications}, Birkh\"{a}user, Cham, 2019.


\bibitem{E-Sbook} M. Elin and D. Shoikhet,
\textsl{Linearization Models for Complex Dynamical Systems. Topics in univalent functions, functional equations and semigroup
theory}, Birkh{\"{a}}user Basel, 2010.


\bibitem{ESS}  M. Elin, D. Shoikhet and T. Sugawa,
Filtration of semi-complete vector fields revisited, in: {\sl Complex analysis and dynamical systems. New trends and open problems}, 93--102, Trends in Math., Birkhäuser/Springer, Cham, 2018.


\bibitem{E-S-T} M. Elin, D. Shoikhet and N. Tarkhanov,
Analytic extension of semigroups of holomorphic mappings and composition operators, {\it Comput. Methods Function Theory} {\bf 18} (2018), 269--294.


\bibitem{E-S-Tu1}	M. Elin, D. Shoikhet and N. Tuneski,
Parametric embedding of starlike functions, {\it Compl. Anal. Oper. Theory} {\bf11} (2017), 1543--1556.


\bibitem{E-S-Tu2}	M. Elin, D. Shoikhet and N. Tuneski,
Radii problems for starlike functions and semigroup generators, {\it Comput. Methods Funct. Theory} {\bf20} (2020), 297--318.

\bibitem{ESZ}  M. Elin, D. Shoikhet and L. Zalcman,
A flower structure of backward flow invariant domains for semigroups, {\it Ann. Acad. Sci. Fenn. Math.} {\bf  33} (2008), 3--34.

\bibitem{Fu} S. Fukui,
On $\alpha$-convex functions of order $\beta$, {\it  Internat. J. Math. Math. Sci.} {\bf 20}  (1997),  769--772.


\bibitem{Gi-Ku} S. Giri and S. S. Kumar,
Radius and convolution problems of analytic functions involving semigroup generators, available in: arXiv:2205.10777.


\bibitem{Golu} G. M. Goluzin,
Some estimates for bounded functions, {\it Mat. Sb. (N.S.)}, {\bf 26(68)}:1 (1950), 7--18


\bibitem{GAW} A.~W. Goodman,
{\sl Univalent functions}, Vol 1, Mariner publishing company, Inc., 1983.


\bibitem {GHK2020} I. Graham, H. Hamada and G. Kohr,
Loewner chains and nonlinear resolvents of the Carath\'{e}odory family on the unit ball in $\C^n$. {\it J. Math. Anal. Appl.} {\bf 491}  (2020). https://doi.org/10.1016/j.jmaa.2020.124289


\bibitem{GI-KG} I. Graham and G. Kohr,
\textsl{Geometric function theory in one and higher dimensions}, Marcel Dekker, New York and Basel, 2003.


\bibitem{GKR-75} K. R. Gurganus,
$\Phi$-like holomorphic functions in $\C^n$ and Banach spaces, {\it Trans. Amer. Math. Soc.} {\bf 205} (1975), 389--406.

\bibitem{Hal} D. J. Hallenbeck, Convex hulls and extreme points of some families of univalent functions, {\it Trans. Amer. Math. Soc.} {\bf 192} (1974),  285--292.

\bibitem{HR75} D.~J. Hallenbeck and St. Ruscheweyh,
Subordination by convex  functions, {\it Proc. Amer. Math. Soc.} \textbf{52} (1975), 191--195.


\bibitem{HKK2021}  H. Hamada, G. Kohr, M. Kohr,
The Fekete--Szeg\"o problem for starlike mappings and nonlinear resolvents of the Carath\'eodory family on the unit balls of complex Banach spaces, {\it Anal. Math. Phys.} {\bf 11} (2021). https://doi.org/10.1007/s13324-021-00557-6.


\bibitem{JLR} F. Jacobzon, M. Levenshtein, S. Reich,
Convergence characteristics of one-parameter continuous semigroups, \textit{Anal. Math. Phys.} \textbf{1} (2011), 311--335.

\bibitem{J-S-S} M. Jahangiri, H. Silverman and E. M. Silvia,
Classes of functions defined by subordination, in: {\it New Trends in Geometric Function theory and applications}, ed. by R. Parvatham and S. Ponnusamy, Word Scientific, Singapore, 1991, 34--41.


\bibitem{Jan} W. Janowski, Some extremal problems for certain families of analytic functions, {\it Ann. Polon. Math.} \textbf{28} (1973), 297--326.

%


%

\bibitem{Ke-Me} F. R. Keogh and E. P. Merkes,
A coefficient inequality for certain classes of analytic functions, {\it Proc Amer. Math. Soc.} {\bf 20} (1969), 8--12.

%
%

\bibitem{K-R} Y. C. Kim and R. R{\o}nning,  Integral transforms of certain subclasses of analytic functions, {\it J. Math. Anal. Appl.} {\bf 258} (2001), 466--486.



\bibitem{Kom} Y. Komatu,
On starlike and convex mappings of a circle, {\it Kodai Math. Sem. Rep.} {\bf 13} (1961), 123--126.
%
%

\bibitem{Ma-Mi} W.~C. Ma and D. Minda, A unifed treatment of some special classes of univalent functions, in {\it Proceedings of the Conference on Complex Analysis } (Tianjin, 1992), 157--169, Conf. Proc. Lecture Notes Anal., I, Int. Press, Cambridge, MA.

\bibitem{MO86} G.~J. Martin and B.~G. Osgood,
The quasihyperbolic metric and associated  estimates on the hyperbolic metric, {\it J. Analyse Math.} {\bf 47} (1986),  37--53.


\bibitem{Mar} A. Marx,
Untersuchungen \"{u}ber schlichte Abbildungen, {\it Math. Ann.}, {\bf 107} (1932), 40--67.


\bibitem{MiMo-85} S. S. Miller and P. T. Mocanu,
Univalent solutions of Briot--Bouquet differential equations, {\it J. Diff. Equations} {\bf 56} (1985), 297--309.


\bibitem{M-M} S. S. Miller and P. T. Mocanu,
\textsl{Differential Subordinations. Theory and Applications}, Marcel Dekker Inc., New York, 2000.



\bibitem{MMR} S. S. Miller, P. T. Mocanu and M. O. Reade,
All alpha-convex functions are starlike, {\it Rev. Roumaine Math. Pure Appl.} {\bf17} (1972), 1395--1397.


\bibitem{M-69} P. T. Mocanu,
Une propri\'{e} de convexit\'{e} g\'{e}n\'{e}ralis\'{e}e dans la th\'{e}oric de la repr\'{e}sentation conforme, {\it Mathemattca (Cluj)} {\bf11} (1969), 127--133.

\bibitem{OPW} M. Obradovi\'{c}, S. Ponnusamy and KJ. Wirths, Geometric Studies on the Class $U(\lambda)$, {\it
 Bull. Malays. Math. Sci. Soc.} {\bf 39} (2016), 1259--1284, https://doi.org/10.1007/s40840-015-0263-5.


\bibitem{Por} T. Poreda,
On generalized differential equations in Banach space, \textit{Dissert. Math.} \textbf{310} (1991).


\bibitem{R-S-96}  S. Reich and D. Shoikhet,
Generation theory for semigroups of holomorphic mappings in Banach spaces, {\it Abstr. Appl. Anal.} {\bf 1} (1996), 1--44.

\bibitem{RS-SD-97a} S. Reich and D. Shoikhet,
Semigroups and generators on convex domains with the hyperbolic metric, {\it Atti Acad. Naz. Lincei Cl. Sci. Fiz. Mat. Natur. Rend. Lincei (9) Mat. Appl.} {\bf 8} (1997), 231--250.

\bibitem{RS-SD-97b} S. Reich and D. Shoikhet,
The Denjoy--Wolff theorem, \textit{Ann. Univ. Mariae Curie--Sklodowska} \textbf{51} (1997), 219--240.





\bibitem{R-S1}  S. Reich and D. Shoikhet,
\textsl{Nonlinear Semigroups, Fixed Points, and the Geometry of Domains in Banach Spaces}, World Scientific Publisher, Imperial College Press, London, 2005.


\bibitem{Rus75} St. Ruscheweyh,
New criteria for univalent functions, {\it Proc. Amer. Math. Soc.} \textbf{49} (1975), 109--115.


\bibitem{Rus77} St. Ruscheweyh,
Linear operators between classes of prestarlike functions, {\it Comment. Math. Helvetici} {\bf 52} (1977), 497--509.


%
\bibitem{DS-03} D. Shoikhet,
Representations of holomorphic generators and distortion theorems for spirallike functions with respect to a boundary
point, {\it Int. J. Pure Appl. Math.}, {\bf 5} (2003), 335--361.


\bibitem{SD} D. Shoikhet,
{\sl Semigroups in Geometrical Function Theory}, Kluwer Academic Publishers, Dordrecht, 2001.

\bibitem{SD16} D. Shoikhet, Rigidity and parametric embedding of semi-complete vector fields on the unit disk, {\it Milan J. Math.} {\bf 84} (2016), 159--202. https://doi.org/10.1007/s00032-016-0254-5.

\bibitem{Sim} B. Simon,
{\sl Orthogonal Polynomials on the Unit Circle, Part 1: Classical Theory}, Colloquium Publications, Amer. Math. Society, 2005.


\bibitem{Stro} E. Strohh\"{a}cker,
Beitr\"{a}ge z\"{u}r Theorie der schhlichten Functionen, {\it Math.Z.}, {\bf 37} (1933), 356--380.

\bibitem{Suff76}  T. J. Suffridge,
Starlike functions as limits of polynomials, {\it Advances in complex function theory}, in: Lecture Notes in Maths., {\bf 505}, Springer, Berlin--Heidelberg--New York, 1976, 164--202.


\bibitem{SW16sp} T.~Sugawa and L.-M. Wang,
Spirallikeness of shifted hypergeometric functions, {\it Ann. Acad. Sci. Fenn. Math.} {\bf 42} (2017), 963--977.

\bibitem{T-09} N. Tuneski,
Some simple sufficient conditions for starlikeness and convexity, {\it Appl. Math. Letters} {\bf 22} (2009), 693--697.


\end{thebibliography}
\end{document}